\documentclass[a4paper, reqno]{amsart}

\usepackage[T1]{fontenc}
\usepackage[latin1]{inputenc}
\usepackage{newlfont, verbatim, indentfirst, enumerate}

\usepackage{amsmath, amsthm, amssymb, amscd}
\usepackage{latexsym, amsfonts, amsbsy, mathrsfs}
\usepackage{array, pst-all, graphicx, rotating}
\usepackage{tikz} \usetikzlibrary{arrows}
  \usetikzlibrary{patterns} \usetikzlibrary{decorations.markings}
\usepackage[unicode]{hyperref}

\allowdisplaybreaks

\numberwithin{equation}{section}

\theoremstyle{plain}
\newtheorem{theorem}{Theorem}[section]
\newtheorem{theo}[theorem]{Theorem}
\newtheorem{prop}[theorem]{Proposition}
\newtheorem{lem}[theorem]{Lemma}   
\newtheorem{cor}[theorem]{Corollary}

\theoremstyle{definition}
\newtheorem{defin}[theorem]{Definition}
\newtheorem{rem}[theorem]{Remark}

\renewenvironment{itemize}{\begin{list}{$\bullet$}{\leftmargin=0.5cm}\parindent=0pt}{\end{list}}

\newcommand {\thickarrowpersonalized} [1]  
  {\draw[rounded corners,decoration={markings, mark=at position 1 with
  {\arrow[scale=2]{>}}}, postaction={decorate}] #1;}

\newcommand {\functor} [1] {\mathcal{#1}}  
\newcommand{\fiber}[4]{{#1}\,_{#2}\times_{#3}\, {#4}}  
\newcommand {\groupoid} [1] {\mathscr{#1}}  
\newcommand {\groupoidtot} [2] {\groupoid{#1}^{#2}_{\bullet}}  
\newcommand {\groupname} [2] {\groupoid{#1}^{#2}_1\rightrightarrows^{\hspace{-0.25 cm}^{s}}_{\hspace{-0.25 cm}_{t}}\groupoid{#1}^{#2}_0}  
\newcommand {\groupoidmap} [2] {({#1}^{#2}_0,{#1}^{#2}_1)}  
\newcommand {\groupoidmaptot} [2] {{#1}^{#2}_{\bullet}}  
\newcommand {\emphatic} [1] {\emph{#1}}

\newcommand {\thetaa} [3] {\theta_{#1,#2,#3}}  
\newcommand {\thetab} [3] {\theta_{#1,#2,#3}^{-1}}  
\newcommand {\Thetaa} [3] {\Theta_{#1,#2,#3}}  
\newcommand {\Thetab} [3] {\Theta_{#1,#2,#3}^{-1}}  

\def \CATA {\mathbf{\mathscr{A}}}  
\def \CATB {\mathbf{\mathscr{B}}}
\def \CATC {\mathbf{\mathscr{C}}}
\def \CATD {\mathbf{\mathscr{D}}}

\def \id {\operatorname{id}}  

\def \SETW {\mathbf{W}} 
\def \SETWsat {\mathbf{W}_{\mathbf{\operatorname{sat}}}}  
\def \SETWsatsat {\mathbf{W}_{\mathbf{\operatorname{sat}},\mathbf{\operatorname{sat}}}} 
\def \SETWinv {\mathbf{W}^{-1}}  
\def \SETWsatinv {\mathbf{W}_{\mathbf{\operatorname{sat}}}^{-1}}  
\def \SETWmin {\mathbf{W}_{\mathbf{\operatorname{min}}}}  
\def \SETWmininv {\mathbf{W}_{\mathbf{\operatorname{min}}}^{-1}}  
\def \SETWequiv {\mathbf{W}_{\mathbf{\operatorname{equiv}}}}  
\def \SETWequivinv {\mathbf{W}_{\mathbf{\operatorname{equiv}}}^{-1}}  

\def \SETWA {\mathbf{W}_{\mathbf{\mathscr{A}}}}  
\def \SETWAsat {\mathbf{W}_{\mathbf{\mathscr{A}},\mathbf{\operatorname{sat}}}}
\def \SETWAinv {\mathbf{W}^{-1}_{\mathbf{\mathscr{A}}}}

\def \SETWB {\mathbf{W}_{\mathbf{\mathscr{B}}}}  
\def \SETWBsat {\mathbf{W}_{\mathbf{\mathscr{B}},\mathbf{\operatorname{sat}}}}

\def \SETWBinv {\mathbf{W}^{-1}_{\mathbf{\mathscr{B}}}}
\def \SETWBsatinv {\mathbf{W}^{-1}_{\mathbf{\mathscr{B}},\mathbf{\operatorname{sat}}}}

\def \EGpd {(\mathbf{\mathcal{\acute{E}}\,\functor{G}pd})}  
\def \PEGpd {(\mathbf{\mathcal{P\acute{E}}\,\mathcal{G}pd})}  
\def \PEEGpd {(\mathbf{\mathcal{PE\acute{E}}\,\mathcal{G}pd})} 

\def \RedAtl {(\mathbf{\mathcal{R}ed\,\mathcal{A}tl})}  

\def \WEGpd {\mathbf{W}_{\mathbf{\mathcal{\acute{E}}\,\functor{G}pd}}}  
\def \WEGpdsat {\mathbf{W}_{\mathbf{\mathcal{\acute{E}}\,\functor{G}pd},\mathbf{\operatorname{sat}}}}  
\def \WEGpdinv {\mathbf{W}^{-1}_{\mathbf{\mathcal{\acute{E}}\,\functor{G}pd}}}  

\def \WPEGpd {\mathbf{W}_{\mathbf{\mathcal{P\acute{E}}\,\mathcal{G}pd}}}  
\def \WPEGpdinv {\mathbf{W}_{\mathbf{\mathcal{P\acute{E}}\,\mathcal{G}pd}}^{-1}}

\def \WPEEGpd {\mathbf{\mathbf{W}_{\mathcal{PE\acute{E}}\,\mathcal{G}pd}}} 
\def \WPEEGpdinv {\mathbf{W}_{\mathbf{\mathcal{PE\acute{E}}\,\mathcal{G}pd}}^{-1}}

\def \WRedAtl {\mathbf{W}_{\mathbf{\mathcal{R}ed\,\mathcal{A}tl}}}  

\hypersetup{unicode=true, pdftoolbar=true, pdfmenubar=true, pdffitwindow=false,
  pdfstartview={FitH}, pdftitle={Some insights on bicategories of fractions - II},
  pdfauthor={Matteo Tommasini}, pdfkeywords={bicategories of fractions, bicalculus of fractions,
  pseudofunctors, \'etale groupoids, Morita equivalences},
  pdfnewwindow=true,colorlinks=false, linkcolor=red, citecolor=red, filecolor=magenta, urlcolor=cyan}

\numberwithin{equation}{section}

\begin{document}

\title[Some insights on bicategories of fractions - II]
{Some insights on bicategories of fractions - II \\ - \\ \mdseries{\small{Right saturations and induced pseudofunctors between bicategories of fractions}}}

\author{Matteo Tommasini}

\address{\flushright Mathematics Research Unit\newline University of Luxembourg\newline
6, rue Richard Coudenhove-Kalergi\newline L-1359 Luxembourg\newline\newline
website: \href{http://matteotommasini.altervista.org/}
{\nolinkurl{http://matteotommasini.altervista.org/}}\newline\newline
email: \href{mailto:matteo.tommasini2@gmail.com}{\nolinkurl{matteo.tommasini2@gmail.com}}}

\date{\today}
\subjclass[2010]{18A05, 18A30, 22A22}
\keywords{bicategories of fractions, bicalculus of fractions, pseudofunctors, \'etale
groupoids, Morita equivalences}

\thanks{I would like to thank Dorette Pronk for several interesting discussions about her work on
bicategories of fraction and for some useful suggestions about this series of papers. This research 
was performed at the Mathematics Research Unit of the University of Luxembourg, thanks to the grant
4773242 by Fonds National de la Recherche Luxembourg.}

\begin{abstract}
We fix any bicategory $\CATA$ together with a class of morphisms $\SETWA$, such that there is a
bicategory of fractions $\CATA\left[\SETWAinv\right]$ (as described by D.~Pronk). Given another such
pair $(\CATB,\SETWB)$ and any pseudofunctor $\functor{F}:
\CATA\rightarrow\CATB$, we find necessary and sufficient conditions in order to have an induced
pseudofunctor $\functor{G}:\CATA\left[\SETWAinv\right]\rightarrow\CATB\left[\SETWBinv\right]$.
Moreover, we give a simple description of $\functor{G}$ in the case when the class $\SETWB$ is
``right saturated''.
\end{abstract}

\maketitle

\begingroup{\hypersetup{linkbordercolor=white}
\tableofcontents}\endgroup

\section*{Introduction}
In 1996 Dorette Pronk introduced the notion of (\emph{right}) \emph{bicalculus of fractions}
(see~\cite{Pr}), generalizing the concept of (right) calculus of fractions (described in 1967 by
Pierre Gabriel and Michel Zisman, see~\cite{GZ}) from the framework of categories to that of
bicategories. Pronk proved that given a bicategory $\CATC$ together with a class of morphisms $\SETW$
(satisfying a set of technical conditions called
(BF)),
there are a bicategory
$\CATC\left[\SETWinv\right]$
(called (\emph{right}) \emph{bicategory of fractions}) and a pseudofunctor $\functor{U}_{\SETW}:\CATC
\rightarrow\CATC\left[\SETWinv\right]$. Such a pseudofunctor sends each element of $\SETW$ to an
internal equivalence and is universal with respect to such property (see~\cite[Theorem~21]{Pr}).
The structure of $\CATC\left[\SETWinv\right]$ depends on a set of choices \hyperref[C]{C}$(\SETW)$
involving axioms (BF) (see \S~\ref{sec-03}); by the universal property of $\functor{U}_{\SETW}$,
different sets of choices give rise to equivalent bicategories.\\

Now let us suppose that we have fixed any $2$ pairs $(\CATA,\SETWA)$ and $(\CATB,\SETWB)$,
both admitting a right bicalculus of fractions, and any pseudofunctor $\functor{F}:\CATA
\rightarrow\CATB$. Then the following $3$ questions arise naturally:

\begin{enumerate}[(a)]
 \item what are the necessary and sufficient conditions such that there are a pseudofunctor
  $\functor{G}$ and a pseudonatural equivalence $\kappa$ as in the following diagram?

  \begin{equation}\label{eq-56}
  \begin{tikzpicture}[xscale=2.2,yscale=-0.9]
    \node (A0_0) at (0, 0) {$\CATA$};
    \node (A0_2) at (2, 0) {$\CATB$};
    \node (A2_0) at (0, 2) {$\CATA\Big[\SETWAinv\Big]$};
    \node (A2_2) at (2, 2) {$\CATB\Big[\SETWBinv\Big]$};
    
    \node (A1_1) [rotate=225] at (0.9, 1) {$\Longrightarrow$};
    \node (B1_1) at (1.2, 1) {$\kappa$};
    
    \path (A0_0) edge [->]node [auto] {$\scriptstyle{\functor{F}}$} (A0_2);
    \path (A0_0) edge [->]node [auto,swap] {$\scriptstyle{\functor{U}_{\SETWA}}$} (A2_0);
    \path (A0_2) edge [->]node [auto] {$\scriptstyle{\functor{U}_{\SETWB}}$} (A2_2);
    \path (A2_0) edge [->, dashed]node [auto,swap] {$\scriptstyle{\functor{G}}$} (A2_2);
  \end{tikzpicture}
  \end{equation}

  \item If a pair $(\functor{G},\kappa)$ as above exists, can we express
   $\functor{G}$ in a simple form, at least in some cases?
  \item Again if $(\functor{G},\kappa)$ as above exists, what are the necessary and sufficient
   conditions such that $\functor{G}$ is an equivalence of bicategories?
\end{enumerate}

We are going to give an answer to (a) and (b) in this paper, while an answer to (c) will be given
in the next paper~\cite{T5}. In order to prove the results of this paper, a key notion will be
that of (\emph{right}) \emph{saturation}: given any pair $(\CATC,\SETW)$ as above, we define the
(right) saturation $\SETWsat$ of $\SETW$ as the class of all morphisms $f:B
\rightarrow A$ in $\CATC$, such that there are a pair of objects $C,D$ and a pair of morphisms $g:C
\rightarrow B$, $h:D\rightarrow C$, such that both $f\circ g$ and $g\circ h$ belong to $\SETW$.
If $(\CATC,\SETW)$ satisfies conditions (BF),
then $\SETW\subseteq\SETWsat$ and $\SETWsat=\SETWsatsat$, thus explaining the name ``saturation'' for
this class. Moreover, we have the following key result:

\begin{prop}\label{prop-04}
\emphatic{(Lemma~\ref{lem-05} and Proposition~\ref{prop-01})}
Let us fix any pair $(\CATC,\SETW)$ satisfying conditions
\emphatic{(BF)}.
Then also the pair $(\CATC,\SETWsat)$ satisfies the same conditions, so there are a bicategory of
fractions $\CATC\left[\SETWsatinv\right]$ and a pseudofunctor

\begin{equation}\label{eq-46}
\functor{U}_{\SETWsat}:\CATC\longrightarrow\CATC\left[\SETWsatinv\right]
\end{equation}
with the universal property. Moreover, there is an equivalence
of bicategories $\functor{H}:\,\CATC\left[\SETWsatinv\right]\rightarrow\CATC\left[\SETWinv\right]$
and a pseudonatural equivalence of pseudofunctors $\tau:\functor{U}_{\SETW}\Rightarrow\functor{H}
\circ\functor{U}_{\SETWsat}$.
\end{prop}

Then an answer to questions (a) is given by the equivalence of (i) and (iii) below.

\begin{theo}\label{theo-04}
Let us fix any $2$ pairs $(\CATA,\SETWA)$ and $(\CATB,\SETWB)$, both satisfying conditions
\emphatic{(BF)}, and any pseudofunctor $\functor{F}:\CATA\rightarrow
\CATB$. Then the following facts are equivalent:

\begin{enumerate}[\emphatic{(}i\emphatic{)}]
 \item $\functor{F}_1(\SETWA)\subseteq\SETWBsat$;
 \item $\functor{F}_1(\SETWAsat)\subseteq\SETWBsat$;
 \item there are a pseudofunctor $\functor{G}$ and a pseudonatural equivalence of pseudofunctors
  $\kappa$ as in \eqref{eq-56};
 \item there is a pair $(\functor{G},\kappa)$ as in \emphatic{(}iii\emphatic{)}, such that the
  pseudofunctor $\mu_{\kappa}:\CATA\rightarrow\operatorname{Cyl}\left(\CATB\left[\SETWBinv\right]
  \right)$ associated to $\kappa$ sends each morphism of $\SETWA$ to an internal
  equivalence \emphatic{(}here $\operatorname{Cyl}(\CATC)$ is the bicategory of cylinders associated
  to any given bicategory $\CATC$, see~\cite[pag.~60]{B}\emphatic{)}.
\end{enumerate}
\end{theo}

Then we are able to give a complete answer to question (b) in the case when $\functor{F}_1(\SETWA)
\subseteq\SETWB$: this condition in general is slightly more restrictive than condition (i) above.
In the case when $\functor{F}_1(\SETWA)$ is only contained in $\SETWBsat$ and not in $\SETWB$, we
can still give a complete answer to question (b), provided that we allow as target the
bicategory $\CATB\left[\SETWBsatinv\right]$ instead of $\CATB\left[\SETWBinv\right]$
(by virtue of Proposition~\ref{prop-04}, this does not make any significant difference). To be
more precise, we have:

\begin{theo}\label{theo-03}
Let us fix any $2$ pairs $(\CATA,\SETWA)$ and $(\CATB,\SETWB)$, both satisfying conditions
\emphatic{(BF)}, and any pseudofunctor $\functor{F}:\CATA\rightarrow
\CATB$.

\emphatic{(}A\emphatic{)} If $\functor{F}_1(\SETWA)\subseteq\SETWBsat$, then there are a pseudofunctor

\[\widetilde{\functor{G}}:\,\CATA\Big[\SETWAinv\Big]\longrightarrow\CATB\Big[\SETWBsatinv\Big]\]
and a pseudonatural equivalence $\widetilde{\kappa}:\functor{U}_{\SETWBsat}\circ\functor{F}
\Rightarrow\widetilde{\functor{G}}\circ\functor{U}_{\SETWA}$, such that:

\begin{enumerate}[\emphatic{(}I\emphatic{)}]
 \item the pseudofunctor $\mu_{\widetilde{\kappa}}:\CATA\rightarrow\operatorname{Cyl}
  \left(\CATB\left[\SETWBsatinv\right]\right)$ associated to $\widetilde{\kappa}$ sends each morphism
  of $\SETWA$ to an internal equivalence;
 
 \item for each object $A_{\CATA}$, we have $\widetilde{\functor{G}}_0(A_{\CATA})=
  \functor{F}_0(A_{\CATA})$;

 \item for each morphism $(A'_{\CATA},\operatorname{w}_{\CATA},f_{\CATA}):A_{\CATA}\rightarrow
  B_{\CATA}$ in $\CATA\left[\SETWAinv\right]$, we have 
  
  \[\widetilde{\functor{G}}_1\Big(A'_{\CATA},\operatorname{w}_{\CATA},f_{\CATA}\Big)=\Big(
  \functor{F}_0(A'_{\CATA}),\functor{F}_1(\operatorname{w}_{\CATA}),\functor{F}_1(f_{\CATA})
  \Big);\]
  
 \item for each $2$-morphism

  \begin{equation}\label{eq-27}
  \Big[A^3_{\CATA},\operatorname{v}^1_{\CATA},\operatorname{v}^2_{\CATA},\alpha_{\CATA},
  \beta_{\CATA}\Big]:\Big(A^1_{\CATA},\operatorname{w}^1_{\CATA},f^1_{\CATA}\Big)\Longrightarrow
  \Big(A^2_{\CATA},\operatorname{w}^2_{\CATA},f^2_{\CATA}\Big)
  \end{equation}
  in $\CATA\left[\SETWAinv\right]$, we have
  
  \begin{gather}
  \label{eq-11} \widetilde{\functor{G}}_2\Big(\Big[A^3_{\CATA},\operatorname{v}^1_{\CATA},
  \operatorname{v}^2_{\CATA},\alpha_{\CATA},\beta_{\CATA}\Big]\Big)=
  \Big[\functor{F}_0(A^3_{\CATA}),\functor{F}_1(\operatorname{v}^1_{\CATA}),\functor{F}_1
  (\operatorname{v}^2_{\CATA}),\\
  \nonumber \psi_{\operatorname{w}^2_{\CATA},
  \operatorname{v}^2_{\CATA}}^{\functor{F}}\odot\functor{F}_2(\alpha_{\CATA})\odot
  \Big(\psi_{\operatorname{w}^1_{\CATA},\operatorname{v}^1_{\CATA}}^{\functor{F}}\Big)^{-1},
  \psi_{f^2_{\CATA},\operatorname{v}^2_{\CATA}}^{\functor{F}}\odot\functor{F}_2(\beta_{\CATA})
  \odot\Big(\psi_{f^1_{\CATA},\operatorname{v}^1_{\CATA}}^{\functor{F}}\Big)^{-1}\Big]
  \end{gather}
  \emphatic{(}where the $2$-morphisms $\psi_{\bullet}^{\functor{F}}$ are the associators of
  $\functor{F}$\emphatic{)}.
\end{enumerate}

\emphatic{(}B\emphatic{)} Furthermore, if $\functor{F}_1(\SETWA)\subseteq\SETWB$,
then there are a pseudofunctor

\[\widetilde{\functor{G}}:\,\CATA\Big[\SETWAinv\Big]\longrightarrow\CATB\Big[\SETWBinv\Big]\]
and a pseudonatural equivalence $\widetilde{\kappa}:\functor{U}_{\SETWB}\circ\functor{F}
\Rightarrow\widetilde{\functor{G}}\circ\functor{U}_{\SETWA}$, such that:

\begin{itemize}
 \item the pseudofunctor $\mu_{\widetilde{\kappa}}:\CATA\rightarrow\operatorname{Cyl}
  \left(\CATB\left[\SETWBinv\right]\right)$ associated to $\widetilde{\kappa}$ sends each morphism
  of $\SETWA$ to an internal equivalence;

\item conditions \emphatic{(}II\emphatic{)}, \emphatic{(}III\emphatic{)} and
  \emphatic{(}IV\emphatic{)} hold.
\end{itemize}
\end{theo}

If $\functor{F}_1(\SETWA)$ is only contained in $\SETWBsat$ but not in $\SETWB$ and if we still
want to describe a pair $(\functor{G},\kappa)$ with $\functor{G}$ with target in $\CATB[
\SETWBinv]$, then $\functor{G}$ can be induced by composing $\widetilde{\functor{G}}$
described in (B) above and the equivalence of bicategories $\functor{H}:
\CATB[\SETWBsatinv]\rightarrow\CATB[\SETWBinv]$ induced by
Proposition~\ref{prop-04} (see Remark~\ref{rem-06}), but in general the explicit description
of $\functor{G}$ is much more complicated than the one of $\widetilde{\functor{G}}$, since
$\functor{H}$ in general is very complicated to describe explicitly.\\

In addition, we have:

\begin{cor}\label{cor-03}
Let us fix any $2$ pairs $(\CATA,\SETWA)$ and $(\CATB,\SETWB)$, both satisfying conditions
\emphatic{(BF)}, and any pseudofunctor $\functor{F}:\CATA\rightarrow\CATB$. Moreover, let us fix
any pair $(\functor{G},\kappa)$ as in
\emphatic{Theorem~\ref{theo-04}(iv)}. Then the following facts are equivalent:

\begin{enumerate}[\emphatic{(}1\emphatic{)}]
 \item $\functor{G}:\CATA\left[\SETWAinv\right]\rightarrow\CATB\left[\SETWBinv\right]$ is an
   equivalence of bicategories;
 \item the pseudofunctor $\widetilde{\functor{G}}:\CATA\left[\SETWAinv\right]\rightarrow\CATB
  \left[\SETWBsatinv\right]$ described in \emphatic{(}A\emphatic{)} above is an equivalence of
  bicategories.
\end{enumerate}
\end{cor}

In the next paper of this series (\cite{T5}) we will find a set of
conditions on $(\CATA,\SETWA,\CATB,$ $\SETWB,\functor{F})$ that are equivalent to (2) above. Combining
with the previous Corollary, this will allow us to give a complete answer to question (c).\\

As an application of the constructions about saturations used in the results above, in the
last part of this paper we will focus on the class of Morita equivalences
in the bicategory of \'etale differentiable (Lie) groupoids, and we will prove that such a class is
right saturated.\\

In all this paper we are going to use the axiom of choice, that we will assume from now on without
further mention. The reason for this is twofold. First of all, the axiom of choice is used heavily
in~\cite{Pr} in order to construct bicategories of fractions. In~\cite[Corollary~0.6]{T3} we proved
that under some restrictive hypothesis the axiom of choice is not necessary, but in the general case
we need it in order to consider any of the bicategories of fractions mentioned above.
Secondly, even in the cases when the axiom of choice is not necessary for the construction of
the bicategories
$\CATA\left[\SETWAinv\right]$ and $\CATB\left[\SETWBinv\right]$, we will have to use
often the universal property of such bicategories of fractions, as stated in~\cite[Theorem~21]{Pr},
and the proof of this property requires the axiom of choice.

\section{Notations and basic facts}
\subsection{Generalities on bicategories}\label{sec-02}
Given any bicategory $\CATC$, we denote its objects by $A,B,\cdots$, its morphisms by $f,g,\cdots$
and its $2$-morphisms by $\alpha,\beta,\cdots$; we will use $A_{\CATC},
f_{\CATC},\alpha_{\CATC},\cdots$ if we have to recall that they belong to $\CATC$ when we are using
more than one bicategory in the computations. Given any triple of morphisms $f:A\rightarrow B$,
$g:B\rightarrow C$, $h:C\rightarrow D$ in $\CATC$, we denote by $\thetaa{h}{g}{f}$ the associator
$h\circ(g\circ f)\Rightarrow(h\circ g)\circ f$ that is part of the structure of $\CATC$; we denote by
$\pi_f:f\circ\id_A\Rightarrow f$ and $\upsilon_f:\id_B\circ f\Rightarrow f$ the right and left unitors
for $\CATC$ relative to any morphism $f$ as above.
We denote any pseudofunctor from $\CATC$ to another bicategory $\CATD$ by $\functor{F}=
(\functor{F}_0,\functor{F}_1,\functor{F}_2,$ $\psi_{\bullet}^{\functor{F}},
\sigma_{\bullet}^{\functor{F}}):\CATC\rightarrow\CATD$. Here for each pair of morphisms $f,g$
as above, $\psi^{\functor{F}}_{g,f}$ is the associator from $\functor{F}_1(g\circ f)$ to
$\functor{F}_1(g)\circ\functor{F}_1(f)$ and for each object $A$, $\sigma^{\functor{F}}_A$ is the
unitor from $\functor{F}_1(\id_A)$ to $\id_{\functor{F}_0(A)}$.\\

We recall that a morphism $e:A\rightarrow B$ in a bicategory $\CATC$ is called an \emph{internal
equivalence} (or, simply, an \emph{equivalence}) of $\CATC$
if and only if there exists a triple $(\overline{e},\delta,\xi)$, where $\overline{e}$ is a morphism
from $B$ to $A$ and $\delta:\id_A\Rightarrow\overline{e}\circ e$ and $\xi:e\circ\overline{e}
\Rightarrow\id_B$ are invertible $2$-morphisms in $\CATC$ (in the literature sometimes the name
``(internal) equivalence'' is used for denoting the whole quadruple $(e,\overline{e},\delta,\xi)$
instead of the morphism $e$ alone). In particular, $\overline{e}$ is an internal equivalence (it
suffices to consider the triple $(e,\xi^{-1},\delta^{-1})$) and it is usually called \emph{a
quasi-inverse} (or \emph{pseudo-inverse}) for $e$ (in general, the quasi-inverse of an internal
equivalence is not unique). An \emph{adjoint equivalence} is a quadruple $(e,\overline{e},\delta,
\xi)$ as above, such that

\begin{equation}\label{eq-10}
\upsilon_e\odot\Big(\xi\ast i_e\Big)\odot\thetaa{e}{\overline{e}}{e}\odot\Big(i_e\ast\delta\Big)
\odot\pi^{-1}_e=i_e
\end{equation}
and

\begin{equation}\label{eq-55}
\pi_{\overline{e}}\odot\Big(i_{\overline{e}}\ast\xi\Big)\odot\thetab{\overline{e}}{e}{\overline{e}}
\odot\Big(\delta\ast i_{\overline{e}}\Big)\odot\upsilon_{\overline{e}}^{-1}=i_{\overline{e}}
\end{equation}
(this more restrictive definition is actually the original definition of internal equivalence used for example
in~\cite[pag.~83]{Mac}). By~\cite[Proposition~1.5.7]{L} a morphism $e$ is (the first component of) an
internal equivalence if and only if it is the first component of a (possibly different)
adjoint equivalence.\\

In the following pages, we will use often the following easy lemmas (a detailed proof of the second
and third lemma is given in the Appendix).

\begin{lem}\label{lem-04}
Let us suppose that $e:A\rightarrow B$ is an internal equivalence in a bicategory $\CATC$ and let
$\gamma:e\Rightarrow\widetilde{e}$ be any invertible $2$-morphism in $\CATC$. Then also
$\widetilde{e}$ is an internal equivalence.
\end{lem}

\begin{lem}\label{lem-06}
Let us fix any bicategory $\CATC$; the class $\SETWequiv$ of all internal equivalences of $\CATC$
satisfies the ``$\,2$-out-of-$3$'' property, i.e.\ given any pair of morphisms $f:B\rightarrow A$ and
$g:C\rightarrow B$, if any $2$ of the $3$ morphisms $f,g$ and $f\circ g$ are internal equivalences,
so is the third one.
\end{lem}

\begin{lem}\label{lem-02}
Let us fix any bicategory $\CATC$ and any triple of morphisms $f:B\rightarrow A$, $g:C\rightarrow B$
and $h:D\rightarrow C$, such that both $f\circ g$ and $g\circ h$ are internal equivalences. Then
the morphisms $f,g$ and $h$ are all internal equivalences.
\end{lem}

We recall from~\cite[(1.33)]{St} that given any pair of bicategories $\CATC$ and $\CATD$, a
pseudofunctor $\functor{F}:\CATC\rightarrow\CATD$ is a \emph{weak equivalence of bicategories} (also
known as \emph{weak biequivalence}) if and only if the following $2$ conditions hold:

\begin{enumerate}[({X}1)]\label{X}
 \item\label{X1} for each object $A_{\CATD}$ there are an object $A_{\CATC}$ and an internal
  equivalence from $\functor{F}_0(A_{\CATC})$ to $A_{\CATD}$ in $\CATD$;
 \item\label{X2} for each pair of objects $A_{\CATC},B_{\CATC}$, the functor $\functor{F}(A_{\CATC},
  B_{\CATC})$ is an equivalence of categories from $\CATC(A_{\CATC},B_{\CATC})$ to $\CATD
  (\functor{F}_0(A_{\CATC}),\functor{F}_0(B_{\CATC}))$.
\end{enumerate}

\emph{Since in all this paper we assume the axiom of choice, then each weak equivalence of
bicategories is a \emphatic{(}strong\emphatic{)} equivalence of bicategories} (also known as
biequivalence, see~\cite[\S~1]{PW}), i.e.\ it admits a quasi-inverse. Conversely, each strong
equivalence of bicategories is a weak equivalence. So in the present setup we will simply write
``equivalence of bicategories'' meaning weak, equivalently strong, equivalence.
Also the proof of the following lemma can be found in the Appendix.

\begin{lem}\label{lem-07}
Let us fix any pair of bicategories $\CATC,\CATD$, any pair of pseudofunctors $\functor{F},
\functor{G}:\CATC\rightarrow\CATD$ and any pseudonatural equivalence $\phi:\functor{F}\Rightarrow
\functor{G}$. If $\functor{F}$ is an equivalence of bicategories, then so is $\functor{G}$.
\end{lem}

In the following pages we will often use the following notations: given any pair of bicategories
$\CATC,\CATD$ and any class of morphisms $\SETW$ in $\CATC$,

\begin{enumerate}[(a)]
 \item $\operatorname{Hom}(\CATC,\CATD)$ is the bicategory of pseudofunctors $\CATC\rightarrow
  \CATD$, Lax natural transformations of them and modifications of Lax natural transformations;
 \item $\operatorname{Hom}'(\CATC,\CATD)$ is the bicategory of pseudofunctors $\CATC\rightarrow
  \CATD$, pseudonatural transformations of them and pseudonatural modifications of pseudonatural
  transformations (a bi-subcategory of (a));
 \item $\operatorname{Hom}_{\SETW}(\CATC,\CATD)$ is the bi-subcategory of (a), such that all the
  pseudofunctors, the Lax natural transformations and the modifications send each element of
  $\SETW$ to an internal equivalence; here a Lax natural transformation is considered as a
  pseudofunctor from $\CATC$ to the bicategory of cylinders $\operatorname{Cyl}(\CATD)$ of $\CATD$
  and a modification is considered as a pseudofunctor from $\CATC$ to $\operatorname{Cyl}
  (\operatorname{Cyl}(\CATD))$ (see~\cite[pag.~60]{B});
 \item $\operatorname{Hom}'_{\SETW}(\CATC,\CATD)$ is the bi-subcategory of (c), obtained by
  restricting morphisms to pseudonatural transformations and $2$-morphisms to pseudonatural
  modifications.
\end{enumerate}

Then it is not difficult to prove that:

\begin{lem}\label{lem-01}
Given any pair of bicategories $\CATC$ and $\CATD$, and any pair of pseudofunctors $\functor{F},
\functor{G}:\CATC\rightarrow\CATD$, there is an internal equivalence from $\functor{F}$ to
$\functor{G}$ in \emphatic{(a)} if and only if there is an internal equivalence in \emphatic{(b)}
between the same $2$ objects \emphatic{(}i.e.\ a pseudonatural equivalence of
pseudofunctors\emphatic{)}.
Moreover, given any class $\SETW$ of morphisms in $\CATC$ and any pair of objects $\functor{F},
\functor{G}$ in \emphatic{(c)}, there is an internal equivalence between such objects in \emphatic{(c)} if and
only if there is an internal equivalence in \emphatic{(d)} between the same $2$ objects 
\emphatic{(}i.e.\ a pseudonatural equivalence of
pseudofunctors\emphatic{)}.
\end{lem}

\subsection{Bicategories of fractions}\label{sec-03}
We refer to the original reference~\cite{Pr} or to our previous paper~\cite{T3} for the
list of axioms (BF1) -- (BF5)
needed for a bicalculus of fractions. We recall in particular the following fundamental result.

\begin{theo}\label{theo-01}
\cite[Theorem~21]{Pr} Given any pair $(\CATC,\SETW)$ satisfying conditions
\emphatic{(BF)}, there are a bicategory
$\CATC\left[\SETWinv\right]$ \emphatic{(}called \emph{(right) bicategory of fractions}\emphatic{)} and a
pseudofunctor $\functor{U}_{\SETW}:\CATC\rightarrow\CATC\left[\SETWinv\right]$ that sends each
element of $\SETW$ to an internal equivalence and that is universal with respect to such property.
Here ``universal'' means that for each bicategory $\CATD$, composition with $\functor{U}_{\SETW}$
gives an equivalence of bicategories

\begin{equation}\label{eq-94}
-\circ\,\functor{U}_{\SETW}:\,\operatorname{Hom}\Big(\CATC\left[\SETWinv\right],\CATD\Big)
\longrightarrow\operatorname{Hom}_{\SETW}\Big(\CATC,\CATD\Big).
\end{equation}
\end{theo}

In particular, \emph{the bicategory} $\CATC\left[\SETWinv\right]$ \emph{is unique up to
equivalences of bicategories}.\\

\begin{rem}
The axiom of choice is used heavily in the construction of bicategories of fractions
(see~\cite[\S~2.2 and 2.3]{Pr}). In some special cases, one can bypass this problem, as we explained
in~\cite[Corollary~0.6]{T3}. However, also in such special cases, in general the proof of
Theorem~\ref{theo-01} relies on the axiom of choice (for the construction of the pseudofunctor
$\widetilde{\functor{F}}$ in~\cite[Theorem~21]{Pr}). The present paper is heavily based on that
result, so this requires implicitly to use the axiom of choice often. For example, even in
order to prove basic results (such as the one in Lemma~\ref{lem-08}(iii) below), one has to use the
axiom of choice. Indeed, in the mentioned Lemma we will implicitly follow
the proof of~\cite[Theorem~21]{Pr}, so we will have to choose
a quasi-inverse for any internal equivalence (of the bicategory $\CATC$ where we are working), and
in general this requires the axiom of choice. One of the few cases when we will not need the axiom
of choice is the proof of Proposition~\ref{prop-06}, (see Remark~\ref{rem-03}).
\end{rem}

\begin{rem}\label{rem-01}
In the notations of~\cite{Pr}, the pseudofunctor $\functor{U}_{\SETW}$ is called a \emph{bifunctor},
but this notation is no more in use. In~\cite{Pr} Theorem~\ref{theo-01} is stated with condition
(BF1) (namely: ``all $1$-identities of $\CATC$ belong to $\SETW$'', see~\cite{T3}) replaced by the
slightly stronger hypothesis 

\begin{enumerate}[({BF}1)$'$:]
\item\label{BF1prime} ``all the internal equivalences of $\CATC$ belong to $\SETW$''.
\end{enumerate}

By looking carefully at the proofs in~\cite{Pr}, it is easy to see that the only part of axiom
(\hyperref[BF1prime]{BF1})$'$ that is really used in all the computations is
(BF1), so we are allowed to state~\cite[Theorem~21]{Pr} under such less restrictive hypothesis.
Note that by virtue of Lemma~\ref{lem-08}(ii) and Proposition~\ref{prop-01} below,
choosing condition
(\hyperref[BF1prime]{BF1})$'$ instead of (BF1) gives equivalent 
bicategories of fractions, so this does not make any significant difference.
\end{rem}

For the explicit construction of bicategories of fractions we refer all the time
either to~\cite{T3} or to the original construction in~\cite{Pr}.
We recall that according to~\cite{Pr} the construction of compositions in $\CATC\left[\SETWinv\right]$
depends on $2$ sets of choices related to axioms (BF3) and (BF4)
respectively. In~\cite[Theorem~0.5]{T3} we proved that actually all the choices
related to axiom (BF4) are not necessary, so in order to have a structure of bicategory on
$\CATC\left[\SETWinv\right]$ it is sufficient to fix a set of choices as follows:

\begin{enumerate}[A$(\SETW)$:]
 \setcounter{enumi}{2}
 \item\label{C} for every set of data in $\CATC$ as follows

  \begin{equation}\label{eq-30}
  \begin{tikzpicture}[xscale=1.5,yscale=-1.2]
    \node (A0_0) at (0, 0) {$A'$};
    \node (A0_1) at (1, 0) {$B$};
    \node (A0_2) at (2, 0) {$B'$};
    \path (A0_0) edge [->]node [auto] {$\scriptstyle{f}$} (A0_1);
    \path (A0_2) edge [->]node [auto,swap] {$\scriptstyle{\operatorname{v}}$} (A0_1);
  \end{tikzpicture}
  \end{equation}
  with $\operatorname{v}$ in $\SETW$, using axiom (BF3)
  we \emph{choose} an object
  $A''$, a pair of morphisms $\operatorname{v}'$ in $\SETW$ and $f'$ and an invertible $2$-morphism
  $\rho$ in $\CATC$, as follows:

  \begin{equation}\label{eq-32}
  \begin{tikzpicture}[xscale=1.5,yscale=-0.8]
    \node (A0_1) at (1, 0) {$A''$};
    \node (A1_0) at (0, 2) {$A'$};
    \node (A1_2) at (2, 2) {$B'$.};
    \node (A2_1) at (1, 2) {$B$};

    \node (A1_1) at (1, 1) {$\rho$};
    \node (B1_1) at (1, 1.4) {$\Rightarrow$};
    
    \path (A1_2) edge [->]node [auto] {$\scriptstyle{\operatorname{v}}$} (A2_1);
    \path (A0_1) edge [->]node [auto] {$\scriptstyle{f'}$} (A1_2);
    \path (A1_0) edge [->]node [auto,swap] {$\scriptstyle{f}$} (A2_1);
    \path (A0_1) edge [->]node [auto,swap] {$\scriptstyle{\operatorname{v}'}$} (A1_0);
  \end{tikzpicture}
  \end{equation}
\end{enumerate}

According to~\cite[\S~2.1]{Pr}, such choices must satisfy the following $2$ conditions:

\begin{enumerate}[({C}1)]
 \item\label{C1} whenever \eqref{eq-30} is such that $B=A'$ and $f=\id_B$, then we choose
   $A'':=B'$, $f':=\id_{B'}$, $\operatorname{v}':=\operatorname{v}$ and $\rho:=
   \pi_{\operatorname{v}}^{-1}\odot\upsilon_{\operatorname{v}}$;
 \item\label{C2} whenever \eqref{eq-30} is such that $B=B'$ and $\operatorname{v}=\id_B$, then we
   choose $A'':=A'$, $f':=f$, $\operatorname{v}':=\id_{A'}$ and $\rho:=\upsilon_f^{-1}\odot\pi_f$.
\end{enumerate}

In the proof of Theorem~\ref{theo-03} below we will have to consider a set of choices
\hyperref[C]{C}$(\SETW)$ satisfying also the following additional condition:

\begin{enumerate}[({C}1)]
 \setcounter{enumi}{2}
 \item\label{C3} whenever \eqref{eq-30} is such that $A'=B'$ and $f=\operatorname{v}$ (with
   $\operatorname{v}$ in $\SETW$), then we choose $A'':=A'$, $f':=\id_{A'}$, $\operatorname{v}':=
   \id_{A'}$ and $\rho:=i_{f\circ\id_{A'}}$.
\end{enumerate}

Condition (\hyperref[C3]{C3}) is not strictly necessary in order to do a right bicalculus of
fractions, but it simplifies lots of the computations below. We have only to check that
(\hyperref[C3]{C3}) is compatible with conditions (\hyperref[C1]{C1}) and (\hyperref[C2]{C2})
required by~\cite{Pr}, but this is obvious. In other terms, given any class $\SETW$ satisfying
condition (BF3), there is always a set of choices \hyperref[C]{C}$(\SETW)$,
satisfying conditions (\hyperref[C1]{C1}), (\hyperref[C2]{C2}) and (\hyperref[C3]{C3}).\\

We refer to~\cite{T3} for a description of the associators
$\Theta_{\bullet}^{\CATC,\SETW}$, the vertical and the horizontal
compositions of $2$-morphisms in $\CATC\left[\SETWinv\right]$; such descriptions
simplify the original constructions given
in~\cite{Pr} and they will be used often in the next pages.
Moreover, we have the following result, whose proof is already implicit in~\cite[Theorem~21]{Pr}.

\begin{theo}\label{theo-02}
\cite{Pr} Let us fix any pair $(\CATA,\SETWA)$ satisfying conditions \emphatic{(BF)},
any bicategory $\CATB$ and any pseudofunctor $\functor{F}:\CATA
\rightarrow\CATB$. Then the following facts are equivalent:

\begin{enumerate}[\emphatic{(}i\emphatic{)}]
 \item $\functor{F}$ sends each morphism of $\SETWA$ to an internal equivalence of $\CATB$;
 \item there are a pseudofunctor $\overline{\functor{G}}:\CATA\left[\SETWAinv\right]\rightarrow
  \CATB$ and a pseudonatural equivalence of pseudofunctors $\overline{\kappa}:\functor{F}\Rightarrow
  \overline{\functor{G}}\circ\functor{U}_{\SETWA}$;
 \item there is a pair $(\overline{\functor{G}},\overline{\kappa})$ as in \emphatic{(}ii\emphatic{)},
  such that the pseudofunctor $\mu_{\overline{\kappa}}:\CATA\rightarrow
  \operatorname{Cyl}(\CATB)$ associated to $\overline{\kappa}$ sends each morphism of $\SETWA$ to an
  internal equivalence
  \emphatic{(}i.e.\ $\overline{\kappa}$ is an internal equivalence in $\operatorname{Hom}'_{\SETWA}
  (\CATA,\CATB)$\emphatic{)}.
\end{enumerate}
\end{theo}

\begin{proof}
Using (\hyperref[X1]{X1}) on the equivalence \eqref{eq-94} for $(\CATC,\SETW,\CATD):=(\CATA,\SETWA,
\CATB)$ together with Lemma~\ref{lem-01} (or looking directly at the first part
of~\cite[Proof of Theorem~21]{Pr}), we have that (i) and (iii) are equivalent. Moreover, (iii)
implies (ii), so in order to conclude it suffices only to prove that (ii) implies (i). So let us 
fix any morphism $\operatorname{w}_{\CATA}:B_{\CATA}\rightarrow A_{\CATA}$
in $\SETWA$; since $\overline{\kappa}$ is a pseudonatural
equivalence of pseudofunctors, we have a pair of equivalences $\overline{\kappa}(A_{\CATA})$,
$\overline{\kappa}(B_{\CATA})$ and an invertible $2$-morphism $\overline{\kappa}
(\operatorname{w}_{\CATA})$ as follows:

\[
\begin{tikzpicture}[xscale=3.8,yscale=-1]
    \node (A0_0) at (0, 0) {$\functor{F}_0(B_{\CATA})$};
    \node (A0_2) at (2, 0) {$\functor{F}_0(A_{\CATA})$};
    \node (A2_0) at (0, 2) {$\overline{\functor{G}}_0(B_{\CATA})=
      \overline{\functor{G}}_0\circ\functor{U}_{\SETWA,0}(B_{\CATA})$};
    \node (A2_2) at (2, 2) {$\overline{\functor{G}}_0\circ
      \functor{U}_{\SETWA,0}(A_{\CATA})=\overline{\functor{G}}_0(A_{\CATA})$.};

    \node (A1_1) at (1.2, 1) {$\overline{\kappa}(\operatorname{w}_{\CATA})$};
    \node (B1_1) [rotate=225] at (0.9, 1) {$\Longrightarrow$};
      
    \path (A2_0) edge [->]node [auto,swap] {$\scriptstyle{\overline{\functor{G}}_1\,
      \circ\,\functor{U}_{\SETWA,1}(\operatorname{w}_{\CATA})}$} (A2_2);
    \path (A0_0) edge [->]node [auto] {$\scriptstyle{\functor{F}_1
      (\operatorname{w}_{\CATA})}$} (A0_2);
    \path (A0_2) edge [->]node [auto] {$\scriptstyle{\overline{\kappa}(A_{\CATA})}$} (A2_2);
    \path (A0_0) edge [->]node [auto,swap] {$\scriptstyle{\overline{\kappa}(B_{\CATA})}$} (A2_0);
\end{tikzpicture}
\]

By Theorem~\ref{theo-01}, we have that $\functor{U}_{\SETWA,1}(\operatorname{w}_{\CATA})$ is an
internal equivalence, hence also $\overline{\functor{G}}_1\circ\functor{U}_{\SETWA,1}
(\operatorname{w}_{\CATA})$ is an internal equivalence. So by Lemmas~\ref{lem-04} and~\ref{lem-06}
we get easily that $\functor{F}_1(\operatorname{w}_{\CATA})$ is an
internal equivalence in $\CATB$, i.e.\ (i) holds.
\end{proof}

\section{Internal equivalences in a bicategory of fractions and (right) saturations}
In this section we will introduce the notion of right saturation of a class of morphisms in a
bicategory and we will prove some useful results about this concept.\\

Let us fix any pair $(\CATC,\SETW)$ satisfying conditions (BF);
according to~\cite[\S~2.4]{Pr}, the pseudofunctor $\functor{U}_{\SETW}:\CATC\rightarrow\CATC
\left[\SETWinv
\right]$ mentioned in Theorem~\ref{theo-01} sends each object $A$ to the same object in the target.
For every morphism $f:A\rightarrow B$, we have

\begin{equation}\label{eq-70}
\begin{tikzpicture}[xscale=1.8,yscale=-1.2]
    \node (A0_0) at (-0.5, 0) {$\functor{U}_{\SETW,1}(f)=\Big(B$};
    \node (A0_1) at (1, 0) {$B$};
    \node (A0_2) at (2, 0) {$A\Big)$;};
    \path (A0_1) edge [->]node [auto,swap] {$\scriptstyle{\id_B}$} (A0_0);
    \path (A0_1) edge [->]node [auto] {$\scriptstyle{f}$} (A0_2);
\end{tikzpicture}
\end{equation}
for every pair of morphisms $f^m:A\rightarrow B$ for $m=1,2$ and for every $2$-morphism $\gamma:f^1
\Rightarrow f^2$ in $\CATC$, we have

\begin{equation}\label{eq-95}
\functor{U}_{\SETW,2}(\gamma)=\Big[A,\id_A,\id_A,i_{\id_A\circ\id_A},\gamma\ast i_{\id_A}\Big].
\end{equation}

In particular, by Theorem~\ref{theo-01}, $\functor{U}_{\SETW,1}(f)$ is an internal equivalence
in $\CATC\left[\SETWinv\right]$ whenever $f$ belongs to $\SETW$ (actually, it is easy to see that
a quasi-inverse for \eqref{eq-70} is the triple $(B,f,\id_B)$). Then a natural question to
ask is the following: \emph{are there other morphisms in $\CATC$ that are sent to an internal 
equivalence by $\functor{U}_{\SETW}$?} In order to give an answer to this question, first of all we
give the following definition.

\begin{defin}\label{def-01}
Let us consider any bicategory $\CATC$ and any class of morphisms $\SETW$ in it (not necessarily
satisfying conditions (BF)).
Then we define the
(\emph{right}) \emph{saturation $\SETWsat$ of $\SETW$} as the class of all morphisms $f:B
\rightarrow A$ in $\CATC$, such that there are a pair of objects $C,D$ and a pair of morphisms $g:C
\rightarrow B$, $h:D\rightarrow C$, such that both $f\circ g$ and $g\circ h$ belong to $\SETW$.
We will say that $\SETW$ is (right) saturated if $\SETW=\SETWsat$.
\end{defin}

\begin{rem}\label{rem-02}
Whenever $\CATC$ is a $1$-category (with associated trivial bicategory $\CATC^2$), the notion above
coincides
with the notion of ``left saturation'' for a left multiplicative system implicitly given
in~\cite[Exercise~7.1]{KS}. In~\cite{KS}, the authors mainly focus on right multiplicative
systems (see~\cite[Definition~7.1.5]{KS}) and ``right saturations'', with left multiplicative systems
only mentioned explicitly in~\cite[Remark~7.1.7]{KS}. Note however that there is not a complete
agreement in the literature about what is a ``left'' and what is a ``right'' multiplicative system.
In the present paper ``right'' corresponds to ``left'' in~\cite{KS}. We prefer to use ``right''
instead of ``left'' in order to be consistent with the theory of right bicalculus of fractions
developed by Pronk (and with the theory of right calculus of fractions by Gabriel and Zisman).
In particular, one can easily see that a family $\SETW$ of
morphisms in a $1$-category $\CATC$ is a \emph{left} multiplicative system according to~\cite{KS}
if and only if the pair $(\CATC^2,\SETW)$ satisfies axioms (BF)
for a \emph{right} bicalculus of fractions as described in~\cite{Pr}. Moreover, in this
case the trivial bicategory associated to the left localization $\CATC^L_{\SETW}$ mentioned
in~\cite[Remark~7.1.18]{KS} is equivalent to the right bicategory of fractions
$\CATC^2\left[\SETWinv\right]$.
\end{rem}

\begin{rem}\label{rem-04}
Whenever the pair $(\CATC,\SETW)$ satisfies conditions (BF1) and (BF2),
we have $\SETW\subseteq\SETWsat$. Moreover, if $\SETW\subseteq\SETW'$, then
$\SETWsat\subseteq\SETW'_{\operatorname{sat}}$.
\end{rem}

We will prove in Proposition~\ref{prop-03}(i) below that $\SETWsat=\SETWsatsat$, thus
explaining the name ``saturation'' for such a class. The simplest example of (right) saturated class
is given by the class $\SETWequiv$ of all internal equivalences of any given bicategory $\CATC$,
as a consequence of Lemma~\ref{lem-02}. We will show in Section~\ref{sec-01} a non-trivial
example of a pair $(\CATC,\SETW)$ such that $\SETW=\SETWsat$.

\begin{defin}\label{def-02}
Let us fix any bicategory $\CATC$. According to~\cite[Definition~3.3]{PP}, we call a morphism $f:A
\rightarrow A$ in $\CATC$ a \emph{quasi-unit} if there is an invertible $2$-morphism $f\Rightarrow
\id_A$. We denote by $\SETWmin$ the class of quasi-units of $\CATC$. A direct check proves that
$(\CATC,\SETWmin)$ satisfies conditions (BF).
\end{defin}

\begin{lem}\label{lem-08}
Let us fix any pair $(\CATC,\SETW)$ satisfying conditions \emphatic{(BF)}.
Then:

\begin{enumerate}[\emphatic{(}i\emphatic{)}]
 \item $\SETWmin\subseteq\SETW$, hence $\SETWmin$ is the \emph{minimal} class satisfying conditions
  \emphatic{(BF)};
 \item the right saturation of $\SETWmin$ is the class $\SETWequiv$ of internal equivalences of
  $\CATC$; in particular $\SETWequiv\subseteq\SETWsat$, i.e.\ 
  $\SETWsat$ satisfies condition \emphatic{(\hyperref[BF1prime]{BF1})}$'$ \emphatic{(}see
  \emphatic{Remark~\ref{rem-01}}\emphatic{)};
 \item the induced pseudofunctors

  \[\functor{U}_{\SETWmin}:\,\CATC\longrightarrow\CATC\left[\SETWmininv\right]\quad\quad\textrm{and}
  \quad\quad\functor{U}_{\SETWequiv}:\,\CATC\longrightarrow\CATC\left[\SETWequivinv\right]\]
  are equivalences of bicategories.
\end{enumerate}
\end{lem}

\begin{proof}
Let us fix any object $A$ in $\CATC$; by (BF1)
(see~\cite{T3}), $\SETW$ contains $\id_A$. Then by (BF5)
$\SETW$ contains any morphism $f:A\rightarrow A$ such that there is an
invertible $2$-morphism $\xi:f\Rightarrow\id_A$. So we have proved that $\SETWmin\subseteq\SETW$.\\

Now let us prove (ii). Clearly $\SETWmin\subseteq\SETWequiv$, so by Remark~\ref{rem-04}
the right saturated of $\SETWmin$
is contained in the saturated of $\SETWequiv$, which is again $\SETWequiv$ by Lemma~\ref{lem-02}.
Conversely, let us suppose that $f:B\rightarrow A$ is an internal equivalence. Then there are a
morphism $g:A\rightarrow B$ and a pair of
invertible $2$-morphisms $\delta:\id_B\Rightarrow g\circ f$ and $\xi:f\circ g\Rightarrow\id_A$. Since
$\xi$ is invertible, then
we get that $f\circ g$ belongs to $\SETWmin$. Analogously, $g\circ f$ belongs to $\SETWmin$.
Therefore, if we set $h:=f$, we have proved that $f$ belongs the right saturation of $\SETWmin$,
so (ii) holds.\\

Now let us prove (iii). We have that $\SETWmin\subseteq\SETWequiv$, so $\id_{\CATC}:\CATC
\rightarrow\CATC$ sends each
morphism in $\SETWmin$ to an internal equivalence of $\CATC$. So  by
Theorem~\ref{theo-02} applied to $(\CATC,\SETWmin,\CATC)$ and to $\functor{F}:=\id_{\CATC}$, there are
a pseudofunctor
$\functor{R}:\CATC\left[\SETWmininv\right]\rightarrow\CATC$ and a pseudonatural equivalence of
pseudofunctors

\[\delta:\,\id_{\CATC}\Longrightarrow\functor{R}\circ\functor{U}_{\SETWmin}.\]

Then we consider the pseudonatural equivalence

\[\widetilde{\xi}:=\Big(i_{\functor{U}_{\SETWmin}}\ast\delta^{-1}\Big)\odot
\thetab{\functor{U}_{\SETWmin}}{\functor{R}}{\functor{U}_{\SETWmin}}:\,\,(\functor{U}_{\SETWmin}\circ
\functor{R})\circ\functor{U}_{\SETWmin}\Longrightarrow\id_{\CATC\left[\SETWmininv\right]}\circ\,
\functor{U}_{\SETWmin},\]
that is an equivalence in the bicategory $\operatorname{Hom}(\CATC,\CATC\left[
\SETWmininv\right])$. Since $\SETWmin\subseteq\SETWequiv$, then the pseudofunctor
$\mu_{\widetilde{\xi}}:\CATC\rightarrow\operatorname{Cyl}(\CATC)$ associated to $\widetilde{\xi}$
sends each morphism of $\SETWmin$ to an internal equivalence, i.e.\ $\widetilde{\xi}$
belongs to $\operatorname{Hom}_{\SETWmin}(\CATC,\CATC\left[
\SETWmininv\right])$. By the universal property of $\functor{U}_{\SETWmin}$ applied to
$\widetilde{\xi}$, there is an internal equivalence from $\functor{U}_{\SETWmin}\circ\functor{R}$
to $\id_{\CATC\left[\SETWmininv\right]}$ in the bicategory $\operatorname{Hom}(\CATC
\left[\SETWmininv\right],\CATC\left[\SETWmininv\right])$. By the first part of Lemma~\ref{lem-01},
this implies that there is an internal equivalence $\xi$ between the same $2$ objects in the
bicategory $\operatorname{Hom}'(\CATC
\left[\SETWmininv\right],\CATC\left[\SETWmininv\right])$. Since $\delta$ is an internal equivalence
in the bicategory $\operatorname{Hom}'(\CATC,\CATC)$, then we have proved that
$\functor{U}_{\SETWmin}$ is an equivalence of bicategories (see~\cite[\S~2.2]{L}), with
$\functor{R}$ as quasi-inverse.
The proof for $\functor{U}_{\SETWequiv}$ is analogous.
\end{proof}

Now we have:

\begin{prop}\label{prop-02}
Let us fix any pair $(\CATC,\SETW)$ satisfying conditions
\emphatic{(BF)} and any
morphism $f:B\rightarrow A$ in $\CATC$. Then the morphism

\[
\begin{tikzpicture}[xscale=2.4,yscale=-1.2]
    \node (A0_0) at (-0.35, 0) {$\functor{U}_{\SETW,1}(f)=\Big(B$};
    \node (A0_1) at (1, 0) {$B$};
    \node (A0_2) at (2, 0) {$A\Big)$};
    \path (A0_1) edge [->]node [auto,swap] {$\scriptstyle{\id_B}$} (A0_0);
    \path (A0_1) edge [->]node [auto] {$\scriptstyle{f}$} (A0_2);
\end{tikzpicture}
\]
\emphatic{(}see \eqref{eq-70}\emphatic{)} is an internal equivalence in
$\CATC\left[\SETWinv\right]$ if and only if $f$ belongs to $\SETWsat$.
\end{prop}

If $\CATC$ is a $1$-category considered as a trivial bicategory, then this result coincides with the
analogous of~\cite[Proposition~7.1.20(i)]{KS} for left multiplicative systems instead of right
multiplicative systems (see Remark~\ref{rem-02}). The case of a non-trivial bicategory is much
longer, but conceptually similar; we refer to the Appendix for the details.

\begin{cor}\label{cor-02}
Let us fix any pair $(\CATC,\SETW)$ satisfying conditions \emphatic{(BF)}.
Given any pair of objects $A^1,A^2$ in $\CATC$, any internal
equivalence from $A^1$ to $A^2$ in $\CATC\left[\SETWinv\right]$ is necessarily of the form

\begin{equation}\label{eq-34}
\begin{tikzpicture}[xscale=2.4,yscale=-1.2]
    \node (A0_0) at (0, 0) {$A^1$};
    \node (A0_1) at (1, 0) {$A^3$};
    \node (A0_2) at (2, 0) {$A^2$,};
    \path (A0_1) edge [->]node [auto,swap] {$\scriptstyle{\operatorname{w}}$} (A0_0);
    \path (A0_1) edge [->]node [auto] {$\scriptstyle{f}$} (A0_2);
\end{tikzpicture}
\end{equation}
with $\operatorname{w}$ in $\SETW$ and $f$ in $\SETWsat$. Conversely, any such morphism is an
internal equivalence in $\CATC\left[\SETWinv\right]$.
\end{cor}

\begin{proof}
Let us suppose that \eqref{eq-34} is an internal equivalence in $\CATC\left[\SETWinv\right]$. By the
description of morphisms in a bicategory of fractions, $\operatorname{w}$ belongs to $\SETW$. So 
by Theorem~\ref{theo-01} the morphism

\[
\begin{tikzpicture}[xscale=2.4,yscale=-1.2]
    \node (A0_0) at (-0.4, 0) {$\functor{U}_{\SETW,1}(\operatorname{w})=\Big(A^3$};
    \node (A0_1) at (1, 0) {$A^3$};
    \node (A0_2) at (2, 0) {$A^1\Big)$};
    
    \path (A0_1) edge [->]node [auto,swap] {$\scriptstyle{\id_{A^3}}$} (A0_0);
    \path (A0_1) edge [->]node [auto] {$\scriptstyle{\operatorname{w}}$} (A0_2);
\end{tikzpicture}
\]
is an internal equivalences in $\CATC\left[\SETWinv\right]$. Then using Lemma~\ref{lem-06}
we get that also $(A^3,\operatorname{w},f)\circ\functor{U}_{\SETW,1}(\operatorname{w})$
is an internal equivalence. Now let us suppose that choices \hyperref[C]{C}$(\SETW)$ give data as
in the upper part of the following diagram, with $\operatorname{v}^1$ in $\SETW$ and $\eta$
invertible:

\[
\begin{tikzpicture}[xscale=2.2,yscale=-0.8]
    \node (A0_1) at (1, 0) {$A^4$};
    \node (A1_0) at (0, 2) {$A^3$};
    \node (A1_2) at (2, 2) {$A^3$.};
    \node (A2_1) at (1, 2) {$A^1$};
    
    \node (A1_1) at (1, 0.95) {$\eta$};
    \node (B1_1) at (1, 1.4) {$\Rightarrow$};
    
    \path (A1_2) edge [->]node [auto] {$\scriptstyle{\operatorname{w}}$} (A2_1);
    \path (A0_1) edge [->]node [auto] {$\scriptstyle{\operatorname{v}^2}$} (A1_2);
    \path (A1_0) edge [->]node [auto,swap] {$\scriptstyle{\operatorname{w}}$} (A2_1);
    \path (A0_1) edge [->]node [auto,swap] {$\scriptstyle{\operatorname{v}^1}$} (A1_0);
\end{tikzpicture}
\]

Then by~\cite[\S~2.2]{Pr} we have

\[\Big(A^3,\operatorname{w},f\Big)\circ\functor{U}_{\SETW,1}(\operatorname{w})=\Big(A^4,
\id_{A^3}\circ\operatorname{v}^1,f\circ\operatorname{v}^2\Big).\]

By (BF4a) and (BF4b) applied to $\eta$,
there are an object $A^5$, a morphism $\operatorname{v}^3:A^5\rightarrow
A^4$ in $\SETW$ and an invertible $2$-morphism $\varepsilon:\operatorname{v}^1\circ\operatorname{v}^3
\Rightarrow\operatorname{v}^2\circ\operatorname{v}^3$. Then we define an invertible $2$-morphism
in $\CATC\left[\SETWinv\right]$ as follows

\begin{gather*}
\Gamma:=\Big[A^5,\operatorname{v}^3,\operatorname{v}^2\circ\operatorname{v}^3,
 \Big(i_{\id_{A^3}}\ast\varepsilon\Big)\odot\thetab{\id_{A^3}}{\operatorname{v}^1}
 {\operatorname{v}^3},\thetab{f}{\operatorname{v}^2}{\operatorname{v}^3}\Big]:\\
\Big(A^3,\operatorname{w},f\Big)\circ\functor{U}_{\SETW,1}(\operatorname{w})
 \Longrightarrow\functor{U}_{\SETW,1}(f).
\end{gather*}

By Lemma~\ref{lem-04} applied to $\Gamma$, we get that
$\functor{U}_{\SETW,1}(f)$ is an internal equivalence, so by Proposition~\ref{prop-02} $f$ belongs
to $\SETWsat$.\\

Conversely,
if $f$ belongs to $\SETWsat$, again by Proposition~\ref{prop-02} we get that $(A^3,\id_{A^3},
f)$ is an internal equivalence in $\CATC\left[\SETWinv\right]$. Since also $(A^3,\operatorname{w},
\id_{A^3})$ is an internal equivalence (because it is a quasi-inverse for $(A^3,\id_{A^3},
\operatorname{w})$), then also the composition

\[\Big(A^3,\id_{A^3},f\Big)\circ\Big(A^3,\operatorname{w},
\id_{A^3}\Big)\stackrel{(\operatorname{C}1)}{=}\Big(A^3,\operatorname{w}\circ\id_{A^3},f
\circ\id_{A^3}\Big)\]
is an internal
equivalence in $\CATC\left[\SETWinv\right]$. From this and Lemma~\ref{lem-04} we get that
\eqref{eq-34} is an internal equivalence.
\end{proof}

Given the previous results, $2$ natural questions arise for any $(\CATC,\SETW)$ satisfying
conditions (BF):

\begin{itemize}
 \item does the pair $(\CATC,\SETWsat)$ satisfy conditions (BF)?
 \item if yes, is the resulting right bicategory of fractions equivalent to $\CATC\left[\SETWinv
  \right]$?
\end{itemize}

Both questions have positive answers, as we are going to show below.

\begin{lem}\label{lem-05}
Let us fix any pair $(\CATC,\SETW)$ satisfying conditions \emphatic{(BF)}.
Then also the pair $(\CATC,\SETWsat)$ satisfies the same conditions.
\end{lem}

This is the analogous of~\cite[Exercise~7.1]{KS} (for left multiplicative systems instead of
right multiplicative systems, see Remark~\ref{rem-02}) in the more complicated
framework of bicategories instead of
$1$-categories. A detailed proof is given in the Appendix.

\begin{lem}\label{lem-03}
Let us fix any any pseudofunctor $\functor{F}:\CATC\rightarrow\CATD$ and any class of morphisms
$\SETW$ in $\CATC$. If $\functor{F}$ sends each morphism of $\SETW$ to an internal equivalence, then
it sends each morphism of $\SETWsat$ to an internal equivalence.
\end{lem}

\begin{proof}
Let us fix any morphism $f:B\rightarrow
A$ in $\SETWsat$ and let us choose any pair of objects $C,D$ and any pair of 
morphisms $g:C\rightarrow B$ and $h:D\rightarrow C$,
such that both $f\circ g$ and $g\circ h$ belong to $\SETW$. Then there is an invertible $2$-morphism
(the associator for $\functor{F}$ relative to the pair $(f,g)$) in $\CATD$ from the internal
equivalence $\functor{F}_1(f\circ g)$ to the morphism $\functor{F}_1(f)\circ\functor{F}_1(g)$. So by
Lemma~\ref{lem-04} we have that $\functor{F}_1(f)\circ\functor{F}_1(g)$ is an internal equivalence of
$\CATD$; analogously $\functor{F}_1(g)\circ\functor{F}_1(h)$ is an internal equivalence. Then by
Lemma~\ref{lem-02} applied to the bicategory $\CATD$ we conclude that $\functor{F}_1(f)$ is
an internal equivalence of $\CATD$.
\end{proof}

\begin{prop}\label{prop-01}
Let us fix any pair $(\CATC,\SETW)$ satisfying conditions \emphatic{(BF)}.
Then there are a bicategory of fractions $\CATC\left[\SETWsatinv\right]$ and a pseudofunctor
as in \eqref{eq-46}, with the universal property. Moreover, there are $2$
equivalences of bicategories

\[\functor{H}:\,\CATC\left[\SETWsatinv\right]\longrightarrow\CATC\left[\SETWinv\right]\quad
\textrm{and}\quad\functor{L}:\,\CATC\left[\SETWinv\right]\longrightarrow\CATC\left[\SETWsatinv
\right],\]
one the quasi-inverse of the other, and a pseudonatural equivalence of pseudofunctors $\tau:
\functor{U}_{\SETW}\Rightarrow\functor{H}\circ\functor{U}_{\SETWsat}$ that is a morphism in
$\operatorname{Hom}'_{\SETWsat}(\CATC,\CATC\left[\SETWinv\right])$.
\end{prop}

\begin{proof}
Using Lemma~\ref{lem-05} and Theorem~\ref{theo-01}, there is a bicategory of fractions $\CATC
\left[\SETWsatinv\right]$ and a pseudofunctor as in \eqref{eq-46}, with the universal property. Since
$\SETW\subseteq\SETWsat$, then $\functor{U}_{\SETWsat}$ sends each morphism of $\SETW$ to an internal
equivalence. So using Theorem~\ref{theo-02} for $(\CATC,\SETW,
\functor{U}_{\SETWsat})$, there are a pseudofunctor $\functor{L}$ as above and
a pseudonatural equivalence

\begin{equation}\label{eq-05}
\zeta:\,\functor{U}_{\SETWsat}\Longrightarrow\functor{L}\circ\functor{U}_{\SETW}\quad\textrm{in}
\quad\operatorname{Hom}'_{\SETW}\left(\CATC,\CATC\left[\SETWsatinv\right]\right).
\end{equation}

By Proposition~\ref{prop-02}, $\functor{U}_{\SETW}$ sends each morphism of $\SETWsat$ to an
internal equivalence of $\CATC\left[\SETWinv\right]$; so using Theorem~\ref{theo-02} for $(\CATC,
\SETWsat,\functor{U}_{\SETW})$ there are a pseudofunctor $\functor{H}$ as in the
claim and a pseudonatural equivalence

\begin{equation}\label{eq-96}
\tau:\,\functor{U}_{\SETW}\Longrightarrow\functor{H}\circ
\functor{U}_{\SETWsat}\quad\operatorname{in}\quad\operatorname{Hom}'_{\SETWsat}\left(\CATC,\CATC
\left[\SETWinv\right]\right)\subseteq\operatorname{Hom}'_{\SETW}\left(\CATC,\CATC\left[\SETWinv
\right]\right).
\end{equation}

Now we consider the pseudonatural equivalence

\[\widetilde{\xi}:=\tau^{-1}\odot\Big(i_{\functor{H}}\ast\zeta^{-1}\Big)\odot\thetab{\functor{H}}
{\functor{L}}{\,\functor{U}_{\SETW}}:\,\,(\functor{H}\circ\functor{L})\circ\functor{U}_{\SETW}
\Longrightarrow\functor{U}_{\SETW}=\id_{\CATC\left[\SETWinv\right]}\circ\,\functor{U}_{\SETW}.\]

By \eqref{eq-05} and \eqref{eq-96}, $\widetilde{\xi}$ is an internal equivalence in the bicategory
$\operatorname{Hom}'_{\SETW}(\CATC,\CATC\left[\SETWinv\right])$ $\subseteq
\operatorname{Hom}_{\SETW}(\CATC,\CATC\left[\SETWinv\right])$, so by the universal property of
$\functor{U}_{\SETW}$, (\hyperref[X2]{X2}) and Lemma~\ref{lem-01}, there is a pseudonatural
equivalence $\xi:\functor{H}\circ\functor{L}\Rightarrow\id_{\CATC\left[\SETWinv\right]}$.\\

Moreover, we consider the pseudonatural equivalence 

\[\widetilde{\delta}:=\thetaa{\functor{L}}{\functor{H}}{\functor{U}_{\SETWsat}}\odot\Big(
i_{\functor{L}}\ast\tau\Big)\odot\zeta:\,\,\id_{\CATC\left[\SETWsatinv\right]}\circ\,
\functor{U}_{\SETWsat}\Longrightarrow(\functor{L}\circ\functor{H})\circ\functor{U}_{\SETWsat}.\]

By \eqref{eq-05} and \eqref{eq-96} the pseudofunctor $\mu_{\widetilde{\delta}}:\CATC\rightarrow
\operatorname{Cyl}
(\CATC\left[\SETWsatinv\right])$ associated to $\widetilde{\delta}$ sends each morphism of $\SETW$ to
an internal equivalence. So by Lemma~\ref{lem-03} we conclude that $\mu_{\widetilde{\delta}}$
sends each morphism of $\SETWsat$ to an internal equivalence, so $\widetilde{\delta}$ is an
internal equivalence in $\operatorname{Hom}_{\SETWsat}(\CATC,\CATC\left[\SETWsatinv\right])$.
Therefore by the universal property of
$\functor{U}_{\SETWsat}$, (\hyperref[X2]{X2}) and Lemma~\ref{lem-01}, there is a pseudonatural
equivalence $\delta:\id_{\CATC\left[\SETWsatinv\right]}\Rightarrow\functor{L}\circ\functor{H}$.
Using~\cite[\S~2.2]{L}, this proves that $\functor{L}$ and $\functor{H}$ are equivalences of
bicategories, one the quasi-inverse of the other.
\end{proof}

\begin{prop}\label{prop-03}
Let us fix any pair $(\CATC,\SETW)$ satisfying axioms \emphatic{(BF)}. Then:

\begin{enumerate}[\emphatic{(}i\emphatic{)}]
 \item the classes $\SETWsat$ and $\SETWsatsat$ coincide \emphatic{(}i.e.\ $\SETWsat$ is
  \emph{(}right\emph{)} saturated\emphatic{)}; 
 \item the class $\SETWsat$ satisfies the ``$\,2$-out-of-$3$'' property, i.e.\ given any pair of
  morphisms $f:B\rightarrow A$ and $g:C\rightarrow B$, if any $2$ of the $3$ morphisms $f,g$ and $f
  \circ g$ belong to $\SETWsat$, then so does the third one.
\end{enumerate}
\end{prop}

In Lemma~\ref{lem-08}(ii) we proved that $\SETWequiv$ is right saturated; for that class,
(ii) above is simply the already stated Lemma~\ref{lem-06}, that we will use explicitly in the
proof below.

\begin{proof}
Using Lemma~\ref{lem-05} and Remark~\ref{rem-04}, we have that
$\SETWsat\subseteq\SETWsatsat$, so we need only to prove the other inclusion. So let us
fix any morphism $f:B\rightarrow A$ belonging to $\SETWsatsat$. By Lemma~\ref{lem-05}, the pair
$(\CATC,\SETWsat)$ satisfies conditions (BF),
so we can apply Proposition~\ref{prop-02} for such a pair. Then we get that the morphism

\[
\begin{tikzpicture}[xscale=2.4,yscale=-1.2]
    \node (A0_0) at (-0.4, 0) {$\functor{U}_{\SETWsat,1}(f)=\Big(B$};
    \node (A0_1) at (1, 0) {$B$};
    \node (A0_2) at (2, 0) {$A\Big)$};
    
    \path (A0_1) edge [->]node [auto,swap] {$\scriptstyle{\id_B}$} (A0_0);
    \path (A0_1) edge [->]node [auto] {$\scriptstyle{f}$} (A0_2);
\end{tikzpicture}
\]
is an internal equivalence in $\CATC\left[\SETWsatinv\right]$. Hence 
$\functor{H}_1\circ\functor{U}_{\SETWsat,1}(f)$ is an internal equivalence in $\CATC\left[
\SETWinv\right]$, where $\functor{H}$ is the pseudofunctor obtained in Proposition~\ref{prop-01}.
Using that proposition, we have an invertible $2$-morphism

\[\tau_{f}:\,\functor{U}_{\SETW,1}(f)\Longrightarrow\functor{H}_1\circ\functor{U}_{\SETWsat,1}(f)\]
in $\CATC\left[\SETWinv\right]$. Then by Lemma~\ref{lem-04} we conclude that also
$\functor{U}_{\SETW,1}(f)$ is an internal equivalence in $\CATC\left[\SETWinv\right]$. By
Proposition~\ref{prop-02}, this implies that $f$ belongs to $\SETWsat$, so (i) holds.\\

Now let us suppose that any $2$ of the $3$ morphisms $f,g$ and $f\circ g$ belong to $\SETWsat$.
Then by Proposition~\ref{prop-02}, $2$ of the $3$ morphisms $\functor{U}_{\SETW,1}(f)$, 
$\functor{U}_{\SETW,1}(g)$ and $\functor{U}_{\SETW,1}(f\circ g)$ are internal equivalences
in $\CATC\left[\SETWinv\right]$. Using Lemma~\ref{lem-04} on the associator of $\functor{U}_{\SETW}$
relative to the pair $(f,g)$, this implies that $2$ of the $3$ morphisms $\functor{U}_{\SETW,1}(f)$, 
$\functor{U}_{\SETW,1}(g)$ and $\functor{U}_{\SETW,1}(f)\circ\functor{U}_{\SETW,1}(g)$
are internal equivalences. By Lemma~\ref{lem-06}, all such $3$ morphisms are internal equivalences.\\

Again by Lemma~\ref{lem-04}, this implies that the $3$ morphisms
$\functor{U}_{\SETW,1}(f)$, $\functor{U}_{\SETW,1}(g)$ and $\functor{U}_{\SETW,1}(f\circ g)$ are
all internal equivalences. Again by Proposition~\ref{prop-02}, this gives the claim.
\end{proof}

\begin{rem}
In general, even if a pair $(\CATC,\SETW)$ satisfies axioms (BF),
$\SETW$ does not have to satisfy the ``$\,2$-out-of-$3$'' property. This is for
example the case when we consider the class $\WRedAtl$ of all ``refinements'' in the $2$-category
$\RedAtl$ of reduced orbifold atlases. We described such data in our paper~\cite{T7}; we refer
directly to it for all the relevant definitions. Again referring to that paper, the class
of all ``unit weak equivalences'' of reduced orbifold atlases
satisfies the ``$\,2$-out-of-$3$'' property, but it is not saturated. This proves that in general the
converse of Proposition~\ref{prop-03}(ii) does not hold, namely $\SETW$ can satisfy the
``$\,2$-out-of-$3$'' property without being right saturated.\\

An interesting known case when
the ``$\,2$-out-of-$3$'' property holds is the case when we consider the class $\WEGpd$ of all
Morita equivalences of \'etale differentiable 
groupoids. In this case, it was proved in~\cite[Lemma~8.1]{PS} that
$\WEGpd$ has the mentioned property. Actually, in the last part of this paper we will prove that
$\WEGpd$ is right saturated, thus giving another proof of~\cite[Lemma~8.1]{PS}.
\end{rem}

Now we are able to give the proof of the first main result of this paper.

\begin{proof}[Proof of Theorem~\ref{theo-04}.]
Let us define a pseudofunctor as follows:

\[\overline{\functor{F}}:=\functor{U}_{\SETWB}\circ\functor{F}:\,\CATA\longrightarrow\CATB\left[
\SETWBinv\right]\]
and let us denote by $\SETWequiv$ the class of internal equivalences of the bicategory
$\CATB\left[\SETWBinv\right]$. By Proposition~\ref{prop-02},

\[\functor{U}_{\SETWB}^{-1}(\SETWequiv)=\SETWBsat,\]
hence

\begin{equation}\label{eq-61}
\overline{\functor{F}}^{\,-1}(\SETWequiv)=\functor{F}^{\,-1}(\SETWBsat). 
\end{equation}

Now we use Theorem~\ref{theo-02} for $\CATB$ replaced by $\CATB\left[\SETWBinv\right]$ and
$\functor{F}$ replaced by $\overline{\functor{F}}$. So we have that (i), (iii) and (iv) are
equivalent.\\

Since $\SETWA\subseteq\SETWAsat$, then (ii) implies (i), so we need only to prove that (i) implies (ii).
So let us suppose that $\functor{F}_1(\SETWA)\subseteq\SETWBsat$; then by \eqref{eq-61},
$\overline{\functor{F}}$ sends each morphism of $\SETWA$ to an internal equivalence.
Therefore by Lemma~\ref{lem-03}, $\overline{\functor{F}}$ sends each morphism of $\SETWAsat$
to an internal equivalence. Again by \eqref{eq-61}, we conclude that $\functor{F}_1(\SETWAsat)
\subseteq\SETWBsat$.
\end{proof}

\section{The induced pseudofunctor $\widetilde{\functor{G}}$}
This section is mainly used to give the proof of Theorem~\ref{theo-03}.
The essential part of such a proof relies on the following proposition.

\begin{prop}\label{prop-06}
Let us fix any $2$ pairs $(\CATA,\SETWA)$ and $(\CATB,\SETWB)$, both satisfying conditions
\emphatic{(BF)} and any pseudofunctor $\functor{F}:\CATA\rightarrow\CATB$
such that $\functor{F}_1(\SETWA)\subseteq\SETWB$. Moreover, let us fix any set of choices
\emphatic{\hyperref[C]{C}}$(\SETWB)$ satisfying condition \emphatic{(\hyperref[C3]{C3})}.
Then there are a pseudofunctor

\[\functor{M}:\,\CATA\Big[\SETWAinv\Big]\longrightarrow\CATB\Big[\SETWBinv\Big]\]
\emphatic{(}where $\CATB[\SETWBinv]$ is the bicategory of fractions induced by choices
\emphatic{\hyperref[C]{C}}$(\SETWB)$\emphatic{)}    
and a pseudonatural equivalence $\zeta:\functor{U}_{\SETWB}\circ\functor{F}
\Rightarrow\functor{M}\circ\functor{U}_{\SETWA}$, such that:

\begin{enumerate}[\emphatic{(}I\emphatic{)}]
 \item the pseudofunctor $\mu_{\zeta}:\CATA\rightarrow\operatorname{Cyl}
  \left(\CATB\left[\SETWBinv\right]\right)$ associated to $\zeta$ sends each morphism
  of $\SETWA$ to an internal equivalence;
 
 \item for each object $A_{\CATA}$, we have $\functor{M}_0(A_{\CATA})=
  \functor{F}_0(A_{\CATA})$;

 \item for each morphism $(A'_{\CATA},\operatorname{w}_{\CATA},f_{\CATA}):A_{\CATA}\rightarrow
  B_{\CATA}$ in $\CATA\left[\SETWAinv\right]$, we have 
   
  \begin{equation}\label{eq-63}
  \functor{M}_1\Big(A'_{\CATA},\operatorname{w}_{\CATA}\Big)=\Big(\functor{F}_0(A'_{\CATA}),
  \functor{F}_1(\operatorname{w}_{\CATA})\circ\id_{\functor{F}_0(A'_{\CATA})},
  \functor{F}_1(f_{\CATA})\circ\id_{\functor{F}_0(A'_{\CATA})}\Big);
  \end{equation}

 \item for each $2$-morphism

  \begin{equation}\label{eq-75}
  \Big[A^3_{\CATA},\operatorname{v}^1_{\CATA},\operatorname{v}^2_{\CATA},\alpha_{\CATA},
  \beta_{\CATA}\Big]:\Big(A^1_{\CATA},\operatorname{w}^1_{\CATA},f^1_{\CATA}\Big)\Longrightarrow
  \Big(A^2_{\CATA},\operatorname{w}^2_{\CATA},f^2_{\CATA}\Big)
  \end{equation}
  in $\CATA\left[\SETWAinv\right]$, we have
  
  \begin{gather}
  \label{eq-68} \functor{M}_2\Big(\Big[A^3_{\CATA},\operatorname{v}^1_{\CATA},
   \operatorname{v}^2_{\CATA},\alpha_{\CATA},\beta_{\CATA}\Big]\Big)= 
   \Big[\functor{F}_0(A^3_{\CATA}),\functor{F}_1(\operatorname{v}^1_{\CATA}),
   \functor{F}_1(\operatorname{v}^1_{\CATA}), \\
  \nonumber \Big(\pi_{\functor{F}_1(\operatorname{w}^2_{\CATA})}^{-1}\ast i_{\functor{F}_1
   (\operatorname{v}^2_{\CATA})}\Big)\odot\psi^{\functor{F}}_{\operatorname{w}^2_{\CATA},
   \operatorname{v}^2_{\CATA}}\odot\functor{F}_2(\alpha_{\CATA})\odot\Big(
   \psi^{\functor{F}}_{\operatorname{w}^1_{\CATA},
   \operatorname{v}^1_{\CATA}}\Big)^{-1}\odot \\
  \nonumber \odot\Big(\pi_{\functor{F}_1(\operatorname{w}^1_{\CATA})}
   \ast i_{\functor{F}_1(\operatorname{v}^1_{\CATA})}\Big),
   \Big(\pi_{\functor{F}_1(f^2_{\CATA})}^{-1}\ast i_{\functor{F}_1
   (\operatorname{v}^2_{\CATA})}\Big)\odot\psi^{\functor{F}}_{f^2_{\CATA},
   \operatorname{v}^2_{\CATA}}\odot \\
  \nonumber \odot\,\functor{F}_2(\beta_{\CATA})\odot\Big(\psi^{\functor{F}}_{f^1_{\CATA},
   \operatorname{v}^1_{\CATA}}\Big)^{-1}\odot\Big(\pi_{\functor{F}_1(f^1_{\CATA})}
   \ast i_{\functor{F}_1(\operatorname{v}^1_{\CATA})}\Big)\Big]
  \end{gather}
  \emphatic{(}where the $2$-morphisms $\psi_{\bullet}^{\functor{F}}$ are the associators of
  $\functor{F}$ and the $2$-morphisms $\pi_{\bullet}$ are the right unitors of $\CATB$\emphatic{)}.
\end{enumerate}
\end{prop}

\begin{proof}
In order to simplify a bit the exposition, \emph{we assume for the moment that all the unitors and
associators of $\CATA$ and $\CATB$ are trivial} (i.e.\ that $\CATA$ and $\CATB$ are $2$-categories).
Using (\hyperref[C1]{C1}) and (\hyperref[C2]{C2}), this implies
easily that also the unitors for $\CATA\left[\SETWAinv\right]$ and $\CATB\left[\SETWBinv\right]$
are trivial (even if in general the same is not true for the associators of such bicategories).
For simplicity, \emph{we assume also that $\functor{F}$ is a strict
pseudofunctor} (i.e.\ that it preserves compositions and identities). At the end of the proof
we will discuss briefly the general case.\\

We set $\overline{\functor{F}}:=\functor{U}_{\SETWB}\circ\functor{F}:\CATA\rightarrow\CATB\left[
\SETWBinv\right]$. Then
we follow the proof of~\cite[Theorem~21]{Pr} in order to give the explicit description of a
pair $(\functor{M},\zeta)$ induced by $\overline{\functor{F}}$ and that satisfies the claim.
According to the mentioned proof 
in~\cite{Pr}, first of all we need to fix some choices as follows.

\begin{enumerate}[(A)]
 \item We have to choose a structure of bicategory on $\CATA\left[\SETWAinv\right]$ and on $\CATB
  \left[\SETWBinv\right]$. For that, we fix \emph{any} set of choices \hyperref[C]{C}$(\SETWA)$;
  moreover, we fix any set of choices \hyperref[C]{C}$(\SETWB)$ satisfying condition
  (\hyperref[C3]{C3}) (and obviously satisfying also conditions
  (\hyperref[C1]{C1}) and (\hyperref[C2]{C2}) by definition of set of choices for $\SETW$,
  see \S~\ref{sec-03}).
 \item We need to fix some choices as in the proof of~\cite[Theorem~21]{Pr}. To be more precise,
  given any morphism $\operatorname{w}_{\CATA}:A'_{\CATA}\rightarrow A_{\CATA}$ in $\SETWA$, we
  need to \emph{choose} data as follows in $\CATB\left[\SETWBinv\right]$:

  \begin{itemize}
   \item a morphism $\functor{P}(\operatorname{w}_{\CATA}):\overline{\functor{F}}_0(A_{\CATA})
    \rightarrow\overline{\functor{F}}_0(A'_{\CATA})$,
   \item an invertible $2$-morphism $\Delta(\operatorname{w}_{\CATA}):\id_{\overline{\functor{F}}_0
    (A_{\CATA})}\Rightarrow\overline{\functor{F}}_1(\operatorname{w}_{\CATA})\circ\functor{P}
    (\operatorname{w}_{\CATA})$,
   \item an invertible $2$-morphism $\Xi(\operatorname{w}_{\CATA}):\functor{P}
    (\operatorname{w}_{\CATA})\circ\overline{\functor{F}}_1(\operatorname{w}_{\CATA})\Rightarrow
    \id_{\overline{\functor{F}}_0(A'_{\CATA})}$,
  \end{itemize}
  such that the quadruple $(\overline{\functor{F}}_1(\operatorname{w}_{\CATA}),\functor{P}
  (\operatorname{w}_{\CATA}),\Delta(\operatorname{w}_{\CATA}),\Xi(\operatorname{w}_{\CATA}))$ is an
  \emph{adjoint equivalence} in $\CATB\left[\SETWBinv\right]$ (see \S~\ref{sec-02}).
  In particular, we need to choose $\functor{P}(\operatorname{w}_{\CATA})$ so that it is a
  quasi-inverse for
  \[
  \begin{tikzpicture}[xscale=3.2,yscale=-1.2]
    \node (A0_0) at (-0.2, 0) {$\overline{\functor{F}}_1(\operatorname{w}_{\CATA})=
      \Big(\functor{F}_0(A'_{\CATA})$};
    \node (A0_1) at (1, 0) {$\functor{F}_0(A'_{\CATA})$};
    \node (A0_2) at (2, 0) {$\functor{F}_0(A_{\CATA})\Big)$.};
    
    \path (A0_1) edge [->]node [auto,swap] {$\scriptstyle{\id_{\functor{F}_0(A'_{\CATA})}}$} (A0_0);
    \path (A0_1) edge [->]node [auto]
      {$\scriptstyle{\functor{F}_1(\operatorname{w}_{\CATA})}$} (A0_2);
  \end{tikzpicture}
  \]
  Since we assumed that $\functor{F}_1(\SETWA)\subseteq\SETWB$, then we are allowed to choose
  
  \begin{equation}\label{eq-37}
  \begin{tikzpicture}[xscale=3.2,yscale=-1.2]
    \node (A0_0) at (-0.2, 0) {$\functor{P}(\operatorname{w}_{\CATA}):=
      \Big(\functor{F}_0(A_{\CATA})$};
    \node (A0_1) at (1, 0) {$\functor{F}_0(A'_{\CATA})$};
    \node (A0_2) at (2, 0) {$\functor{F}_0(A'_{\CATA})\Big)$,};
    \path (A0_1) edge [->]node [auto,swap]
      {$\scriptstyle{\functor{F}_1(\operatorname{w}_{\CATA})}$} (A0_0);
    \path (A0_1) edge [->]node [auto] {$\scriptstyle{\id_{\functor{F}_0(A'_{\CATA})}}$} (A0_2);
  \end{tikzpicture}
  \end{equation}
  and
  
  \begin{gather}
  \label{eq-39}\Xi(\operatorname{w}_{\CATA}):=\Big[\functor{F}_0(A'_{\CATA}),\functor{F}_1
   (\operatorname{w}_{\CATA}),\id_{\functor{F}_0(A'_{\CATA})},i_{\functor{F}_1
   (\operatorname{w}_{\CATA})},i_{\functor{F}_1(\operatorname{w}_{\CATA})}\Big]: \\
  \nonumber \Big(\functor{F}_0(A_{\CATA}),\id_{\functor{F}_0(A_{\CATA})},\id_{\functor{F}_0
   (A_{\CATA})}\Big)\Longrightarrow \\
  \nonumber \Longrightarrow\Big(\functor{F}_0(A'_{\CATA}),\functor{F}_1
   (\operatorname{w}_{\CATA}),\functor{F}_1(\operatorname{w}_{\CATA})\Big)=\overline{\functor{F}}_1
   (\operatorname{w}_{\CATA})\circ\functor{P}(\operatorname{w}_{\CATA}).
  \end{gather}
  Since \hyperref[C]{C}$(\SETWB)$ satisfies condition (\hyperref[C3]{C3}), then we get
  that $\functor{P}
  (\operatorname{w}_{\CATA})\circ\overline{\functor{F}}_1(\operatorname{w}_{\CATA})$ is the identity
  of $\functor{F}_0(A'_{\CATA})$ in $\CATB\left[\SETWB\right]$ so we can choose
  
  \begin{gather*}
  \Delta(\operatorname{w}_{\CATA}):=i_{\id_{\functor{F}_0(A'_{\CATA})}}=\Big[\functor{F}_0(A'_{\CATA}),
   \id_{\functor{F}_0(A'_{\CATA})},\id_{\functor{F}_0(A'_{\CATA})},
   i_{\id_{\functor{F}_0(A'_{\CATA})}},i_{\id_{\functor{F}_0(A'_{\CATA})}}\Big]: \\
  \functor{P}(\operatorname{w}_{\CATA})\circ\overline{\functor{F}}_1(\operatorname{w}_{\CATA})=
  \Big(\functor{F}_0(A'_{\CATA}),\id_{\functor{F}_0(A'_{\CATA})},\id_{\functor{F}_0(A'_{\CATA})}
   \Big)\Longrightarrow \\
  \Longrightarrow\Big(\functor{F}_0(A'_{\CATA}),\id_{\functor{F}_0(A'_{\CATA})},
   \id_{\functor{F}_0(A'_{\CATA})}\Big).
  \end{gather*}
  Using~\cite[Proposition~20]{Pr} and condition (\hyperref[C3]{C3}), we get that 
  the quadruple $(\overline{\functor{F}}_1(\operatorname{w}_{\CATA}),\functor{P}
  (\operatorname{w}_{\CATA}),\Delta(\operatorname{w}_{\CATA}),\Xi(\operatorname{w}_{\CATA}))$ is
  an adjoint equivalence as required.
\end{enumerate}

Now we follow the proof of~\cite[Theorem~21]{Pr} in order to define the pair
$(\functor{M},\zeta)$ induced by $\overline{\functor{F}}$ and by the
previous choices. For every object $A_{\CATA}$, we have to set $\functor{M}_0(A_{\CATA}):=
\overline{\functor{F}}_0(A_{\CATA})=\functor{F}_0(A_{\CATA})$. Given any $1$-morphism

\[\Big(A'_{\CATA},\operatorname{w}_{\CATA},f_{\CATA}\Big):A_{\CATA}\longrightarrow B_{\CATA}\]
in $\CATA\left[\SETWAinv\right]$, following~\cite[pag.~265]{Pr} we have to define

\begin{equation}\label{eq-81}
\functor{M}_1\Big(A'_{\CATA},\operatorname{w}_{\CATA},f_{\CATA}\Big):=
\overline{\functor{F}}_1(f_{\CATA})\circ\functor{P}(\operatorname{w}_{\CATA}).
\end{equation}

By definition of $\functor{U}_{\SETWB}$, we have:

\[
\begin{tikzpicture}[xscale=3.0,yscale=-1.2]
    \node (A0_0) at (-0.8, 0) {$\overline{\functor{F}}_1(f_{\CATA})=
      \functor{U}_{\SETWB,1}\circ\functor{F}_1(f_{\CATA})=\Big(\functor{F}_0(A'_{\CATA})$};
    \node (A0_1) at (1, 0) {$\functor{F}_0(A'_{\CATA})$};
    \node (A0_2) at (2, 0) {$\functor{F}_0(B_{\CATA})\Big)$.};
    \path (A0_1) edge [->]node [auto,swap]
      {$\scriptstyle{\id_{\functor{F}_0(A'_{\CATA})}}$} (A0_0);
    \path (A0_1) edge [->]node [auto] {$\scriptstyle{\functor{F}_1(f_{\CATA})}$} (A0_2);
\end{tikzpicture}
\]

Therefore, by condition (\hyperref[C2]{C2}) and~\cite[\S~2.2]{Pr}, we have

\[
\begin{tikzpicture}[xscale=3.2,yscale=-1.2]
    \node (A0_0) at (-0.7, 0) {$\functor{M}_1\Big(A'_{\CATA},
      \operatorname{w}_{\CATA},f_{\CATA}\Big)=\Big(\functor{F}_0(A_{\CATA})$};
    \node (A0_1) at (1, 0) {$\functor{F}_0(A'_{\CATA})$};
    \node (A0_2) at (2, 0) {$\functor{F}_0(B_{\CATA})\Big)$.};
    \path (A0_1) edge [->]node [auto,swap]
      {$\scriptstyle{\functor{F}_1(\operatorname{w}_{\CATA})}$} (A0_0);
    \path (A0_1) edge [->]node [auto] {$\scriptstyle{\functor{F}_1(f_{\CATA})}$} (A0_2);
\end{tikzpicture}
\]

Now let us fix any pair of morphisms $(A^m_{\CATA},\operatorname{w}^m_{\CATA},f^m_{\CATA}):A_{\CATA}
\rightarrow B_{\CATA}$ for $m=1,2$ in $\CATA\left[\SETWAinv\right]$ and any $2$-morphism in $\CATA
\left[\SETWAinv\right]$ as in \eqref{eq-75}.
Then we recall that the image of such a $2$-morphism via
$\functor{M}_2$ is obtained as the vertical composition of a long series of $2$-morphisms
of $\CATB\left[\SETWBinv\right]$ as listed in~\cite[pag.~266]{Pr}. In the case under
exam (with the already mentioned assumptions on $\CATA,\CATB$ and $\functor{F}$), we have:

\begin{equation}\label{eq-33}
\functor{M}_2\Big(\Big[A^3_{\CATA},\operatorname{v}^1_{\CATA},
\operatorname{v}^2_{\CATA},\alpha_{\CATA},\beta_{\CATA}\Big]\Big)=\Gamma^1\odot\cdots\odot
\Gamma^{12},
\end{equation}
where:

\begin{equation}\label{eq-02}
\begin{tikzpicture}[xscale=3.05,yscale=-1.5]
    \node (B0_0) at (0, 6) {$\functor{F}_0(A_{\CATA})$};
    \node (C0_0) at (4, 6) {$\functor{F}_0(B_{\CATA})$};
   
    \node at (2, 0.5) {$\Downarrow\,\Gamma^{12}:=i_{\overline{\functor{F}}_1(f^1_{\CATA})}
      \ast\Big(i_{\functor{P}(\operatorname{w}^1_{\CATA})}\ast\Xi(\operatorname{w}^1_{\CATA}
      \circ\operatorname{v}^1_{\CATA})\Big)$};
    \node at (2, 1.5) {$\Downarrow\,\Gamma^{11}:=i_{\overline{\functor{F}}_1(f^1_{\CATA})}
      \ast\Big(i_{\functor{P}(\operatorname{w}^1_{\CATA})}\ast\Theta^{-1}_{\overline{\functor{F}}_1
      (\operatorname{w}^1_{\CATA}),\overline{\functor{F}}_1(\operatorname{v}^1_{\CATA}),
      \functor{P}(\operatorname{w}^1_{\CATA}\circ\operatorname{v}^1_{\CATA})}\Big)$};
    \node at (2, 2.5) {$\Downarrow\,\Gamma^{10}:=i_{\overline{\functor{F}}_1(f^1_{\CATA})}\ast
      \Thetaa{\functor{P}(\operatorname{w}^1_{\CATA})}{\overline{\functor{F}}_1
      (\operatorname{w}^1_{\CATA})}{\overline{\functor{F}}_1(\operatorname{v}^1_{\CATA})\circ
      \functor{P}(\operatorname{w}^1_{\CATA}\circ\operatorname{v}^1_{\CATA})}$};
    \node at (2, 3.5) {$\Downarrow\,\Gamma^9:=i_{\overline{\functor{F}}_1(f^1_{\CATA})}\ast
      \Big(\Delta(\operatorname{w}^1_{\CATA})\ast i_{\overline{\functor{F}}_1
      (\operatorname{v}^1_{\CATA})\circ \functor{P}(\operatorname{w}^1_{\CATA}
      \circ\operatorname{v}^1_{\CATA})}\Big)$};
    \node at (2, 4.5) {$\Downarrow\,\Gamma^8:=\Thetaa{\overline{\functor{F}}_1(f^1_{\CATA})}
      {\overline{\functor{F}}_1(\operatorname{v}^1_{\CATA})}{\functor{P}(\operatorname{w}^1_{\CATA}
      \circ\operatorname{v}^1_{\CATA})}$};
    \node at (2, 5.5) {$\Downarrow\,\Gamma^7:=\overline{\functor{F}}_2(\beta_{\CATA})\ast
      i_{\functor{P}(\operatorname{w}^1_{\CATA}\circ\operatorname{v}^1_{\CATA})}$};
    \node at (2, 6.5) {$\Downarrow\,\Gamma^6:=\Big(i_{\overline{\functor{F}}_1(f^2_{\CATA})}
      \ast\Big(\left(\Delta(\operatorname{w}^2_{\CATA})\right)^{-1}\ast i_{\overline{\functor{F}}_1
      (\operatorname{v}^2_{\CATA})}\Big)\Big)\ast i_{\functor{P}(\operatorname{w}^1_{\CATA}
      \circ\operatorname{v}^1_{\CATA})}$};
    \node at (2, 7.5) {$\Downarrow\,\Gamma^5:=\Big(i_{\overline{\functor{F}}_1(f^2_{\CATA})}
      \ast\Thetab{\functor{P}(\operatorname{w}^2_{\CATA})}{\overline{\functor{F}}_1
      (\operatorname{w}^2_{\CATA})}{\overline{\functor{F}}_1(\operatorname{v}^2_{\CATA})}\Big)
      \ast i_{\functor{P}(\operatorname{w}^1_{\CATA}\circ\operatorname{v}^1_{\CATA})}$};
    \node at (2, 8.5) {$\Downarrow\,\Gamma^4:=\Thetaa{\overline{\functor{F}}_1(f^2_{\CATA})}
      {\functor{P}(\operatorname{w}^2_{\CATA})}{\overline{\functor{F}}_1(\operatorname{w}^2_{\CATA}
      \circ\operatorname{v}^2_{\CATA})}\ast i_{\functor{P}(\operatorname{w}^1_{\CATA}\circ
      \operatorname{v}^1_{\CATA})}$};
    \node at (2, 9.5) {$\Downarrow\,\Gamma^3:=\Big(i_{\overline{\functor{F}}_1(f^2_{\CATA})\circ
      \functor{P}(\operatorname{w}^2_{\CATA})}\ast\overline{\functor{F}}_2
      \left(\alpha_{\CATA}^{-1}\right)\Big)\ast i_{\functor{P}
      (\operatorname{w}^1_{\CATA}\circ\operatorname{v}^1_{\CATA})}$};
    \node at (2, 10.5) {$\Downarrow\,\Gamma^2:=\Thetab{\overline{\functor{F}}_1(f^2_{\CATA})
      \circ \functor{P}(\operatorname{w}^2_{\CATA})}{\overline{\functor{F}}_1
      (\operatorname{w}^1_{\CATA}\circ\operatorname{v}^1_{\CATA})}
      {\functor{P}(\operatorname{w}^1_{\CATA}\circ\operatorname{v}^1_{\CATA})}$};
    \node at (2, 11.5) {$\Downarrow\,\Gamma^1:=i_{\overline{\functor{F}}_1(f^2_{\CATA})\circ
      \functor{P}(\operatorname{w}^2_{\CATA})}\ast\left(\Xi(\operatorname{w}^1_{\CATA}
      \circ\operatorname{v}^1_{\CATA})\right)^{-1}$};
    
    \foreach \i in {0,...,12} {\draw[rounded corners,->] (B0_0) to (0.5,\i) to (3.5,\i) to (C0_0);}
    
    \node[color=black,draw=white,fill=white,shape=rectangle,rounded corners=0.5ex,text centered]
      at (2, 0) {$\scriptstyle{f^{12}:=\overline{\functor{F}}_1(f^1_{\CATA})\circ
      \functor{P}(\operatorname{w}^1_{\CATA})}$};
    \node[color=black,draw=white,fill=white,shape=rectangle,rounded corners=0.5ex,text centered]
      at (2, 1) {$\scriptstyle{f^{11}:=\overline{\functor{F}}_1(f^1_{\CATA})\circ\Big(
      \functor{P}(\operatorname{w}^1_{\CATA})\circ\Big(\overline{\functor{F}}_1
      (\operatorname{w}^1_{\CATA}\circ\operatorname{v}^1_{\CATA})\circ
      \functor{P}(\operatorname{w}^1_{\CATA}\circ\operatorname{v}^1_{\CATA})\Big)\Big)}$};
    \node[color=black,draw=white,fill=white,shape=rectangle,rounded corners=0.5ex,text centered]
      at (2, 2) {$\scriptstyle{f^{10}:=\overline{\functor{F}}_1(f^1_{\CATA})\circ\Big(
      \functor{P}(\operatorname{w}^1_{\CATA})\circ\Big(\overline{\functor{F}}_1
      (\operatorname{w}^1_{\CATA})\circ\Big(\overline{\functor{F}}_1(\operatorname{v}^1_{\CATA})
      \circ \functor{P}(\operatorname{w}^1_{\CATA}\circ\operatorname{v}^1_{\CATA})\Big)\Big)\Big)}$};
    \node[color=black,draw=white,fill=white,shape=rectangle,rounded corners=0.5ex,text centered]
      at (2, 3) {$\scriptstyle{f^9:=\overline{\functor{F}}_1(f^1_{\CATA})\circ\Big(\Big(
      \functor{P}(\operatorname{w}^1_{\CATA})\circ\overline{\functor{F}}_1
      (\operatorname{w}^1_{\CATA})\Big)\circ\Big(\overline{\functor{F}}_1(\operatorname{v}^1_{\CATA})
      \circ \functor{P}(\operatorname{w}^1_{\CATA}\circ\operatorname{v}^1_{\CATA})\Big)\Big)}$};
    \node[color=black,draw=white,fill=white,shape=rectangle,rounded corners=0.5ex,text centered]
      at (2, 4) {$\scriptstyle{f^8:=\overline{\functor{F}}_1(f^1_{\CATA})\circ\Big(
      \overline{\functor{F}}_1(\operatorname{v}^1_{\CATA})\circ
      \functor{P}(\operatorname{w}^1_{\CATA}\circ\operatorname{v}^1_{\CATA})\Big)}$};
    \node[color=black,draw=white,fill=white,shape=rectangle,rounded corners=0.5ex,text centered]
      at (2, 5) {$\scriptstyle{f^7:=\overline{\functor{F}}_1(f^1_{\CATA}\circ
      \operatorname{v}^1_{\CATA})\circ\functor{P}(\operatorname{w}^1_{\CATA}\circ
      \operatorname{v}^1_{\CATA})}$};
    \node[color=black,draw=white,fill=white,shape=rectangle,rounded corners=0.5ex,text centered]
      at (2, 6) {$\scriptstyle{f^6:=\overline{\functor{F}}_1(f^2_{\CATA}\circ
      \operatorname{v}^2_{\CATA})\circ\functor{P}(\operatorname{w}^1_{\CATA}\circ
      \operatorname{v}^1_{\CATA})}$};
    \node[color=black,draw=white,fill=white,shape=rectangle,rounded corners=0.5ex,text centered]
      at (2, 7) {$\scriptstyle{f^5:=\Big(\overline{\functor{F}}_1(f^2_{\CATA})\circ\Big(\Big(
      \functor{P}(\operatorname{w}^2_{\CATA})\circ\overline{\functor{F}}_1
      (\operatorname{w}^2_{\CATA})\Big)\circ\overline{\functor{F}}_1(\operatorname{v}^2_{\CATA})\Big)
      \Big)\circ\functor{P}(\operatorname{w}^1_{\CATA}\circ\operatorname{v}^1_{\CATA})}$};
    \node[color=black,draw=white,fill=white,shape=rectangle,rounded corners=0.5ex,text centered]
      at (2, 8) {$\scriptstyle{f^4:=\Big(\overline{\functor{F}}_1(f^2_{\CATA})\circ\Big(
      \functor{P}(\operatorname{w}^2_{\CATA})\circ\overline{\functor{F}}_1
      (\operatorname{w}^2_{\CATA}\circ\operatorname{v}^2_{\CATA})\Big)\Big)\circ
      \functor{P}(\operatorname{w}^1_{\CATA}\circ\operatorname{v}^1_{\CATA})}$};
    \node[color=black,draw=white,fill=white,shape=rectangle,rounded corners=0.5ex,text centered]
      at (2, 9) {$\scriptstyle{f^3:=\Big(\Big(\overline{\functor{F}}_1(f^2_{\CATA})\circ
      \functor{P}(\operatorname{w}^2_{\CATA})\Big)\circ\overline{\functor{F}}_1
      (\operatorname{w}^2_{\CATA}\circ\operatorname{v}^2_{\CATA})\Big)\circ
      \functor{P}(\operatorname{w}^1_{\CATA}\circ\operatorname{v}^1_{\CATA})}$};
    \node[color=black,draw=white,fill=white,shape=rectangle,rounded corners=0.5ex,text centered]
      at (2, 10) {$\scriptstyle{f^2:=\Big(\Big(\overline{\functor{F}}_1(f^2_{\CATA})\circ
      \functor{P}(\operatorname{w}^2_{\CATA})\Big)\circ\overline{\functor{F}}_1
      (\operatorname{w}^1_{\CATA}\circ\operatorname{v}^1_{\CATA})\Big)\circ
      \functor{P}(\operatorname{w}^1_{\CATA}\circ\operatorname{v}^1_{\CATA})}$};
    \node[color=black,draw=white,fill=white,shape=rectangle,rounded corners=0.5ex,text centered]
      at (2, 11) {$\scriptstyle{f^1:=\Big(\overline{\functor{F}}_1(f^2_{\CATA})\circ
      \functor{P}(\operatorname{w}^2_{\CATA})\Big)\circ\Big(\overline{\functor{F}}_1
      (\operatorname{w}^1_{\CATA}\circ\operatorname{v}^1_{\CATA})\circ
      \functor{P}(\operatorname{w}^1_{\CATA}\circ\operatorname{v}^1_{\CATA})\Big)}$};
    \node[color=black,draw=white,fill=white,shape=rectangle,rounded corners=0.5ex,text centered]
      at (2, 12) {$\scriptstyle{f^0:=\overline{\functor{F}}_1(f^2_{\CATA})\circ
      \functor{P}(\operatorname{w}^2_{\CATA})}$};
\end{tikzpicture}
\end{equation}
(here and in the following lines, for simplicity we denote by $\Theta_{\bullet}$ the associators
$\Theta_{\bullet}^{\CATB,\SETWB}$ for $\CATB\left[\SETWBinv\right]$).
In the next pages we are going to compute all the morphisms and $2$-morphisms of \eqref{eq-02}.
In order to do that, for each $m=1,2$ let us consider the following pair of morphisms

\[
\begin{tikzpicture}[xscale=3.7,yscale=-1.2]
    \node (A1_0) at (0, 2) {$\functor{F}_0(A^3_{\CATA})$};
    \node (A1_2) at (2, 2) {$\functor{F}_0(A^m_{\CATA})$};
    \node (A2_1) at (1, 2) {$\functor{F}_0(A_{\CATA})$};
    \path (A1_2) edge [->]node [auto,swap]
      {$\scriptstyle{\functor{F}_1(\operatorname{w}^m_{\CATA})}$} (A2_1);
    \path (A1_0) edge [->]node [auto] {$\scriptstyle{\functor{F}_1
      (\operatorname{w}^m_{\CATA})\circ\functor{F}_1(\operatorname{v}^m_{\CATA})}$} (A2_1);
\end{tikzpicture}
\]
and let us suppose that the fixed choice \hyperref[C]{C}$(\SETWB)$ for such a pair is given by
the data in the upper part of the following diagram

\begin{equation}\label{eq-52}
\begin{tikzpicture}[xscale=3.9,yscale=-0.8]
    \node (A0_1) at (1, 0) {$A^m_{\CATB}$};
    \node (A1_2) at (2, 2) {$\functor{F}_0(A^m_{\CATA})$,};
    \node (A2_1) at (1, 2) {$\functor{F}_0(A_{\CATA})$};
    \node (A1_0) at (0, 2) {$\functor{F}_0(A^3_{\CATA})$};
    
    \node (B1_1) at (2.5, 0.6) {$m=1,2$};
    \node (A1_1) at (1, 0.95) {$\sigma^m_{\CATB}$};
    \node (B1_1) at (1, 1.4) {$\Rightarrow$};
    
    \path (A1_2) edge [->]node [auto]
      {$\scriptstyle{\functor{F}_1(\operatorname{w}^m_{\CATA})}$} (A2_1);
    \path (A0_1) edge [->]node [auto] {$\scriptstyle{\operatorname{z}^m_{\CATB}}$} (A1_2);
    \path (A1_0) edge [->]node [auto,swap] {$\scriptstyle{\functor{F}_1
      (\operatorname{w}^m_{\CATA})\circ\functor{F}_1(\operatorname{v}^m_{\CATA})}$} (A2_1);
    \path (A0_1) edge [->]node [auto,swap] {$\scriptstyle{\operatorname{u}^m_{\CATB}}$} (A1_0);
\end{tikzpicture}
\end{equation} 
with $\operatorname{u}^m_{\CATB}$ in $\SETWB$ and $\sigma^m_{\CATB}$ invertible (we recall
that $\functor{F}_1(\SETWA)\subseteq\SETWB$ by hypothesis, so it makes sense to consider the
choice \hyperref[C]{C}$(\SETWB)$ for the pair above). Moreover, let us 
suppose that the fixed choice \hyperref[C]{C}$(\SETWB)$ for the pair

\[
\begin{tikzpicture}[xscale=3.7,yscale=-1.2]
    \node (A1_0) at (0, 2) {$\functor{F}_0(A^3_{\CATA})$};
    \node (A1_2) at (2, 2) {$\functor{F}_0(A^2_{\CATA})$};
    \node (A2_1) at (1, 2) {$\functor{F}_0(A_{\CATA})$};
    \path (A1_2) edge [->]node [auto,swap]
      {$\scriptstyle{\functor{F}_1(\operatorname{w}^2_{\CATA})}$} (A2_1);
    \path (A1_0) edge [->]node [auto] {$\scriptstyle{\functor{F}_1
      (\operatorname{w}^1_{\CATA})\circ\functor{F}_1(\operatorname{v}^1_{\CATA})}$} (A2_1);
\end{tikzpicture}
\]
is given by the data in the upper part of the following diagram

\begin{equation}\label{eq-51}
\begin{tikzpicture}[xscale=3.7,yscale=-0.8]
    \node (A0_1) at (1, 0) {$A^3_{\CATB}$};
    \node (A1_0) at (0, 2) {$\functor{F}_0(A^3_{\CATA})$};
    \node (A1_2) at (2, 2) {$\functor{F}_0(A^2_{\CATA})$,};
    \node (A2_1) at (1, 2) {$\functor{F}_0(A_{\CATA})$};
    
    \node (A1_1) at (1, 0.95) {$\sigma^3_{\CATB}$};
    \node (B1_1) at (1, 1.4) {$\Rightarrow$};
    
    \path (A1_2) edge [->]node [auto]
      {$\scriptstyle{\functor{F}_1(\operatorname{w}^2_{\CATA})}$} (A2_1);
    \path (A0_1) edge [->]node [auto] {$\scriptstyle{\operatorname{z}^3_{\CATB}}$} (A1_2);
    \path (A1_0) edge [->]node [auto,swap] {$\scriptstyle{\functor{F}_1
      (\operatorname{w}^1_{\CATA})\circ\functor{F}_1(\operatorname{v}^1_{\CATA})}$} (A2_1);
    \path (A0_1) edge [->]node [auto,swap]
      {$\scriptstyle{\operatorname{u}^3_{\CATB}}$} (A1_0);
\end{tikzpicture}
\end{equation}
with $\operatorname{u}^3_{\CATB}$ in $\SETWB$ and $\sigma^3_{\CATB}$ invertible. According to
the definition of composition of $1$-morphism in a bicategory of fractions (see~\cite[pag.~256]{Pr}),
this set of choices completely determines the morphisms $f^0,\cdots,f^{12}$ as follows:

\begin{gather*}
f^{12}=\Big(\functor{F}_0(A^1_{\CATA}),\functor{F}_1(\operatorname{w}^1_{\CATA}),\functor{F}_1
 (f^1_{\CATA})\Big),
\end{gather*}

\begin{gather*}
f^{11}=\Big(\functor{F}_0(A^1_{\CATA}),\id_{\functor{F}_0(A^1_{\CATA})},\functor{F}_1(f^1_{\CATA})
 \Big)\circ \\
\circ\Big[\Big(\functor{F}_0(A^1_{\CATA}),\functor{F}_1(\operatorname{w}^1_{\CATA}),
 \id_{\functor{F}_0(A^1_{\CATA})}\Big)\circ\Big(\functor{F}_0(A^3_{\CATA}),\functor{F}_1
 (\operatorname{w}^1_{\CATA}\circ\operatorname{v}^1_{\CATA}),\functor{F}_1(\operatorname{w}^1_{\CATA}
 \circ\operatorname{v}^1_{\CATA})\Big)\Big]\stackrel{\eqref{eq-52},m=1}{=} \\
\stackrel{\eqref{eq-52},m=1}{=}\Big(\functor{F}_0(A^1_{\CATA}),\id_{\functor{F}_0(A^1_{\CATA})},
 \functor{F}_1(f^1_{\CATA})\Big)\circ\Big(A^1_{\CATB},\functor{F}_1(\operatorname{w}^1_{\CATA}\circ
 \operatorname{v}^1_{\CATA})\circ\operatorname{u}^1_{\CATB},\operatorname{z}^1_{\CATB}\Big)= \\
=\Big(A^1_{\CATB},\functor{F}_1(\operatorname{w}^1_{\CATA}\circ\operatorname{v}^1_{\CATA})\circ
 \operatorname{u}^1_{\CATB},\functor{F}_1(f^1_{\CATA})\circ\operatorname{z}^1_{\CATB}\Big)=f^{10},
\end{gather*}

\begin{gather*}
f^9\stackrel{(\operatorname{C}3)}{=}\Big(\functor{F}_0(A^1_{\CATA}),\id_{\functor{F}_0
 (A^1_{\CATA})},\functor{F}_1(f^1_{\CATA})\Big)\circ\Big(\functor{F}_0(A^3_{\CATA}),\functor{F}_1
 (\operatorname{w}^1_{\CATA}\circ\operatorname{v}^1_{\CATA}),\functor{F}_1
 (\operatorname{v}^1_{\CATA})\Big)= \\
=\Big(\functor{F}_0(A^3_{\CATA}),\functor{F}_1(\operatorname{w}^1_{\CATA}\circ
 \operatorname{v}^1_{\CATA}),\functor{F}_1(f^1_{\CATA}\circ\operatorname{v}^1_{\CATA})\Big)=f^8=f^7,
\end{gather*}

\begin{gather*}
f^6=\Big(\functor{F}_0(A^3_{\CATA}),\functor{F}_1
 (\operatorname{w}^1_{\CATA}\circ\operatorname{v}^1_{\CATA}),\functor{F}_1(f^2_{\CATA}\circ
 \operatorname{v}^2_{\CATA})\Big)\stackrel{(\operatorname{C}3)}{=}f^5,
\end{gather*}

\begin{gather*}
f^4\stackrel{\eqref{eq-52},m=2}{=}\Big[\Big(\functor{F}_0(A^2_{\CATA}),\id_{\functor{F}_0
 (A^2_{\CATA})},\functor{F}_1(f^2_{\CATA})\Big)\circ\Big(A^2_{\CATB},\operatorname{u}^2_{\CATB},
 \operatorname{z}^2_{\CATB}\Big)\Big]\circ \\
\circ\Big(\functor{F}_0(A^3_{\CATA}),\functor{F}_1(\operatorname{w}^1_{\CATA}\circ
 \operatorname{v}^1_{\CATA}),\id_{\functor{F}_0(A^3_{\CATA})}\Big)= \\
=\Big(A^2_{\CATB},\functor{F}_1(\operatorname{w}^1_{\CATA}\circ\operatorname{v}^1_{\CATA})\circ
 \operatorname{u}^2_{\CATB},\functor{F}_1(f^2_{\CATA})\circ\operatorname{z}^2_{\CATB}\Big)=f^3,
\end{gather*}

\[f^2\stackrel{\eqref{eq-51}}{=}\Big(A^3_{\CATB},\functor{F}_1(\operatorname{w}^1_{\CATA}\circ
\operatorname{v}^1_{\CATA})\circ\operatorname{u}^3_{\CATB},\functor{F}_1(f^2_{\CATA})\circ
\operatorname{z}^3_{\CATB}\Big)\stackrel{\eqref{eq-51}}{=}f^1,\]

\[f^0=\Big(\functor{F}_0(A^2_{\CATA}),\functor{F}_1(\operatorname{w}^2_{\CATA}),\functor{F}_1
(f^2_{\CATA})\Big).\]

Then we need to compute all the $2$-morphisms $\Gamma^1,\cdots,\Gamma^{12}$. In order to do that,
we will use all the descriptions of associators, vertical and horizontal compositions that we gave
in~\cite{T3}, applied to the case when the pair $(\CATC,\SETW)$ is given by $(\CATB,\SETWB)$.\\

First of all, we want to compute the associators $\Theta_{\bullet}$ appearing in $\Gamma^{11}$,
$\Gamma^8$, $\Gamma^4$ and $\Gamma^2$ (for the associators appearing in $\Gamma^{10}$ and
$\Gamma^5$, see below). For each of them we can apply~\cite[Corollary~2.2 and Remark~2.3]{T3}, so we
get that all such associators are simply $2$-identities. This implies at once that

\begin{equation}\label{eq-03}
\Gamma^{11}=i_{f^{11}}=i_{f^{10}},\quad\Gamma^8=i_{f^8}=i_{f^7},\quad\Gamma^4=i_{f^4}=i_{f^3}\quad
\textrm{and}\quad\Gamma^2=i_{f^2}=i_{f^1}.
\end{equation}

Moreover, by construction the $2$-morphisms $\Delta(\operatorname{w}^1_{\CATA})$ and $\Delta
(\operatorname{w}^2_{\CATA})$ are $2$-identities, hence

\begin{equation}\label{eq-20}
\Gamma^9=i_{f^9}=i_{f^8}\quad\textrm{and}\quad\Gamma^6=i_{f^6}=i_{f^5},
\end{equation}
so we will simply omit all the $2$-morphisms of \eqref{eq-03} and \eqref{eq-20} in the following
lines. Now let us compute $\Gamma^{12}$. In order to do that, the first step is to compute the
$2$-morphism

\begin{equation}\label{eq-08}
i_{\functor{P}(\operatorname{w}^1_{\CATA})}\ast\Xi(\operatorname{w}^1_{\CATA}\circ
\operatorname{v}^1_{\CATA}):\,\functor{P}(\operatorname{w}^1_{\CATA})\Longrightarrow
\functor{P}(\operatorname{w}^1_{\CATA})\circ\left(
\overline{\functor{F}}_1(\operatorname{w}^1_{\CATA}\circ\operatorname{v}^1_{\CATA})\circ\functor{P}
(\operatorname{w}^1_{\CATA}\circ\operatorname{v}^1_{\CATA})\right).
\end{equation}
Using \eqref{eq-37} we have

\[
\begin{tikzpicture}[xscale=2.8,yscale=-1.2]
    \node (A0_0) at (-0.4, 0) {$\functor{P}(\operatorname{w}^1_{\CATA})=
      \Big(\functor{F}_0(A_{\CATA})$};
    \node (A0_1) at (1.05, 0) {$\functor{F}_0(A^1_{\CATA})$};
    \node (A0_2) at (2.2, 0) {$\functor{F}_0(A^1_{\CATA})\Big)$;};
    
    \path (A0_1) edge [->]node [auto,swap] {$\scriptstyle{\functor{F}_1
      (\operatorname{w}^1_{\CATA})}$} (A0_0);
    \path (A0_1) edge [->]node [auto] {$\scriptstyle{\id_{\functor{F}_0(A^1_{\CATA})}}$} (A0_2);
\end{tikzpicture}
\]
moreover by \eqref{eq-39}, $\Xi(\operatorname{w}^1_{\CATA}\circ\operatorname{v}^1_{\CATA})$ is
represented by the following diagram:

\begin{equation}\label{eq-28}
\begin{tikzpicture}[xscale=2.4,yscale=-1]
    \node (A0_2) at (2, 0) {$\functor{F}_0(A_{\CATA})$};
    \node (A2_0) at (0, 2) {$\functor{F}_0(A_{\CATA})$};
    \node (A2_2) at (2, 2) {$\functor{F}_0(A^3_{\CATA})$};
    \node (A2_4) at (4, 2) {$\functor{F}_0(A_{\CATA})$.};
    \node (A4_2) at (2, 4) {$\functor{F}_0(A^3_{\CATA})$};

    \node (A2_1) at (1, 2) {$\Downarrow\,i_{\functor{F}_1(\operatorname{w}^1_{\CATA}
      \circ\operatorname{v}^1_{\CATA})}$};
    \node (A2_3) at (3, 2) {$\Downarrow\,i_{\functor{F}_1(\operatorname{w}^1_{\CATA}
      \circ\operatorname{v}^1_{\CATA})}$};
    
    \path (A0_2) edge [->]node [auto,swap] {$\scriptstyle{\id_{\functor{F}_0(A_{\CATA})}}$} (A2_0);
    \path (A0_2) edge [->]node [auto] {$\scriptstyle{\id_{\functor{F}_0(A_{\CATA})}}$} (A2_4);
    \path (A4_2) edge [->]node [auto,swap] {$\scriptstyle{\functor{F}_1
      (\operatorname{w}^1_{\CATA}\circ\operatorname{v}^1_{\CATA})}$} (A2_4);
    \path (A2_2) edge [->]node [auto,swap] {$\scriptstyle{\functor{F}_1
      (\operatorname{w}^1_{\CATA}\circ\operatorname{v}^1_{\CATA})}$} (A0_2);
    \path (A2_2) edge [->]node [auto] {$\scriptstyle{\id_{\functor{F}_0(A^3_{\CATA})}}$} (A4_2);
    \path (A4_2) edge [->]node [auto] {$\scriptstyle{\functor{F}_1
      (\operatorname{w}^1_{\CATA}\circ\operatorname{v}^1_{\CATA})}$} (A2_0);
\end{tikzpicture}
\end{equation}

Therefore, using (\hyperref[C1]{C1}) and diagram \eqref{eq-52} for $m=1$, we get that \eqref{eq-08}
is defined between the following morphisms:

\[\Big(\functor{F}_0(A^1_{\CATA}),\functor{F}_1
(\operatorname{w}^1_{\CATA}),\id_{\functor{F}_0(A^1_{\CATA})}\Big)\Longrightarrow\Big(
A^1_{\CATB},\functor{F}_1(\operatorname{w}^1_{\CATA}\circ\operatorname{v}^1_{\CATA})\circ
\operatorname{u}^1_{\CATB},\operatorname{z}^1_{\CATB}\Big).\]

In order to compute \eqref{eq-08} we are going to use~\cite[Proposition~0.4]{T3} with

\begin{gather*}
\Gamma:=\Delta(\operatorname{w}^1_{\CATA}\circ\operatorname{v}^1_{\CATA}),\quad\quad\underline{f}^1:=
 \Big(\functor{F}_0(A_{\CATA}),\id_{\functor{F}_0(A_{\CATA})},\id_{\functor{F}_0(A_{\CATA})}\Big), \\
\underline{f}^2:=\Big(\functor{F}_0(A^3_{\CATA}),\functor{F}_1(\operatorname{w}^1_{\CATA}\circ
 \operatorname{v}^1_{\CATA}),\functor{F}_1(\operatorname{w}^1_{\CATA}\circ
 \operatorname{v}^1_{\CATA})\Big).
\end{gather*}

In this case, the $2$-commutative diagrams of~\cite[Proposition~0.4(0.14)]{T3} are given by

\[
\begin{tikzpicture}[xscale=3.9,yscale=-0.8]
    \node (A0_1) at (0, 2) {$A^1=\functor{F}_0(A_{\CATA})$};
    \node (A1_0) at (1, 0) {$A^{\prime 1}:=\functor{F}_0(A^1_{\CATA})$};
    \node (A1_2) at (1, 2) {$B=\functor{F}_0(A_{\CATA})$};
    \node (A2_1) at (2, 2) {$B'=\functor{F}_0(A^1_{\CATA})$};

    \node (A1_1) at (1, 0.95) {$\rho^1:=i_{\functor{F}_1(\operatorname{w}^1_{\CATA})}$};
    \node (B1_1) at (1, 1.4) {$\Rightarrow$};

    \path (A1_0) edge [->]node [auto,swap] {$\scriptstyle{\operatorname{u}^{\prime 1}:=
      \functor{F}_1(\operatorname{w}^1_{\CATA})}$} (A0_1);
    \path (A1_0) edge [->]node [auto] {$\scriptstyle{f^{\prime 1}:=
      \id_{\functor{F}_0(A^1_{\CATA})}}$} (A2_1);
    \path (A0_1) edge [->]node [auto,swap] {$\scriptstyle{f^1=
      \id_{\functor{F}_0(A_{\CATA})}}$} (A1_2);
    \path (A2_1) edge [->]node [auto] {$\scriptstyle{\operatorname{u}=
      \functor{F}_1(\operatorname{w}^1_{\CATA})}$} (A1_2);
\end{tikzpicture}
\]
(since choices \hyperref[C]{C}$(\SETWB)$ must satisfy condition (\hyperref[C1]{C1})) and by

\[
\begin{tikzpicture}[xscale=4.3,yscale=-0.8]
    \node (A0_1) at (0, 2) {$A^2=\functor{F}_0(A^3_{\CATA})$};
    \node (A1_0) at (1, 0) {$A^{\prime 2}:=A^1_{\CATB}$};
    \node (A1_2) at (1, 2) {$B=\functor{F}_0(A_{\CATA})$};
    \node (A2_1) at (2, 2) {$B'=\functor{F}_0(A^1_{\CATA})$};

    \node (A1_1) at (1, 0.95) {$\rho^2:=\sigma^1_{\CATB}$};
    \node (B1_1) at (1, 1.4) {$\Rightarrow$};

    \path (A1_0) edge [->]node [auto,swap] {$\scriptstyle{\operatorname{u}^{\prime 2}:=
      \operatorname{u}^1_{\CATB}}$} (A0_1);
    \path (A1_0) edge [->]node [auto] {$\scriptstyle{f^{\prime 2}:=
      \operatorname{z}^1_{\CATB}}$} (A2_1);
    \path (A0_1) edge [->]node [auto,swap] {$\scriptstyle{f^2=\functor{F}_1
      (\operatorname{w}^1_{\CATA}\circ\operatorname{v}^1_{\CATA})}$} (A1_2);
    \path (A2_1) edge [->]node [auto] {$\scriptstyle{\operatorname{u}=
      \functor{F}_1(\operatorname{w}^1_{\CATA})}$} (A1_2);
\end{tikzpicture}
\]
(since in \eqref{eq-52} for $m=1$ we assumed that this was the fixed choice
\hyperref[C]{C}$(\SETWB)$ for the pair $(\functor{F}_1(\operatorname{w}^1_{\CATA}\circ
\operatorname{v}^1_{\CATA}),\functor{F}_1(\operatorname{w}^1_{\CATA}))$). Then we need to fix a set
of choices as in (F8) -- (F10) in~\cite[Proposition~0.4]{T3}:

\begin{description}
 \item[(F8)] for $m=1$, we choose the data in the upper part of the following diagram:
 
  \[
  \begin{tikzpicture}[xscale=-4.5,yscale=-0.8]
    \node (A0_1) at (0, 2) {$A^{\prime 1}=\functor{F}_0(A^1_{\CATA})$;};
    \node (A1_0) at (1, 0) {$A^{\prime\prime 1}:=\functor{F}_0(A^3_{\CATA})$};
    \node (A1_2) at (1, 2) {$A^1=\functor{F}_0(A_{\CATA})$};
    \node (A2_1) at (2, 2) {$A^3=\functor{F}_0(A^3_{\CATA})$};

    \node (A1_1) at (1, 0.95) {$\eta^1:=i_{\functor{F}_1(\operatorname{w}^1_{\CATA}
      \circ\operatorname{v}^1_{\CATA})}$};
    \node (B1_1) at (1, 1.45) {$\Rightarrow$};

    \path (A1_0) edge [->]node [auto] {$\scriptstyle{\operatorname{v}^{\prime 1}:=
      \functor{F}_1(\operatorname{v}^1_{\CATA})}$} (A0_1);
    \path (A1_0) edge [->]node [auto,swap] {$\scriptstyle{\operatorname{u}^{\prime\prime 1}:=
      \id_{\functor{F}_0(A^3_{\CATA})}}$} (A2_1);
    \path (A0_1) edge [->]node [auto] {$\scriptstyle{\operatorname{u}^{\prime 1}=
      \functor{F}_1(\operatorname{w}^1_{\CATA})}$} (A1_2);
    \path (A2_1) edge [->]node [auto,swap] {$\scriptstyle{\operatorname{v}^1=
      \functor{F}_1(\operatorname{w}^1_{\CATA}\circ\operatorname{v}^1_{\CATA})}$} (A1_2);
  \end{tikzpicture}
  \]
  for $m=2$, we choose the data in the upper part of the following diagram:
 
  \[
  \begin{tikzpicture}[xscale=-4.5,yscale=-0.8]
    \node (A0_1) at (0, 2) {$A^{\prime 2}=A^1_{\CATB}$;};
    \node (A1_0) at (1, 0) {$A^{\prime\prime 2}:=A^1_{\CATB}$};
    \node (A1_2) at (1, 2) {$A^2=\functor{F}_0(A^3_{\CATA})$};
    \node (A2_1) at (2, 2) {$A^3=\functor{F}_0(A^3_{\CATA})$};

    \node (A1_1) at (1, 0.95) {$\eta^2:=i_{\operatorname{u}^1_{\CATB}}$};
    \node (B1_1) at (1, 1.4) {$\Rightarrow$};

    \path (A1_0) edge [->]node [auto] {$\scriptstyle{\operatorname{v}^{\prime 2}:=
      \id_{A^1_{\CATB}}}$} (A0_1);
    \path (A1_0) edge [->]node [auto,swap] {$\scriptstyle{\operatorname{u}^{\prime\prime 2}:=
      \operatorname{u}^1_{\CATB}}$} (A2_1);
    \path (A0_1) edge [->]node [auto] {$\scriptstyle{\operatorname{u}^{\prime 2}=
      \operatorname{u}^1_{\CATB}}$} (A1_2);
    \path (A2_1) edge [->]node [auto,swap] {$\scriptstyle{\operatorname{v}^2=
      \id_{\functor{F}_0(A^3_{\CATA})}}$} (A1_2);
  \end{tikzpicture}
  \]
  
 \item[(F9) and (F10)] we use axioms (BF4a) and (BF4b)
  for $(\CATB,\SETWB)$ and $\sigma^1_{\CATB}$ (see \eqref{eq-52} for $m=1$); so we get
  an object $\overline{A}^1_{\CATB}$, a morphism $\operatorname{t}^1_{\CATB}:
  \overline{A}^1_{\CATB}\rightarrow A^1_{\CATB}$ in $\SETWB$ and an invertible $2$-morphism
 
  \[\mu^1_{\CATB}:\,\functor{F}_1(\operatorname{v}^1_{\CATA})\circ\operatorname{u}^1_{\CATB}
  \circ\operatorname{t}^1_{\CATB}\Longrightarrow\operatorname{z}^1_{\CATB}\circ
  \operatorname{t}^1_{\CATB},\]
  such that
  
  \begin{equation}\label{eq-13}
  \sigma^1_{\CATB}\ast i_{\operatorname{t}^1_{\CATB}}=i_{\functor{F}_1(\operatorname{w}^1_{\CATA})}
  \ast\mu^1_{\CATB}
  \end{equation}
  (in general such data are not unique, we make any arbitrary choice as above). Then we choose the
  data of (F9) as the data of the upper part of the following diagram
  
  \[
  \begin{tikzpicture}[xscale=3.9,yscale=-1]
    \node (A0_1) at (0, 2) {$A^{\prime\prime 1}=\functor{F}_0(A^3_{\CATA})$};
    \node (A1_0) at (1, 0) {$A^{\prime\prime}:=\overline{A}^1_{\CATB}$};
    \node (A1_2) at (1, 2) {$A^3=\functor{F}_0(A^3_{\CATA})$};
    \node (A2_1) at (2, 2) {$A^{\prime\prime 2}=A^1_{\CATB}$};

    \node (A1_1) at (1, 0.95) {$\eta^3:=i_{\operatorname{u}^1_{\CATB}
      \circ\operatorname{t}^1_{\CATB}}$};
    \node (B1_1) at (1, 1.4) {$\Rightarrow$};

    \path (A1_0) edge [->]node [auto,swap] {$\scriptstyle{\operatorname{z}^1:=
      \operatorname{u}^1_{\CATB}\circ\operatorname{t}^1_{\CATB}}$} (A0_1);
    \path (A1_0) edge [->]node [auto] {$\scriptstyle{\operatorname{z}^2:=
      \operatorname{t}^1_{\CATB}}$} (A2_1);
    \path (A0_1) edge [->]node [auto,swap]
      {$\scriptstyle{\operatorname{u}^{\prime\prime 1}=\id_{\functor{F}_0(A^3_{\CATA})}}$} (A1_2);
    \path (A2_1) edge [->]node [auto] {$\scriptstyle{\operatorname{u}^{\prime\prime 2}=
     \operatorname{u}^1_{\CATB}}$} (A1_2);
  \end{tikzpicture}
  \]
  (in this way all the $2$-morphisms of~\cite[Proposition~0.4(0.15)]{T3} are trivial, except
  possibly for $\rho^2=\sigma^1_{\CATB}$); moreover, we choose the datum of (F10) as $\beta':=
  \mu^1_{\CATB}$; so the technical condition of~\cite[Proposition~0.4(F10)]{T3} is satisfied because
  of \eqref{eq-13}.
\end{description}

Since $\eta^1,\eta^2$ and $\eta^3$ are all $2$-identities, then using~\cite[Proposition~0.4]{T3},
we get immediately that \eqref{eq-08} is represented by the following diagram:

\begin{equation}\label{eq-42}
\begin{tikzpicture}[xscale=2.6,yscale=-0.8]
    \node (A0_2) at (2, 0) {$\functor{F}_0(A^1_{\CATA})$};
    \node (A2_0) at (0, 2) {$\functor{F}_0(A_{\CATA})$};
    \node (A2_2) at (2, 2) {$\overline{A}^1_{\CATB}$};
    \node (A2_4) at (4, 2) {$\functor{F}_0(A^1_{\CATA})$.};
    \node (A4_2) at (2, 4) {$A^1_{\CATB}$};

    \node (A2_1) at (1.1, 2) {$\Downarrow\,i_{\functor{F}_1(\operatorname{w}^1_{\CATA}
      \circ\operatorname{v}^1_{\CATA})\circ\operatorname{u}^1_{\CATB}
      \circ\operatorname{t}^1_{\CATB}}$};
    \node (A2_3) at (2.8, 2) {$\Downarrow\,\mu^1_{\CATB}$};
    
    \node (B1_1) at (2.25, 0.9) {$\scriptstyle{\functor{F}_1(\operatorname{v}^1_{\CATA})\circ}$};    
    \node (B2_2) at (2.28, 1.35) {$\scriptstyle{\circ\operatorname{u}^1_{\CATB}
      \circ\operatorname{t}^1_{\CATB}}$};

    \path (A2_2) edge [->]node [auto,swap] {} (A0_2);  
    \path (A4_2) edge [->]node [auto,swap] {$\scriptstyle{\operatorname{z}^1_{\CATB}}$} (A2_4);
    \path (A4_2) edge [->]node [auto] {$\scriptstyle{\functor{F}_1(\operatorname{w}^1_{\CATA}
      \circ\operatorname{v}^1_{\CATA})\circ\operatorname{u}^1_{\CATB}}$} (A2_0);
    \path (A0_2) edge [->]node [auto,swap] {$\scriptstyle{\functor{F}_1
      (\operatorname{w}^1_{\CATA})}$} (A2_0);
    \path (A0_2) edge [->]node [auto] {$\scriptstyle{\id_{\functor{F}_0(A^1_{\CATA})}}$} (A2_4);
    \path (A2_2) edge [->]node [auto] {$\scriptstyle{\operatorname{t}^1_{\CATB}}$} (A4_2);
\end{tikzpicture}
\end{equation}

Therefore, we get easily that $\Gamma^{12}=i_{\overline{\functor{F}}_1(f^1_{\CATA})}\ast
\eqref{eq-42}$ is represented by the following diagram:

\begin{equation}\label{eq-04}
\begin{tikzpicture}[xscale=2.6,yscale=-0.8]
    \node (A0_2) at (2, 0) {$\functor{F}_0(A^1_{\CATA})$};
    \node (A2_0) at (0, 2) {$\functor{F}_0(A_{\CATA})$};
    \node (A2_2) at (2, 2) {$\overline{A}^1_{\CATB}$};
    \node (A2_4) at (4, 2) {$\functor{F}_0(B_{\CATA})$.};
    \node (A4_2) at (2, 4) {$A^1_{\CATB}$};

    \node (A2_1) at (1.1, 2) {$\Downarrow\,i_{\functor{F}_1(\operatorname{w}^1_{\CATA}
      \circ\operatorname{v}^1_{\CATA})\circ\operatorname{u}^1_{\CATB}
      \circ\operatorname{t}^1_{\CATB}}$};
    \node (A2_3) at (2.8, 2.1) {$\Downarrow\,i_{\functor{F}_1(f^1_{\CATA})}\ast\mu^1_{\CATB}$};
    
    \node (B1_1) at (2.25, 0.9) {$\scriptstyle{\functor{F}_1(\operatorname{v}^1_{\CATA})\circ}$};    
    \node (B2_2) at (2.28, 1.35) {$\scriptstyle{\circ\operatorname{u}^1_{\CATB}
      \circ\operatorname{t}^1_{\CATB}}$};

    \path (A2_2) edge [->]node [auto,swap] {} (A0_2);  
    \path (A4_2) edge [->]node [auto,swap] {$\scriptstyle{\functor{F}_1(f^1_{\CATA})
      \circ\operatorname{z}^1_{\CATB}}$} (A2_4);
    \path (A4_2) edge [->]node [auto] {$\scriptstyle{\functor{F}_1(\operatorname{w}^1_{\CATA}
      \circ\operatorname{v}^1_{\CATA})\circ\operatorname{u}^1_{\CATB}}$} (A2_0);
    \path (A0_2) edge [->]node [auto,swap] {$\scriptstyle{\functor{F}_1
      (\operatorname{w}^1_{\CATA})}$} (A2_0);
    \path (A0_2) edge [->]node [auto] {$\scriptstyle{\functor{F}_1(f^1_{\CATA})}$} (A2_4);
    \path (A2_2) edge [->]node [auto] {$\scriptstyle{\operatorname{t}^1_{\CATB}}$} (A4_2);
\end{tikzpicture}
\end{equation}

In order to compute $\Gamma^{10}$, we have first of all to compute the associator

\begin{equation}\label{eq-14}
\Thetaa{\functor{P}(\operatorname{w}^1_{\CATA})}{\overline{\functor{F}}_1
(\operatorname{w}^1_{\CATA})}{\overline{\functor{F}}_1(\operatorname{v}^1_{\CATA})
\circ\functor{P}(\operatorname{w}^1_{\CATA}\circ\operatorname{v}^1_{\CATA})}.
\end{equation}

In order to do that, we are going to use~\cite[Proposition~0.1]{T3} for the triple of morphisms

\[
\begin{tikzpicture}[xscale=2.9,yscale=-1.2]
    \node (A0_0) at (-1, 0) {$\underline{f}:=\overline{\functor{F}}_1(\operatorname{v}^1_{\CATA})
      \circ\functor{P}(\operatorname{w}^1_{\CATA}\circ\operatorname{v}^1_{\CATA})
      =\Big(\functor{F}_0(A_{\CATA})$};
    \node (A0_1) at (0.95, 0) {$\functor{F}_0(A^3_{\CATA})$};
    \node (A0_2) at (2, 0) {$\functor{F}_0(A^1_{\CATA})\Big)$,};

    \path (A0_1) edge [->]node [auto,swap] {$\scriptstyle{\functor{F}_1
      (\operatorname{w}^1_{\CATA}\circ\operatorname{v}^1_{\CATA})}$} (A0_0);
    \path (A0_1) edge [->]node [auto] {$\scriptstyle{\functor{F}_1
      (\operatorname{v}^1_{\CATA})}$} (A0_2);
    \node (B0_0) at (-0.6, 1) {$\underline{g}:=\overline{\functor{F}}_1
      (\operatorname{w}^1_{\CATA})=\Big(\functor{F}_0(A^1_{\CATA})$};
    \node (B0_1) at (0.95, 1) {$\functor{F}_0(A^1_{\CATA})$};
    \node (B0_2) at (2, 1) {$\functor{F}_0(A_{\CATA})\Big)$,};

    \path (B0_1) edge [->]node [auto,swap] {$\scriptstyle{\id_{\functor{F}_0(A^1_{\CATA})}}$} (B0_0);
    \path (B0_1) edge [->]node [auto] {$\scriptstyle{\functor{F}_1
      (\operatorname{w}^1_{\CATA})}$} (B0_2);
   \node (C0_0) at (-0.6, 2) {$\underline{h}:=\functor{P}(\operatorname{w}^1_{\CATA})
      =\Big(\functor{F}_0(A_{\CATA})$};
    \node (C0_1) at (0.95, 2) {$\functor{F}_0(A^1_{\CATA})$};
    \node (C0_2) at (2, 2) {$\functor{F}_0(A^1_{\CATA})\Big)$.};
    
    \path (C0_1) edge [->]node [auto,swap] {$\scriptstyle{\functor{F}_1
      (\operatorname{w}^1_{\CATA})}$} (C0_0);
    \path (C0_1) edge [->]node [auto] {$\scriptstyle{\id_{\functor{F}_0(A^1_{\CATA})}}$} (C0_2);
\end{tikzpicture}
\]
 
In this case, the $4$ diagrams listed in~\cite[Proposition~0.1(0.4)]{T3} are given as follows;
the ones with $\delta$ and $\eta$ are a consequence of condition (\hyperref[C2]{C2}),
the one with $\xi$ is a consequence of (\hyperref[C3]{C3}), while the one with $\sigma$ is
a consequence of the fact that we have supposed that choices \hyperref[C]{C}$(\SETWB)$ give diagram
\eqref{eq-52} for $m=1$ when applied to the pair $(\functor{F}_1(\operatorname{w}^1_{\CATA}\circ
\operatorname{v}^1_{\CATA}),\functor{F}_1(\operatorname{w}^1_{\CATA}))$:

\[
\begin{tikzpicture}[xscale=2.5,yscale=-0.8]
    \node (A3_2) at (2, 4) {$\functor{F}_0(A^3_{\CATA})$};
    \node (A6_1) at (1, 6) {$\functor{F}_0(A^3_{\CATA})$};
    \node (A6_2) at (1.9, 6) {$\functor{F}_0(A^1_{\CATA})$};
    \node (A6_3) at (3, 6) {$\functor{F}_0(A^1_{\CATA})$,};
    
    \node (A5_2) at (2, 4.9) {$\Rightarrow$};
    \node (A4_2) at (2, 5.3) {$\delta:=i_{\functor{F}_1(\operatorname{v}^1_{\CATA})}$};

    \path (A3_2) edge [->]node [auto] {$\scriptstyle{\functor{F}_1(\operatorname{v}^1_{\CATA})}$} (A6_3);
    \path (A6_1) edge [->]node [auto,swap] {$\scriptstyle{\functor{F}_1(\operatorname{v}^1_{\CATA})}$} (A6_2);
    \path (A3_2) edge [->]node [auto,swap] {$\scriptstyle{\id_{\functor{F}_0(A^3_{\CATA})}}$} (A6_1);
    \path (A6_3) edge [->]node [auto] {$\scriptstyle{\id_{\functor{F}_0(A^1_{\CATA})}}$} (A6_2);
    
    \def \x {2.6}
    
    \node (B3_2) at (2+\x, 4) {$A^1_{\CATB}$};
    \node (B6_1) at (1+\x, 6) {$\functor{F}_0(A^3_{\CATA})$};
    \node (B6_2) at (2.1+\x, 6) {$\functor{F}_0(A_{\CATA})$};
    \node (B6_3) at (3+\x, 6) {$\functor{F}_0(A^1_{\CATA})$,};
    
    \node (B5_2) at (2+\x, 4.9) {$\Rightarrow$};
    \node (B4_2) at (2+\x, 5.3) {$\sigma:=\sigma^1_{\CATB}$};

    \path (B3_2) edge [->]node [auto] {$\scriptstyle{\operatorname{z}^1_{\CATB}}$} (B6_3);
    \path (B6_1) edge [->]node [auto,swap] {$\scriptstyle{\functor{F}_1
      (\operatorname{w}^1_{\CATA}\circ\operatorname{v}^1_{\CATA})}$} (B6_2);
    \path (B3_2) edge [->]node [auto,swap] {$\scriptstyle{\operatorname{u}^1_{\CATB}}$} (B6_1);
    \path (B6_3) edge [->]node [auto] {$\scriptstyle{\functor{F}_1
      (\operatorname{w}^1_{\CATA})}$} (B6_2);
\end{tikzpicture}
\]

\[
\begin{tikzpicture}[xscale=-2.5,yscale=-0.8]
    \def \x {2.6}
    
    \node (B3_2) at (2+\x, 4) {$\functor{F}_0(A^1_{\CATA})$};
    \node (B6_1) at (1+\x, 6) {$\functor{F}_0(A^1_{\CATA})$,};
    \node (B6_2) at (2+\x, 6) {$\functor{F}_0(A_{\CATA})$};
    \node (B6_3) at (3+\x, 6) {$\functor{F}_0(A^1_{\CATA})$};
    
    \node (B5_2) at (2+\x, 4.9) {$\Rightarrow$};
    \node (B4_2) at (2+\x, 5.3) {$\xi:=i_{\functor{F}_1(\operatorname{w}^1_{\CATA})}$};
    
    \path (B3_2) edge [->]node [auto,swap] {$\scriptstyle{\id_{\functor{F}_0(A^1_{\CATA})}}$} (B6_3);
    \path (B6_1) edge [->]node [auto] {$\scriptstyle{\functor{F}_1
      (\operatorname{w}^1_{\CATA})}$} (B6_2);
    \path (B3_2) edge [->]node [auto] {$\scriptstyle{\id_{\functor{F}_0(A^1_{\CATA})}}$} (B6_1);
    \path (B6_3) edge [->]node [auto,swap] {$\scriptstyle{\functor{F}_1
      (\operatorname{w}^1_{\CATA})}$} (B6_2);
    
    \node (A3_2) at (2, 4) {$\functor{F}_0(A^3_{\CATA})$};
    \node (A6_1) at (1, 6) {$\functor{F}_0(A^1_{\CATA})$.};
    \node (A6_2) at (2.1, 6) {$\functor{F}_0(A^1_{\CATA})$};
    \node (A6_3) at (3, 6) {$\functor{F}_0(A^3_{\CATA})$};
    
    \node (A5_2) at (2, 4.9) {$\Rightarrow$};
    \node (A4_2) at (2, 5.3) {$\eta:=i_{\functor{F}_1(\operatorname{v}^1_{\CATA})}$};
    
    \path (A3_2) edge [->]node [auto,swap] {$\scriptstyle{\id_{\functor{F}_0(A^3_{\CATA})}}$} (A6_3);
    \path (A6_1) edge [->]node [auto] {$\scriptstyle{\id_{\functor{F}_0(A^1_{\CATA})}}$} (A6_2);
    \path (A3_2) edge [->]node [auto] {$\scriptstyle{\functor{F}_1
      (\operatorname{v}^1_{\CATA})}$} (A6_1);
    \path (A6_3) edge [->]node [auto,swap] {$\scriptstyle{\functor{F}_1
      (\operatorname{v}^1_{\CATA})}$} (A6_2);
\end{tikzpicture}
\]

Then following~\cite[Proposition~0.1]{T3}, we choose data as follows:

\begin{description}
 \item[(F1)] we set $A^4:=\overline{A}^1_{\CATB}$, $\operatorname{u}^4:=\operatorname{t}^1_{\CATB}$,
  $\operatorname{u}^5:=\operatorname{u}^1_{\CATB}\circ\operatorname{t}^1_{\CATB}$ and $\gamma:=
  i_{\operatorname{u}^1_{\CATB}\circ\operatorname{t}^1_{\CATB}}$;
 \item[(F2)] we choose $\omega:=i_{\functor{F}_1(\operatorname{v}^1_{\CATA})\circ
  \operatorname{u}^1_{\CATB}\circ\operatorname{t}^1_{\CATB}}$;
 \item[(F3)] given the choices above, then the only possibly non-trivial $2$-morphism
  in~\cite[Proposition~0.1(0.8)]{T3} is $(\sigma^1_{\CATB})^{-1}\ast i_{\operatorname{t}^1_{\CATB}}$,
  so using \eqref{eq-13} we can choose $\rho:=(\mu^1_{\CATB})^{-1}$.
\end{description}

So by~\cite[Proposition~0.1]{T3} we conclude that the associator \eqref{eq-14} is represented by
the following diagram:

\begin{equation}\label{eq-21}
\begin{tikzpicture}[xscale=2.6,yscale=-0.8]
    \node (A0_2) at (2, 0) {$A^1_{\CATB}$};
    \node (A2_0) at (0, 2) {$\functor{F}_0(A_{\CATA})$};
    \node (A2_2) at (2, 2) {$\overline{A}^1_{\CATB}$};
    \node (A2_4) at (4, 2) {$\functor{F}_0(A^1_{\CATA})$.};
    \node (A4_2) at (2, 4) {$\functor{F}_0(A^3_{\CATA})$};

    \node (A2_1) at (1.1, 2) {$\Downarrow\,i_{\functor{F}_1(\operatorname{w}^1_{\CATA}
      \circ\operatorname{v}^1_{\CATA})\circ\operatorname{u}^1_{\CATB}
      \circ\operatorname{t}^1_{\CATB}}$};
    \node (A2_3) at (2.9, 2) {$\Downarrow\,(\mu^1_{\CATB})^{-1}$};
    
    \path (A0_2) edge [->]node [auto,swap] {$\scriptstyle{\functor{F}_1
      (\operatorname{w}^1_{\CATA}\circ\operatorname{v}^1_{\CATA})\circ
      \operatorname{u}^1_{\CATB}}$} (A2_0);
    \path (A0_2) edge [->]node [auto] {$\scriptstyle{\operatorname{z}^1_{\CATB}}$} (A2_4);
    \path (A4_2) edge [->]node [auto] {$\scriptstyle{\functor{F}_1(\operatorname{w}^1_{\CATA}
      \circ\operatorname{v}^1_{\CATA})}$} (A2_0);
    \path (A4_2) edge [->]node [auto,swap] {$\scriptstyle{\functor{F}_1
      (\operatorname{v}^1_{\CATA})}$} (A2_4);
    \path (A2_2) edge [->]node [auto,swap] {$\scriptstyle{\operatorname{t}^1_{\CATB}}$} (A0_2);
    \path (A2_2) edge [->]node [auto] {$\scriptstyle{\operatorname{u}^1_{\CATB}
      \circ\operatorname{t}^1_{\CATB}}$} (A4_2);
\end{tikzpicture}
\end{equation}

This implies easily that $\Gamma^{10}=i_{\overline{\functor{F}}_1(f^1_{\CATA})}\ast\eqref{eq-21}$
is represented by the following diagram:

\begin{equation}\label{eq-15}
\begin{tikzpicture}[xscale=2.6,yscale=-0.8]
    \node (A0_2) at (2, 0) {$A^1_{\CATB}$};
    \node (A2_0) at (0, 2) {$\functor{F}_0(A_{\CATA})$};
    \node (A2_2) at (2, 2) {$\overline{A}^1_{\CATB}$};
    \node (A2_4) at (4, 2) {$\functor{F}_0(B_{\CATA})$.};
    \node (A4_2) at (2, 4) {$\functor{F}_0(A^3_{\CATA})$};

    \node (A2_1) at (1.1, 2) {$\Downarrow\,i_{\functor{F}_1(\operatorname{w}^1_{\CATA}
      \circ\operatorname{v}^1_{\CATA})\circ\operatorname{u}^1_{\CATB}
      \circ\operatorname{t}^1_{\CATB}}$};
    \node (A2_3) at (2.9, 2) {$\Downarrow\,i_{\functor{F}_1(f^1_{\CATA})}\ast(\mu^1_{\CATB})^{-1}$};
    
    \path (A0_2) edge [->]node [auto,swap] {$\scriptstyle{\functor{F}_1
      (\operatorname{w}^1_{\CATA}\circ\operatorname{v}^1_{\CATA})\circ
      \operatorname{u}^1_{\CATB}}$} (A2_0);
    \path (A0_2) edge [->]node [auto] {$\scriptstyle{\functor{F}_1(f^1_{\CATA})
      \circ\operatorname{z}^1_{\CATB}}$} (A2_4);
    \path (A4_2) edge [->]node [auto] {$\scriptstyle{\functor{F}_1(\operatorname{w}^1_{\CATA}
      \circ\operatorname{v}^1_{\CATA})}$} (A2_0);
    \path (A4_2) edge [->]node [auto,swap] {$\scriptstyle{\functor{F}_1
      (f^1_{\CATA}\circ\operatorname{v}^1_{\CATA})}$} (A2_4);
    \path (A2_2) edge [->]node [auto,swap] {$\scriptstyle{\operatorname{t}^1_{\CATB}}$} (A0_2);
    \path (A2_2) edge [->]node [auto] {$\scriptstyle{\operatorname{u}^1_{\CATB}
      \circ\operatorname{t}^1_{\CATB}}$} (A4_2);
\end{tikzpicture}
\end{equation}

Moreover, it is easy to see that $\Gamma^7$ is represented by the following diagram:

\begin{equation}\label{eq-36}
\begin{tikzpicture}[xscale=2.6,yscale=-0.8]
    \node (A0_2) at (2, 0) {$\functor{F}_0(A^3_{\CATA})$};
    \node (A2_0) at (0, 2) {$\functor{F}_0(A_{\CATA})$};
    \node (A2_2) at (2, 2) {$\functor{F}_0(A^3_{\CATA})$};
    \node (A2_4) at (4, 2) {$\functor{F}_0(B_{\CATA})$.};
    \node (A4_2) at (2, 4) {$\functor{F}_0(A^3_{\CATA})$};

    \node (A2_1) at (1.1, 2) {$\Downarrow\,i_{\functor{F}_1(\operatorname{w}^1_{\CATA}
      \circ\operatorname{v}^1_{\CATA})}$};
    \node (A2_3) at (2.9, 2) {$\Downarrow\,\functor{F}_2(\beta_{\CATA})$};
    
    \path (A0_2) edge [->]node [auto,swap] {$\scriptstyle{\functor{F}_1
      (\operatorname{w}^1_{\CATA}\circ\operatorname{v}^1_{\CATA})}$} (A2_0);
    \path (A0_2) edge [->]node [auto] {$\scriptstyle{\functor{F}_1
      (f^1_{\CATA}\circ\operatorname{v}^1_{\CATA})}$} (A2_4);
    \path (A4_2) edge [->]node [auto] {$\scriptstyle{\functor{F}_1(\operatorname{w}^1_{\CATA}
      \circ\operatorname{v}^1_{\CATA})}$} (A2_0);
    \path (A4_2) edge [->]node [auto,swap] {$\scriptstyle{\functor{F}_1(f^2_{\CATA}
      \circ\operatorname{v}^2_{\CATA})}$} (A2_4);
    \path (A2_2) edge [->]node [auto,swap]
      {$\scriptstyle{\id_{\functor{F}_0(A^3_{\CATA})}}$} (A0_2);
    \path (A2_2) edge [->]node [auto]
      {$\scriptstyle{\id_{\functor{F}_0(A^3_{\CATA})}}$} (A4_2);
\end{tikzpicture}
\end{equation}

Therefore, using \eqref{eq-04}, \eqref{eq-15} and \eqref{eq-36} together
with~\cite[Proposition~0.2]{T3}, we get that $\Gamma^7\odot
\Gamma^{10}\odot\Gamma^{12}$ is represented by the following diagram

\begin{equation}\label{eq-41}
\begin{tikzpicture}[xscale=2.6,yscale=-0.8]
    \node (A0_2) at (2, 0) {$\functor{F}_0(A^1_{\CATA})$};
    \node (A2_0) at (0, 2) {$\functor{F}_0(A_{\CATA})$};
    \node (A2_2) at (2, 2) {$\functor{F}_0(A^3_{\CATA})$};
    \node (A2_4) at (4, 2) {$\functor{F}_0(B_{\CATA})$.};
    \node (A4_2) at (2, 4) {$\functor{F}_0(A^3_{\CATA})$};

    \node (A2_1) at (1.1, 2) {$\Downarrow\,i_{\functor{F}_1(\operatorname{w}^1_{\CATA}
      \circ\operatorname{v}^1_{\CATA})}$};
    \node (A2_3) at (2.9, 2) {$\Downarrow\,\functor{F}_2(\beta_{\CATA})$};
    
    \path (A0_2) edge [->]node [auto,swap] {$\scriptstyle{\functor{F}_1
      (\operatorname{w}^1_{\CATA})}$} (A2_0);
    \path (A0_2) edge [->]node [auto] {$\scriptstyle{\functor{F}_1(f^1_{\CATA})}$} (A2_4);
    \path (A4_2) edge [->]node [auto] {$\scriptstyle{\functor{F}_1(\operatorname{w}^1_{\CATA}
      \circ\operatorname{v}^1_{\CATA})}$} (A2_0);
    \path (A4_2) edge [->]node [auto,swap] {$\scriptstyle{\functor{F}_1
      (f^2_{\CATA}\circ\operatorname{v}^2_{\CATA})}$} (A2_4);
    \path (A2_2) edge [->]node [auto,swap]
      {$\scriptstyle{\functor{F}_1(\operatorname{v}^1_{\CATA})}$} (A0_2);
    \path (A2_2) edge [->]node [auto]
      {$\scriptstyle{\id_{\functor{F}_0(A^3_{\CATA})}}$} (A4_2);
\end{tikzpicture}
\end{equation}

Now we want to compute $\Gamma^5$. In order to do that, firstly we compute the associator

\begin{equation}\label{eq-19}
\Thetaa{\functor{P}(\operatorname{w}^2_{\CATA})}{\overline{\functor{F}}_1
(\operatorname{w}^2_{\CATA})}{\overline{\functor{F}}_1(\operatorname{v}^2_{\CATA})}.
\end{equation}
For that, we use~\cite[Proposition~0.1]{T3} on the triple of morphisms

\[
\begin{tikzpicture}[xscale=2.9,yscale=-1.2]
    \node (A0_0) at (-0.6, 0) {$\underline{f}:=\overline{\functor{F}}_1
      (\operatorname{v}^2_{\CATA})=\Big(\functor{F}_0(A^3_{\CATA})$};
    \node (A0_1) at (0.95, 0) {$\functor{F}_0(A^3_{\CATA})$};
    \node (A0_2) at (2, 0) {$\functor{F}_0(A^2_{\CATA})\Big)$,};

    \path (A0_1) edge [->]node [auto,swap]
      {$\scriptstyle{\id_{\functor{F}_0(A^3_{\CATA})}}$} (A0_0);
    \path (A0_1) edge [->]node [auto] {$\scriptstyle{\functor{F}_1
      (\operatorname{v}^2_{\CATA})}$} (A0_2);
    \node (B0_0) at (-0.6, 1) {$\underline{g}:=\overline{\functor{F}}_1
      (\operatorname{w}^2_{\CATA})=\Big(\functor{F}_0(A^2_{\CATA})$};
    \node (B0_1) at (0.95, 1) {$\functor{F}_0(A^2_{\CATA})$};
    \node (B0_2) at (2, 1) {$\functor{F}_0(A_{\CATA})\Big)$,};

    \path (B0_1) edge [->]node [auto,swap] {$\scriptstyle{\id_{\functor{F}_0(A^2_{\CATA})}}$} (B0_0);
    \path (B0_1) edge [->]node [auto] {$\scriptstyle{\functor{F}_1
      (\operatorname{w}^2_{\CATA})}$} (B0_2);
    \node (C0_0) at (-0.6, 2) {$\underline{h}:=\functor{P}(\operatorname{w}^2_{\CATA})
      =\Big(\functor{F}_0(A_{\CATA})$};
    \node (C0_1) at (0.95, 2) {$\functor{F}_0(A^2_{\CATA})$};
    \node (C0_2) at (2, 2) {$\functor{F}_0(A^2_{\CATA})\Big)$.};
    
    \path (C0_1) edge [->]node [auto,swap] {$\scriptstyle{\functor{F}_1
      (\operatorname{w}^2_{\CATA})}$} (C0_0);
    \path (C0_1) edge [->]node [auto] {$\scriptstyle{\id_{\functor{F}_0(A^2_{\CATA})}}$} (C0_2);
\end{tikzpicture}
\]

In order to do that, first of all we have to identify the $2$-morphisms
in~\cite[Proposition~0.1(0.4)]{T3}; using conditions (\hyperref[C1]{C1}) and (\hyperref[C2]{C2})
and \eqref{eq-52} for $m=2$, such $2$-morphisms are given as follows:

\[\delta:=i_{\functor{F}_1(\operatorname{v}^2_{\CATA})},\quad\sigma:=\sigma^2_{\CATB},\quad
\xi:=i_{\functor{F}_1(\operatorname{w}^2_{\CATA})},\quad\eta:=i_{\functor{F}_1
(\operatorname{v}^2_{\CATA})}.\]

Using (BF4a) and (BF4b)
for $(\CATB,\SETWB)$ and $\sigma^2_{\CATB}$ (see \eqref{eq-52} for $m=2$),
there are an object
$\overline{A}^2_{\CATB}$, a morphism $\operatorname{t}^2_{\CATB}:\overline{A}^2_{\CATB}\rightarrow
A^2_{\CATB}$ in $\SETWB$
and an invertible $2$-morphism $\mu^2_{\CATB}:\functor{F}_1(\operatorname{v}^2_{\CATA})
\circ\operatorname{u}^2_{\CATB}\circ\operatorname{t}^2_{\CATB}\Rightarrow\operatorname{z}^2_{\CATB}
\circ\operatorname{t}^2_{\CATB}$, such that

\begin{equation}\label{eq-16}
\sigma^2_{\CATB}\ast i_{\operatorname{t}^2_{\CATB}}=i_{\functor{F}_1(\operatorname{w}^2_{\CATA})}
\ast\mu^2_{\CATB}.
\end{equation}

Then we perform a series of computations analogous to those leading to \eqref{eq-21}, so we get
that the associator \eqref{eq-19} is represented by the following diagram:

\begin{equation}\label{eq-43}
\begin{tikzpicture}[xscale=2.6,yscale=-0.8]
    \node (A0_2) at (2, 0) {$A^2_{\CATB}$};
    \node (A2_0) at (0, 2) {$\functor{F}_0(A^3_{\CATA})$};
    \node (A2_2) at (2, 2) {$\overline{A}^2_{\CATB}$};
    \node (A2_4) at (4, 2) {$\functor{F}_0(A^2_{\CATA})$.};
    \node (A4_2) at (2, 4) {$\functor{F}_0(A^3_{\CATA})$};

    \node (A2_1) at (1.1, 2) {$\Downarrow\,i_{\operatorname{u}^2_{\CATB}
      \circ\operatorname{t}^2_{\CATB}}$};
    \node (A2_3) at (2.9, 2) {$\Downarrow\,(\mu^2_{\CATB})^{-1}$};
    
    \path (A0_2) edge [->]node [auto,swap] {$\scriptstyle{\operatorname{u}^2_{\CATB}}$} (A2_0);
    \path (A0_2) edge [->]node [auto] {$\scriptstyle{\operatorname{z}^2_{\CATB}}$} (A2_4);
    \path (A4_2) edge [->]node [auto] {$\scriptstyle{\id_{\functor{F}_0(A^3_{\CATA})}}$} (A2_0);
    \path (A4_2) edge [->]node [auto,swap] {$\scriptstyle{\functor{F}_1
      (\operatorname{v}^2_{\CATA})}$} (A2_4);
    \path (A2_2) edge [->]node [auto,swap] {$\scriptstyle{\operatorname{t}^2_{\CATB}}$} (A0_2);
    \path (A2_2) edge [->]node [auto] {$\scriptstyle{\operatorname{u}^2_{\CATB}
      \circ\operatorname{t}^2_{\CATB}}$} (A4_2);
\end{tikzpicture}
\end{equation}

Then taking the inverse of the previous associator, it is easy to prove that the $2$-morphism

\[\Gamma^5=\left(i_{\left(\functor{F}_0(A^2_{\CATA}),\id_{\functor{F}_0(A^2_{\CATA})},
\functor{F}_1(f^2_{\CATA})\right)}\ast\eqref{eq-43}^{-1}\right)\ast i_{\left(\functor{F}_0
(A^3_{\CATA}),\functor{F}_1(\operatorname{w}^1_{\CATA}\circ\operatorname{v}^1_{\CATA}),
\id_{\functor{F}_0(A^3_{\CATA})}\right)}\]
is represented by the following diagram:

\begin{equation}\label{eq-38}
\begin{tikzpicture}[xscale=2.6,yscale=0.8]
    \node (A0_2) at (2, 0) {$A^2_{\CATB}$};
    \node (A2_0) at (0, 2) {$\functor{F}_0(A_{\CATA})$};
    \node (A2_2) at (2, 2) {$\overline{A}^2_{\CATB}$};
    \node (A2_4) at (4, 2) {$\functor{F}_0(B_{\CATA})$.};
    \node (A4_2) at (2, 4) {$\functor{F}_0(A^3_{\CATA})$};

    \node (A2_1) at (1.1, 2) {$\Downarrow\,i_{\functor{F}_1(\operatorname{w}^1_{\CATA}
      \circ\operatorname{v}^1_{\CATA})\circ\operatorname{u}^2_{\CATB}
      \circ\operatorname{t}^2_{\CATB}}$};
    \node (A2_3) at (2.9, 2) {$\Downarrow\,i_{\functor{F}_1(f^2_{\CATA})}
      \ast\mu^2_{\CATB}$};
    
    \path (A0_2) edge [->]node [auto] {$\scriptstyle{\functor{F}_1
      (\operatorname{w}^1_{\CATA}\circ\operatorname{v}^1_{\CATA})
      \circ\operatorname{u}^2_{\CATB}}$} (A2_0);
    \path (A0_2) edge [->]node [auto,swap] {$\scriptstyle{\functor{F}_1
      (f^2_{\CATA})\circ\operatorname{z}^2_{\CATB}}$} (A2_4);
    \path (A4_2) edge [->]node [auto,swap] {$\scriptstyle{\functor{F}_1
      (\operatorname{w}^1_{\CATA}\circ\operatorname{v}^1_{\CATA})}$} (A2_0);
    \path (A4_2) edge [->]node [auto] {$\scriptstyle{\functor{F}_1
      (f^2_{\CATA}\circ\operatorname{v}^2_{\CATA})}$} (A2_4);
    \path (A2_2) edge [->]node [auto] {$\scriptstyle{\operatorname{t}^2_{\CATB}}$} (A0_2);
    \path (A2_2) edge [->]node [auto,swap] {$\scriptstyle{\operatorname{u}^2_{\CATB}
      \circ\operatorname{t}^2_{\CATB}}$} (A4_2);
\end{tikzpicture}
\end{equation}

Now we need to compute $\Gamma^3$; by definition of $\overline{\functor{F}}$ and of
$\functor{U}_{\SETWB,2}$ (see \eqref{eq-95}), we have that $\overline{\functor{F}}_2
(\alpha_{\CATA}^{-1})$ is represented by the following diagram:

\begin{equation}\label{eq-44}
\begin{tikzpicture}[xscale=2.6,yscale=-0.8]
    \node (A0_2) at (2, 0) {$\functor{F}_0(A^3_{\CATA})$};
    \node (A2_0) at (0, 2) {$\functor{F}_0(A^3_{\CATA})$};
    \node (A2_2) at (2, 2) {$\functor{F}_0(A^3_{\CATA})$};
    \node (A2_4) at (4, 2) {$\functor{F}_0(A_{\CATA})$.};
    \node (A4_2) at (2, 4) {$\functor{F}_0(A^3_{\CATA})$};

    \node (A2_3) at (2.9, 2) {$\Downarrow\,\functor{F}_2\left(\alpha_{\CATA}
      \right)^{-1}$};
    \node (A2_1) at (1.1, 2) {$\Downarrow\,i_{\id_{\functor{F}_0(A^3_{\CATA})}}$};
    
    \path (A0_2) edge [->]node [auto,swap] {$\scriptstyle{\id_{\functor{F}_0(A^3_{\CATA})}}$} (A2_0);
    \path (A4_2) edge [->]node [auto] {$\scriptstyle{\id_{\functor{F}_0(A^3_{\CATA})}}$} (A2_0);
    \path (A4_2) edge [->]node [auto,swap] {$\scriptstyle{\functor{F}_1
      (\operatorname{w}^1_{\CATA}\circ\operatorname{v}^1_{\CATA})}$} (A2_4);
    \path (A2_2) edge [->]node [auto,swap] {$\scriptstyle{\id_{\functor{F}_0(A^3_{\CATA})}}$} (A0_2);
    \path (A2_2) edge [->]node [auto] {$\scriptstyle{\id_{\functor{F}_0(A^3_{\CATA})}}$} (A4_2);
    \path (A0_2) edge [->]node [auto] {$\scriptstyle{\functor{F}_1(\operatorname{w}^2_{\CATA}
      \circ\operatorname{v}^2_{\CATA})}$} (A2_4);
\end{tikzpicture}
\end{equation}

Then we need to compute the composition

\begin{equation}\label{eq-24}
i_{\left(\functor{F}_0(A^2_{\CATA}),\functor{F}_1(\operatorname{w}^2_{\CATA}),
\functor{F}_1(f^2_{\CATA})\right)}\ast\eqref{eq-44}
\end{equation}

For that, we are going to use~\cite[Proposition~0.4]{T3} with $\Gamma$ given by the class of
\eqref{eq-44} and $\underline{g}:=(\functor{F}_0(A^2_{\CATA}),\functor{F}_1
(\operatorname{w}^2_{\CATA}),\functor{F}_1(f^2_{\CATA}))$. In order to do that, we have to
identify the pair of diagrams appearing in~\cite[Proposition~0.4(0.14)]{T3} for $m=1,2$. In this
case, we use \eqref{eq-52} for $m=2$ and \eqref{eq-51}, so we are in the hypothesis
of~\cite[Proposition~0.4]{T3} if we set

\[
\begin{tikzpicture}[xscale=4.2,yscale=-0.8]
    \node (A0_1) at (0, 2) {$A^1:=\functor{F}_0(A^3_{\CATA})$};
    \node (A1_0) at (1, 0) {$A^{\prime 1}:=A^2_{\CATB}$};
    \node (A1_2) at (1.1, 2) {$B:=\functor{F}_0(A_{\CATA})$};
    \node (A2_1) at (2, 2) {$B':=\functor{F}_0(A^2_{\CATA})$};

    \node (A1_1) at (1, 1) {$\rho^1:=\sigma^2_{\CATB}$};
    \node (B1_1) at (1, 1.4) {$\Rightarrow$};

    \path (A1_0) edge [->]node [auto,swap] {$\scriptstyle{\operatorname{u}^{\prime 1}:=
      \operatorname{u}^2_{\CATB}}$} (A0_1);
    \path (A1_0) edge [->]node [auto] {$\scriptstyle{f^{\prime 1}:=
      \operatorname{z}^2_{\CATB}}$} (A2_1);
    \path (A0_1) edge [->]node [auto,swap] {$\scriptstyle{f^1:=\functor{F}_1
      (\operatorname{w}^2_{\CATA}\circ\operatorname{v}^2_{\CATA})}$} (A1_2);
    \path (A2_1) edge [->]node [auto] {$\scriptstyle{\operatorname{u}:=
      \functor{F}_1(\operatorname{w}^2_{\CATA})}$} (A1_2);
\end{tikzpicture}
\]
and

\[
\begin{tikzpicture}[xscale=4.2,yscale=-0.8]
    \node (A0_1) at (0, 2) {$A^2:=\functor{F}_0(A^3_{\CATA})$};
    \node (A1_0) at (1, 0) {$A^{\prime 2}:=A^3_{\CATB}$};
    \node (A1_2) at (1.1, 2) {$B=\functor{F}_0(A_{\CATA})$};
    \node (A2_1) at (2, 2) {$B'=\functor{F}_0(A^2_{\CATA})$.};

    \node (A1_1) at (1, 1) {$\rho^2:=\sigma^3_{\CATB}$};
    \node (B1_1) at (1, 1.4) {$\Rightarrow$};

    \path (A1_0) edge [->]node [auto,swap] {$\scriptstyle{\operatorname{u}^{\prime 2}:=
      \operatorname{u}^3_{\CATB}}$} (A0_1);
    \path (A1_0) edge [->]node [auto] {$\scriptstyle{f^{\prime 2}:=
      \operatorname{z}^3_{\CATB}}$} (A2_1);
    \path (A0_1) edge [->]node [auto,swap] {$\scriptstyle{f^2:=\functor{F}_1
      (\operatorname{w}^1_{\CATA}\circ\operatorname{v}^1_{\CATA})}$} (A1_2);
    \path (A2_1) edge [->]node [auto] {$\scriptstyle{\operatorname{u}=
      \functor{F}_1(\operatorname{w}^2_{\CATA})}$} (A1_2);
\end{tikzpicture}
\]

Then we need to fix a set of choices (F8) -- (F10) as in~\cite[Proposition~0.4]{T3}. In order
to do that, we do a preliminary step as follows: using (BF3)
for $(\CATB,\SETWB)$, we choose data as in
the upper part of the following diagram, with $\operatorname{r}^2_{\CATB}$ in $\SETWB$ and
$\eta_{\CATB}$ invertible.

\begin{equation}\label{eq-50}
\begin{tikzpicture}[xscale=3.5,yscale=-0.8]
    \node (A0_1) at (1, 0) {$\overline{A}_{\CATB}$};
    \node (A1_0) at (0, 2) {$\overline{A}^2_{\CATB}$};
    \node (A1_2) at (2, 2) {$A^3_{\CATB}$.};
    \node (A2_1) at (1, 2) {$\functor{F}_0(A^3_{\CATA})$};
    
    \node (A1_1) at (1, 1) {$\eta_{\CATB}$};
    \node (B1_1) at (1, 1.4) {$\Rightarrow$};
    
    \path (A1_2) edge [->]node [auto] {$\scriptstyle{\operatorname{u}^3_{\CATB}}$} (A2_1);
    \path (A0_1) edge [->]node [auto] {$\scriptstyle{\operatorname{r}^3_{\CATB}}$} (A1_2);
    \path (A1_0) edge [->]node [auto,swap] {$\scriptstyle{\operatorname{u}^2_{\CATB}
      \circ\operatorname{t}^2_{\CATB}}$} (A2_1);
    \path (A0_1) edge [->]node [auto,swap] {$\scriptstyle{\operatorname{r}^2_{\CATB}}$} (A1_0);
\end{tikzpicture}
\end{equation}

Moreover, using (BF4a) and (BF4b)
for $(\CATB,\SETWB)$,
we choose an object
$\widetilde{A}_{\CATB}$, a morphism $\operatorname{s}_{\CATB}:\widetilde{A}_{\CATB}\rightarrow
\overline{A}_{\CATB}$ in $\SETWB$ and an invertible $2$-morphism

\begin{equation}\label{eq-01}
\rho_{\CATB}:\,\operatorname{z}^2_{\CATB}\circ\operatorname{t}^2_{\CATB}\circ
\operatorname{r}^2_{\CATB}\circ\operatorname{s}_{\CATB}\Longrightarrow\operatorname{z}^3_{\CATB}
\circ\operatorname{r}^3_{\CATB}\circ\operatorname{s}_{\CATB},
\end{equation}
such that $i_{\functor{F}_1(\operatorname{w}^2_{\CATA})}\ast\rho_{\CATB}$ coincides with the
following composition:

\begin{equation}\label{eq-22}
\begin{tikzpicture}[xscale=2.4,yscale=-0.8]
    \node (A0_3) at (1.6, 0) {$A^2_{\CATB}$};
    \node (A0_4) at (4.4, 0) {$\functor{F}_0(A^2_{\CATA})$};
    \node (A2_0) at (0.4, 2) {$\widetilde{A}_{\CATB}$};
    \node (A2_1) at (1, 2) {$\overline{A}_{\CATB}$};
    \node (A2_3) at (2.2, 2) {$\functor{F}_0(A^3_{\CATA})$};
    \node (A2_5) at (5, 2) {$\functor{F}_0(A_{\CATA})$.};
    \node (A4_2) at (1.6, 4) {$A^3_{\CATB}$};
    \node (A4_4) at (4.4, 4) {$\functor{F}_0(A^2_{\CATA})$};
    
    \node (A1_4) at (2.8, 0.7) {$\Downarrow\,(\sigma^2_{\CATB})^{-1}$};
    \node (A2_4) at (3.6, 2) {$\Downarrow\,\functor{F}_2(\alpha_{\CATA})^{-1}$};
    \node (A3_4) at (2.8, 3.3) {$\Downarrow\,\sigma^3_{\CATB}$};
    \node (A2_2) at (1.6, 2) {$\Downarrow\,\eta_{\CATB}$};
    
    \path (A2_3) edge [->,bend right=25]node [auto] {$\scriptstyle{\functor{F}_1
      (\operatorname{w}^2_{\CATA}\circ\operatorname{v}^2_{\CATA})}$} (A2_5);
    \path (A2_3) edge [->,bend left=25]node [auto,swap] {$\scriptstyle{\functor{F}_1
      (\operatorname{w}^1_{\CATA}\circ\operatorname{v}^1_{\CATA})}$} (A2_5);
    \path (A4_2) edge [->]node [auto,swap] {$\scriptstyle{\operatorname{z}^3_{\CATB}}$} (A4_4);
    \path (A0_3) edge [->]node [auto] {$\scriptstyle{\operatorname{u}^2_{\CATB}}$} (A2_3);
    \path (A0_3) edge [->]node [auto] {$\scriptstyle{\operatorname{z}^2_{\CATB}}$} (A0_4);
    \path (A2_1) edge [->]node [auto,swap] {$\scriptstyle{\operatorname{r}^3_{\CATB}}$} (A4_2);
    \path (A4_4) edge [->]node [auto,swap] {$\scriptstyle{\functor{F}_1
      (\operatorname{w}^2_{\CATA})}$} (A2_5);
    \path (A0_4) edge [->]node [auto] {$\scriptstyle{\functor{F}_1
      (\operatorname{w}^2_{\CATA})}$} (A2_5);
    \path (A4_2) edge [->]node [auto,swap] {$\scriptstyle{\operatorname{u}^3_{\CATB}}$} (A2_3);
    \path (A2_0) edge [->]node [auto] {$\scriptstyle{\operatorname{s}_{\CATB}}$} (A2_1);
    \path (A2_1) edge [->]node [auto] {$\scriptstyle{\operatorname{t}^2_{\CATB}
      \circ\operatorname{r}^2_{\CATB}}$} (A0_3);
\end{tikzpicture}
\end{equation}

Using \eqref{eq-16}, this implies that

\begin{gather}
\nonumber \Big(i_{\functor{F}_1(\operatorname{w}^2_{\CATA})}\ast\rho_{\CATB}\Big)\odot
 \Big(i_{\functor{F}_1(\operatorname{w}^2_{\CATA})}\ast\mu^2_{\CATB}\ast
 i_{\operatorname{r}^2_{\CATB}\circ\operatorname{s}_{\CATB}}\Big)\odot\Big(\functor{F}_2
 (\alpha_{\CATA})\ast i_{\operatorname{u}^2_{\CATB}\circ\operatorname{t}^2_{\CATB}\circ
 \operatorname{r}^2_{\CATB}\circ\operatorname{s}_{\CATB}}\Big)= \\
\label{eq-29} =\Big(\sigma^3_{\CATB}\ast i_{\operatorname{r}^3_{\CATB}\circ\operatorname{s}_{\CATB}}
 \Big)\odot\Big(i_{\functor{F}_1(\operatorname{w}^1_{\CATA}\circ\operatorname{v}^1_{\CATA})}\ast
 \eta_{\CATB}\ast i_{\operatorname{s}_{\CATB}}\Big).
\end{gather}

Then we fix the following choices:

\begin{description}
 \item[(F8)] we choose the data in the upper part of the following $2$ diagrams:
  
  \[
  \begin{tikzpicture}[xscale=-3.9,yscale=-0.8]
    \node (A0_1) at (0, 2) {$A^{\prime 1}=A^2_{\CATB}$,};
    \node (A1_0) at (1, 0) {$A^{\prime\prime 1}:=\overline{A}^2_{\CATB}$};
    \node (A1_2) at (1, 2) {$A^1=\functor{F}_0(A^3_{\CATA})$};
    \node (A2_1) at (2, 2) {$A^3=\functor{F}_0(A^3_{\CATA})$};

    \node (A1_1) at (1, 1) {$\eta^1:=i_{\operatorname{u}^2_{\CATB}
      \circ\operatorname{t}^2_{\CATB}}$};
    \node (B1_1) at (1, 1.45) {$\Rightarrow$};

    \path (A1_0) edge [->]node [auto] {$\scriptstyle{\operatorname{v}^{\prime 1}:=
      \operatorname{t}^2_{\CATB}}$} (A0_1);
    \path (A1_0) edge [->]node [auto,swap]
      {$\scriptstyle{\operatorname{u}^{\prime\prime 1}:=
      \operatorname{u}^2_{\CATB}\circ\operatorname{t}^2_{\CATB}}$} (A2_1);
    \path (A0_1) edge [->]node [auto] {$\scriptstyle{\operatorname{u}^{\prime 1}=
      \operatorname{u}^2_{\CATB}}$} (A1_2);
    \path (A2_1) edge [->]node [auto,swap] {$\scriptstyle{\operatorname{v}^1=
      \id_{\functor{F}_0(A^3_{\CATA})}}$} (A1_2);
  \end{tikzpicture}
  \]

  \[
  \begin{tikzpicture}[xscale=-3.9,yscale=-0.8]
    \node (A0_1) at (0, 2) {$A^{\prime 2}=A^3_{\CATB}$;};
    \node (A1_0) at (1, 0) {$A^{\prime\prime 2}:=A^3_{\CATB}$};
    \node (A1_2) at (1, 2) {$A^2=\functor{F}_0(A^3_{\CATA})$};
    \node (A2_1) at (2, 2) {$A^3=\functor{F}_0(A^3_{\CATA})$};

    \node (A1_1) at (1, 0.95) {$\eta^2:=i_{\operatorname{u}^3_{\CATB}}$};
    \node (B1_1) at (1, 1.4) {$\Rightarrow$};

    \path (A1_0) edge [->]node [auto] {$\scriptstyle{\operatorname{v}^{\prime 2}:=
      \id_{A^3_{\CATB}}}$} (A0_1);
    \path (A1_0) edge [->]node [auto,swap]
      {$\scriptstyle{\operatorname{u}^{\prime\prime 2}:=
      \operatorname{u}^3_{\CATB}}$} (A2_1);
    \path (A0_1) edge [->]node [auto] {$\scriptstyle{\operatorname{u}^{\prime 2}=
      \operatorname{u}^3_{\CATB}}$} (A1_2);
    \path (A2_1) edge [->]node [auto,swap] {$\scriptstyle{\operatorname{v}^2=
      \id_{\functor{F}_0(A^3_{\CATA})}}$} (A1_2);
  \end{tikzpicture}
  \]
 
 \item[(F9)] we choose $A'':=\widetilde{A}_{\CATB}$, $\operatorname{z}^1:=
  \operatorname{r}^2_{\CATB}\circ\operatorname{s}_{\CATB}$,
  $\operatorname{z}^2:=\operatorname{r}^3_{\CATB}\circ\operatorname{s}_{\CATB}$ and
  $\eta^3:=\eta_{\CATB}\ast i_{\operatorname{s}_{\CATB}}$;

 \item[(F10)] we choose $\beta':=\rho_{\CATB}$ (the fact that $i_{\functor{F}_1
  (\operatorname{w}^2_{\CATA})}\ast\rho_{\CATB}$ coincides with diagram \eqref{eq-22} implies
  immediately that the condition of~\cite[Proposition~0.4(F10)]{T3} is verified).
\end{description}

Then according to~\cite[Proposition~0.4]{T3}, we get that \eqref{eq-24} is represented by
the following diagram:

\begin{equation}\label{eq-45}
\begin{tikzpicture}[xscale=2.6,yscale=-0.8]
    \node (A0_2) at (2, 0) {$A^2_{\CATB}$};
    \node (A2_0) at (0, 2) {$\functor{F}_0(A^3_{\CATA})$};
    \node (A2_2) at (2, 2) {$\widetilde{A}_{\CATB}$};
    \node (A2_4) at (4, 2) {$\functor{F}_0(B_{\CATA})$.};
    \node (A4_2) at (2, 4) {$A^3_{\CATB}$};

    \node (A2_1) at (1.1, 2) {$\Downarrow\,\eta_{\CATB}\ast i_{\operatorname{s}_{\CATB}}$};
    \node (A2_3) at (2.9, 2) {$\Downarrow\,i_{\functor{F}_1(f^2_{\CATA})}\ast\rho_{\CATB}$};
    
    \path (A4_2) edge [->]node [auto,swap] {$\scriptstyle{\functor{F}_1(f^2_{\CATA})
      \circ\operatorname{z}^3_{\CATB}}$} (A2_4);
    \path (A0_2) edge [->]node [auto] {$\scriptstyle{\functor{F}_1(f^2_{\CATA})
      \circ\operatorname{z}^2_{\CATB}}$} (A2_4);
    \path (A0_2) edge [->]node [auto,swap] {$\scriptstyle{\operatorname{u}^2_{\CATB}}$} (A2_0);
    \path (A2_2) edge [->]node [auto,swap] {$\scriptstyle{\operatorname{t}^2_{\CATB}
      \circ\operatorname{r}^2_{\CATB}\circ\operatorname{s}_{\CATB}}$} (A0_2);
    \path (A2_2) edge [->]node [auto] {$\scriptstyle{\operatorname{r}^3_{\CATB}
      \circ\operatorname{s}_{\CATB}}$} (A4_2);
    \path (A4_2) edge [->]node [auto] {$\scriptstyle{\operatorname{u}^3_{\CATB}}$} (A2_0);
\end{tikzpicture}
\end{equation}

Lastly, in order to compute $\Gamma^3$ we need to compose the class of \eqref{eq-45} with the
$2$-identity of the morphism

\[\functor{P}(\operatorname{w}^1_{\CATA}\circ\operatorname{v}^1_{\CATA})=\left(\functor{F}_0
(A^3_{\CATA}),\functor{F}_1(\operatorname{w}^1_{\CATA}\circ\operatorname{v}^1_{\CATA}),
\id_{\functor{F}_0(A^3_{\CATA})}\right)\]
(applied on the left of \eqref{eq-45}). Then we get easily that $\Gamma^3$ is represented by
the following diagram:

\begin{equation}\label{eq-26}
\begin{tikzpicture}[xscale=2.6,yscale=-0.8]
    \node (A0_2) at (2, 0) {$A^2_{\CATB}$};
    \node (A2_0) at (0, 2) {$\functor{F}_0(A_{\CATA})$};
    \node (A2_2) at (2, 2) {$\widetilde{A}_{\CATB}$};
    \node (A2_4) at (4, 2) {$\functor{F}_0(B_{\CATA})$.};
    \node (A4_2) at (2, 4) {$A^3_{\CATB}$};

    \node (A2_1) at (1.1, 2) {$\Downarrow\,i_{\functor{F}_1(\operatorname{w}^1_{\CATA}
      \circ\operatorname{v}^1_{\CATA})}\ast\eta_{\CATB}\ast i_{\operatorname{s}_{\CATB}}$};
    \node (A2_3) at (2.8, 2.1) {$\Downarrow\,i_{\functor{F}_1(f^2_{\CATA})}\ast\rho_{\CATB}$};
         
    \path (A4_2) edge [->]node [auto,swap] {$\scriptstyle{\functor{F}_1(f^2_{\CATA})
      \circ\operatorname{z}^3_{\CATB}}$} (A2_4);
    \path (A0_2) edge [->]node [auto] {$\scriptstyle{\functor{F}_1(f^2_{\CATA})
      \circ\operatorname{z}^2_{\CATB}}$} (A2_4);
    \path (A0_2) edge [->]node [auto,swap] {$\scriptstyle{\functor{F}_1
      (\operatorname{w}^1_{\CATA}\circ\operatorname{v}^1_{\CATA})
      \circ\operatorname{u}^2_{\CATB}}$} (A2_0);
    \path (A2_2) edge [->]node [auto,swap] {$\scriptstyle{\operatorname{t}^2_{\CATB}\circ
      \operatorname{r}^2_{\CATB}\circ\operatorname{s}_{\CATB}}$} (A0_2);
    \path (A2_2) edge [->]node [auto] {$\scriptstyle{\operatorname{r}^3_{\CATB}
      \circ\operatorname{s}_{\CATB}}$} (A4_2);
    \path (A4_2) edge [->]node [auto] {$\scriptstyle{\functor{F}_1
      (\operatorname{w}^1_{\CATA}\circ\operatorname{v}^1_{\CATA})
      \circ\operatorname{u}^3_{\CATB}}$} (A2_0);
\end{tikzpicture}
\end{equation}

Now we need to compute $\Gamma^1$; we recall that the inverse of $\Xi(\operatorname{w}^1_{\CATA}
\circ\operatorname{v}^1_{\CATA})$ has a representative given by the inverse of \eqref{eq-28}
(namely, the same diagram with upper and lower part interchanged); moreover,

\[\overline{\functor{F}}_1(f^2_{\CATA})\circ\functor{P}(\operatorname{w}^2_{\CATA})=
\Big(\functor{F}_0(A^2_{\CATA}),\functor{F}_1(\operatorname{w}^2_{\CATA}),
\functor{F}_1(f^2_{\CATA})\Big).\]

We also recall that by construction the morphisms $\operatorname{u}^2_{\CATB}$,
$\operatorname{t}_{\CATB}$, $\operatorname{r}^2_{\CATB}$ and $\operatorname{s}_{\CATB}$ belong
all to $\SETW$. So by (BF2) and (BF5) applied to $\eta_{\CATB}^{-1}
\ast i_{\operatorname{s}_{\CATB}}$, we get that $\operatorname{u}^3_{\CATB}\circ
\operatorname{r}^3_{\CATB}\circ\operatorname{s}_{\CATB}$ belongs to $\SETWB$. Since also 
$\operatorname{u}^3_{\CATB}$ belongs to $\SETWB$, then by Proposition~\ref{prop-03}(ii) we conclude
that $\operatorname{r}^3_{\CATB}\circ\operatorname{s}_{\CATB}$ belongs to $\SETWBsat$.
So by Definition~\ref{def-01} there are an object $\widetilde{A}'_{\CATB}$ and a morphism
$\operatorname{s}'_{\CATB}:\widetilde{A}'_{\CATB}\rightarrow\widetilde{A}_{\CATB}$, such that
$\operatorname{r}^3_{\CATB}\circ\operatorname{s}_{\CATB}\circ
\operatorname{s}'_{\CATB}$ belongs to $\SETWB$.\\

Then we use~\cite[Proposition~0.4]{T3} with $\rho^1:=\sigma^3_{\CATB}$ and $\rho^2=i_{\functor{F}
(\operatorname{w}^2_{\CATA})}$; a set of choices (F8) -- (F10) for this case is easily given by
$\eta^1:=i_{\operatorname{u}^3_{\CATB}}$, $\eta^2:=\sigma^3_{\CATB}$, $\eta^3:=
i_{\operatorname{u}^3_{\CATB}\circ\operatorname{r}^3_{\CATB}\circ\operatorname{s}_{\CATB}
\circ\operatorname{s}'_{\CATB}}$
and $\beta':=i_{\operatorname{z}^3_{\CATB}\circ\operatorname{r}^3_{\CATB}\circ
\operatorname{s}_{\CATB}\circ\operatorname{s}'_{\CATB}}$. Then we get that $\Gamma^1$ is represented
by the following diagram:

\[
\begin{tikzpicture}[xscale=2.8,yscale=-0.8]
    \node (A0_2) at (2, 0) {$A^3_{\CATB}$};
    \node (A2_0) at (0, 2) {$\functor{F}_0(A_{\CATA})$};
    \node (A2_2) at (2, 2) {$\widetilde{A}'_{\CATB}$};
    \node (A2_4) at (4, 2) {$\functor{F}_0(B_{\CATA})$.};
    \node (A4_2) at (2, 4) {$\functor{F}_0(A^2_{\CATA})$};
    
    \node (A2_3) at (2.9, 2) {$\Downarrow\,i_{\functor{F}_1(f^2_{\CATA})
      \circ\operatorname{z}^3_{\CATB}\circ\operatorname{r}^3_{\CATB}
      \circ\operatorname{s}_{\CATB}\circ\operatorname{s}'_{\CATB}}$};
    \node (A2_1) at (1.1, 2) {$\Downarrow\,\sigma^3_{\CATB}\ast
      i_{\operatorname{r}^3_{\CATB}\circ\operatorname{s}_{\CATB}\circ\operatorname{s}'_{\CATB}}$};

    \node (C1_1) at (2.48, 2.78) {$\scriptstyle{\operatorname{z}^3_{\CATB}
      \circ\operatorname{r}^3_{\CATB}\circ\operatorname{s}_{\CATB}
      \circ\operatorname{s}'_{\CATB}}$};
    
    \path (A0_2) edge [->]node [auto,swap] {$\scriptstyle{\functor{F}_1
      (\operatorname{w}^1_{\CATA}\circ\operatorname{v}^1_{\CATA})
      \circ\operatorname{u}^3_{\CATB}}$} (A2_0);
    \path (A4_2) edge [->]node [auto] {$\scriptstyle{\functor{F}_1
      (\operatorname{w}^2_{\CATA})}$} (A2_0);
    \path (A4_2) edge [->]node [auto,swap] {$\scriptstyle{\functor{F}_1(f^2_{\CATA})}$} (A2_4);
    \path (A2_2) edge [->]node [auto,swap] {$\scriptstyle{\operatorname{r}^3_{\CATB}
      \circ\operatorname{s}_{\CATB}\circ\operatorname{s}'_{\CATB}}$} (A0_2);
    \path (A2_2) edge [->]node [auto] {} (A4_2);
    \path (A0_2) edge [->]node [auto] {$\scriptstyle{\functor{F}_1(f^2_{\CATA})
      \circ\operatorname{z}^3_{\CATB}}$} (A2_4);
\end{tikzpicture}
\]

Therefore, $\Gamma^1$ is also represented by the following diagram:

\begin{equation}\label{eq-25}
\begin{tikzpicture}[xscale=2.6,yscale=-0.8]
    \node (A0_2) at (2, 0) {$A^3_{\CATB}$};
    \node (A2_0) at (0, 2) {$\functor{F}_0(A_{\CATA})$};
    \node (A2_2) at (2, 2) {$\widetilde{A}_{\CATB}$};
    \node (A2_4) at (4, 2) {$\functor{F}_0(B_{\CATA})$.};
    \node (A4_2) at (2, 4) {$\functor{F}_0(A^2_{\CATA})$};
    
    \node (A2_3) at (2.9, 2) {$\Downarrow\,i_{\functor{F}_1(f^2_{\CATA})
      \circ\operatorname{z}^3_{\CATB}\circ\operatorname{r}^3_{\CATB}
      \circ\operatorname{s}_{\CATB}}$};
    \node (A2_1) at (1.1, 2) {$\Downarrow\,\sigma^3_{\CATB}\ast
      i_{\operatorname{r}^3_{\CATB}\circ\operatorname{s}_{\CATB}}$};

    \path (A0_2) edge [->]node [auto,swap] {$\scriptstyle{\functor{F}_1
      (\operatorname{w}^1_{\CATA}\circ\operatorname{v}^1_{\CATA})
      \circ\operatorname{u}^3_{\CATB}}$} (A2_0);
    \path (A4_2) edge [->]node [auto] {$\scriptstyle{\functor{F}_1
      (\operatorname{w}^2_{\CATA})}$} (A2_0);
    \path (A4_2) edge [->]node [auto,swap] {$\scriptstyle{\functor{F}_1(f^2_{\CATA})}$} (A2_4);
    \path (A2_2) edge [->]node [auto,swap] {$\scriptstyle{\operatorname{r}^3_{\CATB}
      \circ\operatorname{s}_{\CATB}}$} (A0_2);
    \path (A2_2) edge [->]node [auto] {$\scriptstyle{\operatorname{z}^3_{\CATB}
      \circ\operatorname{r}^3_{\CATB}\circ\operatorname{s}_{\CATB}}$} (A4_2);
    \path (A0_2) edge [->]node [auto] {$\scriptstyle{\functor{F}_1(f^2_{\CATA})
      \circ\operatorname{z}^3_{\CATB}}$} (A2_4);
\end{tikzpicture}
\end{equation}

Then using \eqref{eq-25}, \eqref{eq-26} and \eqref{eq-29}, we get that $\Gamma^1\odot\Gamma^3$
is represented by the following diagram:

\[
\begin{tikzpicture}[xscale=2.6,yscale=-0.8]
    \node (A0_2) at (2, 0) {$A^2_{\CATB}$};
    \node (A2_0) at (0, 2) {$\functor{F}_0(A_{\CATA})$};
    \node (A2_2) at (2, 2) {$\widetilde{A}_{\CATB}$};
    \node (A2_4) at (4, 2) {$\functor{F}_0(B_{\CATA})$,};
    \node (A4_2) at (2, 4) {$\functor{F}_0(A^2_{\CATA})$};

    \node (A2_3) at (2.8, 2) {$\Downarrow\,i_{\functor{F}_1(f^2_{\CATA})}\ast\rho_{\CATB}$};
    \node (B1_1) at (1.1, 2) {$\Downarrow\,\varepsilon_{\CATB}$};
    
    \path (A4_2) edge [->]node [auto,swap] {$\scriptstyle{\functor{F}_1(f^2_{\CATA})}$} (A2_4);
    \path (A0_2) edge [->]node [auto] {$\scriptstyle{\functor{F}_1(f^2_{\CATA})
      \circ\operatorname{z}^2_{\CATB}}$} (A2_4);
    \path (A4_2) edge [->]node [auto] {$\scriptstyle{\functor{F}_1
      (\operatorname{w}^2_{\CATA})}$} (A2_0);
    \path (A0_2) edge [->]node [auto,swap] {$\scriptstyle{\functor{F}_1
      (\operatorname{w}^1_{\CATA}\circ\operatorname{v}^1_{\CATA})
      \circ\operatorname{u}^2_{\CATB}}$} (A2_0);
    \path (A2_2) edge [->]node [auto] {$\scriptstyle{\operatorname{t}^2_{\CATB}
      \circ\operatorname{r}^2_{\CATB}\circ\operatorname{s}_{\CATB}}$} (A0_2);
    \path (A2_2) edge [->]node [auto,swap] {$\scriptstyle{\operatorname{z}^3_{\CATB}
      \circ\operatorname{r}^3_{\CATB}\circ\operatorname{s}_{\CATB}}$} (A4_2);
\end{tikzpicture}
\]
where $\varepsilon_{\CATB}$ is the following composition

\[\Big(i_{\functor{F}_1(\operatorname{w}^2_{\CATA})}\ast\rho_{\CATB}\Big)\odot
\Big(i_{\functor{F}_1(\operatorname{w}^2_{\CATA})}\ast\mu^2_{\CATB}\ast i_{\operatorname{r}^2_{\CATB}
\circ\operatorname{s}_{\CATB}}\Big)\odot\Big(\functor{F}_2(\alpha_{\CATA})\ast
i_{\operatorname{u}^2_{\CATB}\circ\operatorname{t}^2_{\CATB}\circ\operatorname{r}^2_{\CATB}
\circ\operatorname{s}_{\CATB}}\Big).\]

Then we get easily that $\Gamma^1\odot\Gamma^3$ is also represented by the following diagram:

\begin{equation}\label{eq-31}
\begin{tikzpicture}[xscale=2.6,yscale=-0.8]
    \node (A0_2) at (2, 0) {$A^2_{\CATB}$};
    \node (A2_0) at (0, 2) {$\functor{F}_0(A_{\CATA})$};
    \node (A2_2) at (2, 2) {$\overline{A}^2_{\CATB}$};
    \node (A2_4) at (4, 2) {$\functor{F}_0(B_{\CATA})$.};
    \node (A4_2) at (2, 4) {$\functor{F}_0(A^2_{\CATA})$};

    \node (A2_3) at (2.75, 1.9) {$\Downarrow\,i_{\functor{F}_1(f^2_{\CATA})}
      \ast(\mu^2_{\CATB})^{-1}$};
    \node (B3_3) at (1.2, 2) {$\Downarrow\,\functor{F}_2(\alpha_{\CATA})\ast
      i_{\operatorname{u}^2_{\CATB}\circ\operatorname{t}^2_{\CATB}}$};
    
    \node (B1_1) at (2.38, 2.65) {$\scriptstyle{\functor{F}_1(\operatorname{v}^2_{\CATA})
      \circ\operatorname{u}^2_{\CATB}\circ}$};
    \node (B2_2) at (2.15, 3.1) {$\scriptstyle{\circ\operatorname{t}^2_{\CATB}}$};

    \path (A4_2) edge [->]node [auto,swap] {$\scriptstyle{\functor{F}_1(f^2_{\CATA})}$} (A2_4);
    \path (A0_2) edge [->]node [auto] {$\scriptstyle{\functor{F}_1(f^2_{\CATA})
      \circ\operatorname{z}^2_{\CATB}}$} (A2_4);
    \path (A4_2) edge [->]node [auto] {$\scriptstyle{\functor{F}_1
      (\operatorname{w}^2_{\CATA})}$} (A2_0);
    \path (A0_2) edge [->]node [auto,swap] {$\scriptstyle{\functor{F}_1
      (\operatorname{w}^1_{\CATA}\circ\operatorname{v}^1_{\CATA})
      \circ\operatorname{u}^2_{\CATB}}$} (A2_0);
    \path (A2_2) edge [->]node [auto,swap] {$\scriptstyle{\operatorname{t}^2_{\CATB}}$} (A0_2);
    \path (A2_2) edge [->]node [auto] {} (A4_2);
\end{tikzpicture}
\end{equation}

If we compose \eqref{eq-31} with \eqref{eq-38}, we get that $\Gamma^1\odot\Gamma^3\odot\Gamma^5$
is represented by the following diagram:

\begin{equation}\label{eq-35}
\begin{tikzpicture}[xscale=2.8,yscale=-0.8]
    \node (A0_2) at (2, 0) {$\functor{F}_0(A^3_{\CATA})$};
    \node (A2_0) at (0, 2) {$\functor{F}_0(A_{\CATA})$};
    \node (A2_2) at (2, 2) {$\functor{F}_0(A^3_{\CATA})$};
    \node (A2_4) at (4, 2) {$\functor{F}_0(B_{\CATA})$.};
    \node (A4_2) at (2, 4) {$\functor{F}_0(A^2_{\CATA})$};

    \node (A2_3) at (2.8, 2) {$\Downarrow\,i_{\functor{F}_1(f^2_{\CATA}
      \circ\operatorname{v}^2_{\CATA})}$};
    \node (B3_3) at (1.2, 2) {$\Downarrow\,\functor{F}_2(\alpha_{\CATA})$};
    
    \path (A4_2) edge [->]node [auto,swap] {$\scriptstyle{\functor{F}_1(f^2_{\CATA})}$} (A2_4);
    \path (A0_2) edge [->]node [auto] {$\scriptstyle{\functor{F}_1(f^2_{\CATA}
      \circ\operatorname{v}^2_{\CATA})}$} (A2_4);
    \path (A4_2) edge [->]node [auto] {$\scriptstyle{\functor{F}_1
      (\operatorname{w}^2_{\CATA})}$} (A2_0);
    \path (A0_2) edge [->]node [auto,swap] {$\scriptstyle{\functor{F}_1
      (\operatorname{w}^1_{\CATA}\circ\operatorname{v}^1_{\CATA})}$} (A2_0);
    \path (A2_2) edge [->]node [auto,swap] {$\scriptstyle{\id_{\functor{F}_0(A^3_{\CATA})}}$} (A0_2);
    \path (A2_2) edge [->]node [auto] {$\scriptstyle{\functor{F}_1
      (\operatorname{v}^2_{\CATA})}$} (A4_2);
\end{tikzpicture}
\end{equation}

Lastly, using together \eqref{eq-41} and \eqref{eq-35}, we get that $\Gamma^1\odot\Gamma^3\odot
\Gamma^5\odot\Gamma^7\odot\Gamma^{10}\odot\Gamma^{12}$ is represented by the following composition:

\begin{equation}\label{eq-12}
\begin{tikzpicture}[xscale=2.8,yscale=-0.8]
    \node (A0_2) at (2, 0) {$\functor{F}_0(A^1_{\CATA})$};
    \node (A2_0) at (0, 2) {$\functor{F}_0(A_{\CATA})$};
    \node (A2_2) at (2, 2) {$\functor{F}_0(A^3_{\CATA})$};
    \node (A2_4) at (4, 2) {$\functor{F}_0(B_{\CATA})$.};
    \node (A4_2) at (2, 4) {$\functor{F}_0(A^2_{\CATA})$};

    \node (A2_3) at (2.8, 2) {$\Downarrow\,\functor{F}_2(\beta_{\CATA})$};
    \node (B3_3) at (1.2, 2) {$\Downarrow\,\functor{F}_2(\alpha_{\CATA})$};
    
    \path (A4_2) edge [->]node [auto,swap] {$\scriptstyle{\functor{F}_1(f^2_{\CATA})}$} (A2_4);
    \path (A0_2) edge [->]node [auto] {$\scriptstyle{\functor{F}_1(f^1_{\CATA})}$} (A2_4);
    \path (A4_2) edge [->]node [auto] {$\scriptstyle{\functor{F}_1
      (\operatorname{w}^2_{\CATA})}$} (A2_0);
    \path (A0_2) edge [->]node [auto,swap] {$\scriptstyle{\functor{F}_1
      (\operatorname{w}^1_{\CATA})}$} (A2_0);
    \path (A2_2) edge [->]node [auto,swap] {$\scriptstyle{\functor{F}_1
      (\operatorname{v}^1_{\CATA})}$} (A0_2);
    \path (A2_2) edge [->]node [auto] {$\scriptstyle{\functor{F}_1
      (\operatorname{v}^2_{\CATA})}$} (A4_2);
\end{tikzpicture}
\end{equation}

By the proof of~\cite[Theorem~21]{Pr} we have that $\functor{M}_2$ is well-defined on
classes (i.e.\ the equivalence class of \eqref{eq-12} does not depend on the choice of a
representative for \eqref{eq-75}). Moreover the same result implies that there is a set
of unitors $\Sigma^{\functor{M}}_{\bullet}$ and associators
$\Psi^{\functor{M}}_{\bullet}$, making the data $(\functor{M}_0,
\functor{M}_1,\functor{M}_2)$ into a pseudofunctor $\functor{M}$.
Since we are assuming that $\functor{F}$ is a strict pseudofunctor, then the class of \eqref{eq-12}
coincides with the class of \eqref{eq-11}.\\

If we denote by $\functor{U}_{\SETWB}$
the universal pseudofunctor
from $\CATB$ to $\CATB\left[\SETWBinv\right]$, again
following the proof of~\cite[Theorem~21]{Pr}, we get a pseudonatural equivalence

\[\zeta:\,\functor{U}_{\SETWB}\circ\functor{F}\Longrightarrow
\functor{M}\circ\,\functor{U}_{\SETWA}\quad\operatorname{in}\quad
\operatorname{Hom}_{\SETWA}\left(\CATA,
\CATB\left[\SETWBinv\right]\right).\]

Since we don't need to describe $\zeta$ explicitly here, we postpone its description to
Remark~\ref{rem-05} below.\\

So we have completely proved Proposition~\ref{prop-06} in the particular case where $\CATA$
and $\CATB$ are $2$-categories and $\functor{F}$ is a strict pseudofunctor (i.e.\ a $2$-functor)
between them.\\

In the general case, in the proof above we have to set

\begin{gather*}
\Xi(\operatorname{w}_{\CATA}):=\Big[\functor{F}_0(A'_{\CATA}),\functor{F}_1
 (\operatorname{w}_{\CATA}),\id_{\functor{F}_0(A'_{\CATA})}, \\
\left(\pi_{\functor{F}_1(\operatorname{w}_{\CATA})\circ\id_{\functor{F}_0(A'_{\CATA})}}
 \right)^{-1}\odot\left(\pi_{\functor{F}_1(\operatorname{w}_{\CATA})}\right)^{-1}
 \odot\upsilon_{\functor{F}_1(\operatorname{w}_{\CATA})}, \\
\left(\pi_{\functor{F}_1(\operatorname{w}_{\CATA})\circ\id_{\functor{F}_0(A'_{\CATA})}}
 \right)^{-1}\odot\left(\pi_{\functor{F}_1(\operatorname{w}_{\CATA})}\right)^{-1}\odot
 \upsilon_{\functor{F}_1(\operatorname{w}_{\CATA})}\Big]: \\
\Big(\functor{F}_0(A_{\CATA}),\id_{\functor{F}_0(A_{\CATA})},\id_{\functor{F}_0(A_{\CATA})}\Big)
 \Longrightarrow \\
\Longrightarrow\Big(\functor{F}_0(A'_{\CATA}),\functor{F}_1(\operatorname{w}_{\CATA})\circ
 \id_{\functor{F}_0(A'_{\CATA})},\functor{F}_1(\operatorname{w}_{\CATA})\circ
 \id_{\functor{F}_0(A'_{\CATA})}\Big)
\end{gather*}
and

\begin{gather*}
\Delta(\operatorname{w}_{\CATA}):=\Big[\functor{F}_0(A'_{\CATA}),
 \id_{\functor{F}_0(A'_{\CATA})},\id_{\functor{F}_0(A'_{\CATA})}, \\
\pi_{\id_{\functor{F}_0(A'_{\CATA})}\circ\id_{\functor{F}_0(A'_{\CATA})}},
 \pi_{\id_{\functor{F}_0(A'_{\CATA})}\circ\id_{\functor{F}_0(A'_{\CATA})}}\Big]: \\
\Big(\functor{F}_0(A'_{\CATA}),\id_{\functor{F}_0(A'_{\CATA})}\circ\id_{\functor{F}_0(A'_{\CATA})},
 \id_{\functor{F}_0(A'_{\CATA})}\circ\id_{\functor{F}_0(A'_{\CATA})}\Big)\Longrightarrow \\
\Longrightarrow\Big(\functor{F}_0(A'_{\CATA}),\id_{\functor{F}_0(A'_{\CATA})},\id_{\functor{F}_0
 (A'_{\CATA})}\Big),
\end{gather*}
where $\pi_{\bullet}$ and $\upsilon_{\bullet}$ are the right and left unitors of $\CATB$. 
Moreover, we have to add unitors and associators for $\CATA$, $\CATB$ and
$\functor{F}$ wherever it is necessary. Then following the previous computations, we 
get a pseudofunctor
$\functor{M}$ such that $\functor{M}_0(A_{\CATA})=
\functor{F}_0(A_{\CATA})$ for each object $A_{\CATA}$, but that has the slightly more complicated
form mentioned in \eqref{eq-63} and \eqref{eq-68}. For example, \eqref{eq-63}
follows from \eqref{eq-37} and \eqref{eq-81}: in
a $2$-category we can omit the pair of identities that we obtain in \eqref{eq-63},
but we cannot do the same if $\CATB$ is simply a bicategory.
\end{proof}

\begin{rem}\label{rem-03}
The proof above implicitly uses the axiom of choice because in (A) we had to
fix a structure of bicategory
on $\CATA\left[\SETWAinv\right]$ and on $\CATB\left[\SETWBinv\right]$, and this implicitly
requires the axiom of choice in~\cite{Pr}. However, choices (B) in the proof above do not need the
axiom of choice since we have a precise prescription on how to define each morphism $\functor{P}
(\operatorname{w}_{\CATA})$ and each $2$-morphism $\Delta(\operatorname{w}_{\CATA})$ and
$\Xi(\operatorname{w}_{\CATA})$, not relying on axioms (BF).
In other terms, the construction of $\functor{M}$ above
does not require the axiom of choice if we can fix:

\begin{itemize}
 \item a set choices \hyperref[C]{C}$(\SETWA)$ and
 \item a set of choices \hyperref[C]{C}$(\SETWB)$ satisfying condition (\hyperref[C3]{C3}),
\end{itemize}
in such a way that the axiom of choice is not used (see also~\cite[Corollary~0.6]{T3} for
more details).
\end{rem}

\begin{cor}\label{cor-01}
Let us fix any $2$ pairs $(\CATA,\SETWA)$ and $(\CATB,\SETWB)$, both satisfying conditions
\emphatic{(BF)} and any pseudofunctor $\functor{F}:\CATA\rightarrow\CATB$
such that $\functor{F}_1(\SETWA)\subseteq\SETWB$. Moreover, let us fix any set of choices
\emphatic{\hyperref[C]{C}}$(\SETWB)$ satisfying condition \emphatic{(\hyperref[C3]{C3})}.
Then there are a pseudofunctor

\[\functor{N}:\,\CATA\Big[\SETWAinv\Big]\longrightarrow\CATB\Big[\SETWBinv\Big]\]
\emphatic{(}where $\CATB[\SETWBinv]$ is the bicategory of fractions induced by choices
\emphatic{\hyperref[C]{C}}$(\SETWB)$\emphatic{)}    
and a pseudonatural equivalence $\partial:\functor{U}_{\SETWB}\circ\functor{F}
\Rightarrow\functor{N}\circ\functor{U}_{\SETWA}$, such that:

\begin{enumerate}[\emphatic{(}I\emphatic{)}]
 \item the pseudofunctor $\mu_{\partial}:\CATA\rightarrow\operatorname{Cyl}
  \left(\CATB\left[\SETWBinv\right]\right)$ associated to $\partial$ sends each morphism
  of $\SETWA$ to an internal equivalence;
 
 \item for each object $A_{\CATA}$, we have $\functor{N}_0(A_{\CATA})=
  \functor{F}_0(A_{\CATA})$;

 \item for each morphism $(A'_{\CATA},\operatorname{w}_{\CATA},f_{\CATA}):A_{\CATA}\rightarrow
  B_{\CATA}$ in $\CATA\left[\SETWAinv\right]$, we have 
   
  \[\functor{N}_1\Big(A'_{\CATA},\operatorname{w}_{\CATA}\Big)=\Big(\functor{F}_0(A'_{\CATA}),
  \functor{F}_1(\operatorname{w}_{\CATA}),\functor{F}_1(f_{\CATA})\Big);\]

 \item for each $2$-morphism as in \eqref{eq-75} in $\CATA\left[\SETWAinv\right]$, we have
  
  \begin{gather*}
  \functor{N}_2\Big(\Big[A^3_{\CATA},\operatorname{v}^1_{\CATA},
   \operatorname{v}^2_{\CATA},\alpha_{\CATA},\beta_{\CATA}\Big]\Big)= 
   \Big[\functor{F}_0(A^3_{\CATA}),\functor{F}_1(\operatorname{v}^1_{\CATA}),
   \functor{F}_1(\operatorname{v}^1_{\CATA}), \\
  \psi^{\functor{F}}_{\operatorname{w}^2_{\CATA},
   \operatorname{v}^2_{\CATA}}\odot\functor{F}_2(\alpha_{\CATA})\odot\Big(
   \psi^{\functor{F}}_{\operatorname{w}^1_{\CATA},\operatorname{v}^1_{\CATA}}\Big)^{-1},
   \psi^{\functor{F}}_{f^2_{\CATA},\operatorname{v}^2_{\CATA}}\odot
   \functor{F}_2(\beta_{\CATA})\odot\Big(\psi^{\functor{F}}_{f^1_{\CATA},
   \operatorname{v}^1_{\CATA}}\Big)^{-1}\Big]
  \end{gather*}
  \emphatic{(}where the $2$-morphisms $\psi_{\bullet}^{\functor{F}}$ are the associators of
  $\functor{F}$\emphatic{)}.
\end{enumerate}
\end{cor}

\begin{proof}
It is easy to prove that $\functor{N}$ is well-defined on $2$-morphisms, i.e.\ that it does not
depend on the representative chosen for \eqref{eq-75}. So the statement gives
a description of $\functor{N}$ on objects, morphisms and $2$-morphisms, hence
it suffices to describe a set of associators and unitors for $\functor{N}$ and to prove that
the axioms of a pseudofunctor are satisfied. We want to induce such data from the associators
and unitors for the pseudofunctor $\functor{M}$ constructed in Proposition~\ref{prop-06}. Given
any morphism 

\[\underline{f}:=\Big(A'_{\CATA},\operatorname{w}_{\CATA},f_{\CATA}\Big):\,A_{\CATA}\longrightarrow
B_{\CATA}\]
in $\CATA\left[\SETWAinv\right]$, we define an invertible $2$-morphism

\[\varphi\left(\underline{f}\right):\,
\functor{N}_1(\underline{f})\Longrightarrow
\functor{M}_1(\underline{f})\]
as the class of the following diagram:

\[
\begin{tikzpicture}[xscale=2.6,yscale=-0.8]
    \node (A0_2) at (2, 0) {$\functor{F}_0(A'_{\CATA})$};
    \node (A2_0) at (0, 2) {$\functor{F}_0(A_{\CATA})$};
    \node (A2_2) at (2, 2) {$\functor{F}_0(A'_{\CATA})$};
    \node (A2_4) at (4, 2) {$\functor{F}_0(B_{\CATA})$.};
    \node (A4_2) at (2, 4) {$\functor{F}_0(A'_{\CATA})$};

    \node (A2_3) at (2.9, 2) {$\Downarrow\,\pi^{-1}_{\functor{F}_1(f_{\CATA})\circ
      \id_{\functor{F}_0(A'_{\CATA})}}$};
    \node (A2_1) at (1.1, 2) {$\Downarrow\,\pi^{-1}_{\functor{F}_1(\operatorname{w}_{\CATA})
      \circ\id_{\functor{F}_0(A'_{\CATA})}}$};
    
    \path (A4_2) edge [->]node [auto,swap] {$\scriptstyle{\functor{F}_1(f)\circ
      \id_{\functor{F}_0(A'_{\CATA})}}$} (A2_4);
    \path (A0_2) edge [->]node [auto] {$\scriptstyle{\functor{F}_1(f_{\CATA})}$} (A2_4);
    \path (A0_2) edge [->]node [auto,swap] {$\scriptstyle{\functor{F}_1
      (\operatorname{w}_{\CATA})}$} (A2_0);
    \path (A2_2) edge [->]node [auto,swap] {$\scriptstyle{\id_{\functor{F}_0(A'_{\CATA})}}$} (A0_2);
    \path (A2_2) edge [->]node [auto] {$\scriptstyle{\id_{\functor{F}_0(A'_{\CATA})}}$} (A4_2);
    \path (A4_2) edge [->]node [auto] {$\scriptstyle{\functor{F}_1(\operatorname{w}_{\CATA})
      \circ\id_{\functor{F}_0(A'_{\CATA})}}$} (A2_0);
\end{tikzpicture}
\]

Then given any $2$-morphism $\Gamma:\underline{f}^1\Rightarrow\underline{f}^2$ in
$\CATA\left[\SETWAinv\right]$ we have easily the following identity:

\begin{equation}\label{eq-76}
\functor{N}_2(\Gamma)=\varphi\left(\underline{f}^2\right)^{-1}\odot\functor{M}_2
(\Gamma)\odot\varphi\left(\underline{f}^1\right).
\end{equation}

Now given any other morphism $\underline{g}:B_{\CATA}\rightarrow C_{\CATA}$ in $\CATA\left[
\SETWAinv\right]$, we define the associator for $\functor{N}$ relative to the pair $(\underline{g},
\underline{f})$ as the following composition:

\[\Psi^{\functor{N}}_{\underline{g},\underline{f}}:=
\Big(\varphi\left(\underline{g}\right)^{-1}\ast\varphi\left(\underline{f}\right)^{-1}\Big)\odot
\Psi^{\functor{M}}_{\underline{g},\underline{f}}\odot\varphi\left(\underline{g}\circ
\underline{f}\right):\,\functor{N}_1(\underline{g}\circ\underline{f})\Longrightarrow
\functor{N}_1(\underline{g})\circ\functor{N}_1(\underline{f}).\]

Moreover, for any object $A_{\CATA}$, we define the unitor for $\functor{N}$ relative to
$A_{\CATA}$ as the following composition:

\[\Sigma^{\functor{N}}_{A_{\CATA}}:=\Sigma^{\functor{M}}_{A_{\CATA}}
\odot\varphi\left(\id_{A_{\CATA}}\right):\,\functor{N}_1(\id_{A_{\CATA}})\Longrightarrow
\id_{\functor{M}_0(A_{\CATA})}=\id_{\functor{F}_0(A_{\CATA})}=
\id_{\functor{N}_0(A_{\CATA})}.\]

Then we claim that the set of data 

\[\functor{N}:=\Big(\functor{N}_0=\functor{F}_0,\functor{N}_1,
\functor{N}_2,\Psi^{\functor{N}}_{\bullet},\Sigma^{\functor{N}}_{\bullet}\Big)\]
is a pseudofuntor. First of all, we have to verify that $\functor{N}$ preserves
associators, namely that given any pair of morphisms $\underline{f},\underline{g}$ as above and
any morphism $\underline{h}:C_{\CATA}\rightarrow D_{\CATA}$ in $\CATA\left[\SETWAinv\right]$,
the associator

\begin{equation}\label{eq-79}
\Thetaa{\functor{N}_1(\underline{h})}{\functor{N}_1(\underline{g})}
{\functor{N}_1(\underline{f})}:\,
\functor{N}_1(\underline{h})\circ\Big(\functor{N}_1(\underline{g})\circ
\functor{N}_1(\underline{f})\Big)\Longrightarrow\Big(\functor{N}_1
(\underline{h})\circ\functor{N}_1(\underline{g})\Big)\circ
\functor{N}_1(\underline{f})
\end{equation}
in $\CATB\left[\SETWBinv\right]$ coincides with the composition:

\begin{equation}\label{eq-77}
\Big(\Psi^{\functor{N}}_{\underline{h},\underline{g}}\ast i_{\functor{N}_1
(\underline{f})}\Big)\odot\Psi^{\functor{N}}_{\underline{h}\circ\underline{g},
\underline{f}}\odot\functor{N}_2\left(\Thetaa{\underline{h}}{\underline{g}}
{\underline{f}}\right)
\odot\Big(\Psi^{\functor{N}}_{\underline{h},\underline{g}\circ\underline{f}}\Big)^{-1}
\odot\Big(i_{\functor{N}_1(\underline{h})}\ast
\Psi^{\functor{N}}_{\underline{g},\underline{f}}\Big)^{-1}.
\end{equation}

In \eqref{eq-77} we replace the definition of the associators
$\Psi^{\functor{N}}_{\bullet}$ and we use \eqref{eq-76} for $\Gamma=
\Thetaa{\underline{h}}{\underline{g}}{\underline{f}}$. So after some simplifications, \eqref{eq-77}
is equal to the following composition

\begin{gather}
\nonumber\Big(\Big(\varphi\left(\underline{h}\right)^{-1}\ast\varphi\left(\underline{g}\right)^{-1}
 \Big)\ast\varphi\left(\underline{f}\right)^{-1}\Big)\odot \\
\nonumber \odot\Big(\Psi^{\functor{M}}_{\underline{h},\underline{g}}\ast
 i_{\functor{M}_1(\underline{f})}\Big)\odot\Psi^{\functor{M}}_{\underline{h}
  \circ\underline{g},\underline{f}}\odot
 \functor{M}_2\left(\Thetaa{\underline{h}}{\underline{g}}{\underline{f}}\right)\odot
 \Big(\Psi^{\functor{M}}_{\underline{h},\underline{g}\circ\underline{f}}\Big)^{-1}
 \odot\Big(i_{\functor{M}_1(\underline{h})}\ast
 \Psi^{\functor{M}}_{\underline{g},\underline{f}}\Big)^{-1}\odot \\
\label{eq-78}\odot\Big(\varphi\left(\underline{h}\right)\ast\Big(\varphi\left(\underline{g}\right)
 \ast\varphi\left(\underline{f}\right)\Big)\Big).
\end{gather}

Since $\functor{M}$ is a pseudofunctor by Proposition~\ref{prop-06},
then the central line of \eqref{eq-78}
is equal to the associator $\Thetaa{\functor{M}_1(\underline{h})}
{\functor{M}_1(\underline{g})}{\functor{M}_1(\underline{f})}$.
Then by~\cite[Corollary~5.1]{T3} applied to the bicategory $\CATB\left[\SETWBinv\right]$ and
for $\chi(\functor{N}_1(\underline{f}))=\varphi(\underline{f})$ (and analogously
for $\underline{g}$ and $\underline{h}$),
we conclude that \eqref{eq-78} coincides with \eqref{eq-79}.\\

All the remaining axioms of a pseudofunctor for $\functor{N}$ follow easily from
the analogous conditions for the pseudofuntor $\functor{M}$ and from \eqref{eq-76}.
So we have proved that there is a pseudofunctor $\functor{N}$ satisfying the claim
of Corollary~\ref{cor-01}. Then it remains only to define a pseudonatural equivalence
$\partial$
as in the claim. For that, we remark that the set of invertible $2$-morphisms $\{\varphi
(\underline{f})\}_{\underline{f}}$ (indexed on all morphisms $\underline{f}$ of $\CATA\left[
\SETWAinv\right]$) induces a pseudonatural equivalence $\varphi:\functor{N}
\Rightarrow\functor{M}$. Then we define

\[\partial:=\Big(\varphi^{-1}\ast i_{\functor{U}_{\SETWA}}\Big)\odot\zeta:\,\,\functor{U}_{\SETWB}
\circ\functor{F}\Longrightarrow\functor{N}\circ\functor{U}_{\SETWA},\]
where $\zeta:\functor{U}_{\SETWB}
\circ\functor{F}\Rightarrow\functor{M}\circ\functor{U}_{\SETWA}$ is the pseudonatural equivalence
obtained in Proposition~\ref{prop-06}. By~\cite[Proposition~20]{Pr}
the pseudofunctor $\functor{U}_{\SETWA}$ sends each morphism of $\SETWA$ to an internal
equivalence, hence so does the pseudofunctor associated to $i_{\functor{U}_{\SETWA}}$, therefore
also the pseudofunctor associated to $\varphi^{-1}\ast i_{\functor{U}_{\SETWA}}$ sends each morphism
of $\SETWA$ to an internal equivalence. Since also $\zeta$ does the same by
Proposition~\ref{prop-06}, then we have proved that $\partial$ satisfies the claim.
\end{proof}

\begin{rem}
In the proof of~\cite[Theorem~21]{Pr} the unitors and the associators for the induced
pseudofunctor $\functor{M}$ ($\widetilde{\functor{F}}$ in Pronk's notations)
are not described explicitly, nor it is explicitly shown that all the axioms of a pseudofunctor
are satisfied. This is why we have not described them explicitly in the present paper, nor we
have described explicitly the induced unitors and associators for $\functor{N}$. The reader
interested in
such (long, but most of the time straightforward) details can download an additional appendix from
our website (\href{http://matteotommasini.altervista.org}{http://matteotommasini.altervista.org}).
\end{rem}

\begin{rem}\label{rem-05}
Given the pseudofunctor $\functor{M}$ constructed in Proposition~\ref{prop-06},
for each object $A_{\CATA}$ we have

\[\functor{U}_{\SETWB,0}\circ\functor{F}_0(A_{\CATA})=\functor{F}_0(A_{\CATA})=
\functor{M}_0(A_{\CATA})=\functor{M}_0\circ\functor{U}_{\SETWA,0}
(A_{\CATA}).\]

Moreover, for each morphism $f_{\CATA}:A_{\CATA}\rightarrow B_{\CATA}$ we have

\[\functor{U}_{\SETWB,1}\circ\functor{F}_1(f_{\CATA})=\Big(\functor{F}_0(A_{\CATA}),
\id_{\functor{F}_0(A_{\CATA})},\functor{F}_1(f_{\CATA})\Big)\]
and

\[\functor{M}_1\circ\functor{U}_{\SETWA,1}(f_{\CATA})=\Big(\functor{F}_0(A_{\CATA}),
\functor{F}_1(\id_{A_{\CATA}})\circ\id_{\functor{F}_0(A_{\CATA})},
\functor{F}_1(f_{\CATA})\circ\id_{\functor{F}_0(A_{\CATA})}\Big).\]

Then a pseudonatural equivalence $\zeta:\functor{U}_{\SETWB}
\circ\functor{F}\Rightarrow\functor{M}\circ\functor{U}_{\SETWA}$ as in Proposition~\ref{prop-06}
has to be given by the data of an internal equivalence

\[\zeta(A_{\CATA}):\,\,\functor{F}_0(A_{\CATA})\longrightarrow\functor{F}_0(A_{\CATA})\]
in $\CATB\left[\SETWBinv\right]$ for each object $A_{\CATA}$ and by the data of an
invertible $2$-morphism

\begin{gather*}
\zeta(f_{\CATA}):\,\zeta(B_{\CATA})\circ\Big(\functor{F}_0(A_{\CATA}),
 \id_{\functor{F}_0(A_{\CATA})},\functor{F}_1(f_{\CATA})\Big)\Longrightarrow \\
\Longrightarrow\Big(\functor{F}_0(A_{\CATA}),\functor{F}_1(\id_{A_{\CATA}})
 \circ\id_{\functor{F}_0(A_{\CATA})},\functor{F}_1(f_{\CATA})\circ
 \id_{\functor{F}_0(A_{\CATA})}\Big)\circ\zeta(A_{\CATA})
\end{gather*}
for each morphism $f_{\CATA}:A_{\CATA}\rightarrow B_{\CATA}$. Then following the proof
of~\cite[Theorem~21]{Pr}, a possible choice for $\zeta$ is given as follows. First of all,
we set

\begin{equation}\label{eq-62}
\zeta(A_{\CATA}):=\Big(\functor{F}_0(A_{\CATA}),\id_{\functor{F}_0(A_{\CATA})},
\id_{\functor{F}_0(A_{\CATA})}\Big)
\end{equation}
for each object $A_{\CATA}$; then we declare that for each morphism $f_{\CATA}$ as above,
$\zeta(f_{\CATA})$ is the invertible $2$-morphism represented by the following diagram:

\[
\begin{tikzpicture}[xscale=2.2,yscale=-0.8]
    \node (A0_2) at (2, 0) {$\functor{F}_0(A_{\CATA})$};
    \node (A2_0) at (0, 2) {$\functor{F}_0(A_{\CATA})$};
    \node (A2_2) at (2, 2) {$\functor{F}_0(A_{\CATA})$};
    \node (A2_4) at (4, 2) {$\functor{F}_0(B_{\CATA})$,};
    \node (A4_2) at (2, 4) {$\functor{F}_0(A_{\CATA})$};

    \node (A2_1) at (1, 2) {$\Downarrow\,\varepsilon^1_{\CATB}$};
    \node (A2_3) at (3, 2) {$\Downarrow\,\varepsilon^2_{\CATB}$};
    
    \path (A0_2) edge [->]node [auto,swap] {$\scriptstyle{\id_{\functor{F}_0(A_{\CATA})}
      \circ\id_{\functor{F}_0(A_{\CATA})}}$} (A2_0);
    \path (A0_2) edge [->]node [auto] {$\scriptstyle{\id_{\functor{F}_0(B_{\CATA})}
      \circ\functor{F}_1(f_{\CATA})}$} (A2_4);
    \path (A4_2) edge [->]node [auto,swap] {$\scriptstyle{(\functor{F}_1(f_{\CATA})
      \circ\id_{\functor{F}_0(A_{\CATA})})\circ\id_{\functor{F}_0(A_{\CATA})}}$} (A2_4);
    \path (A2_2) edge [->]node [auto,swap] {$\scriptstyle{\id_{\functor{F}_0(A_{\CATA})}}$} (A0_2);
    \path (A2_2) edge [->]node [auto] {$\scriptstyle{\id_{\functor{F}_0(A_{\CATA})}}$} (A4_2);
    \path (A4_2) edge [->]node [auto] {$\scriptstyle{\id_{\functor{F}_0(A_{\CATA})}
      \circ(\functor{F}_1(\id_{A_{\CATA}})\circ\id_{\functor{F}_0(A_{\CATA})})}$} (A2_0);
\end{tikzpicture}
\]
where

\begin{gather*}
\varepsilon^1_{\CATB}:=\Big(i_{\id_{\functor{F}_0(A_{\CATA})}}\ast
 \Big(\Big(\left(\sigma^{\functor{F}}_{A_{\CATA}}\right)^{-1}\ast
 i_{\id_{\functor{F}_0(A_{\CATA})}}\Big)\odot\pi^{-1}_{\id_{\functor{F}_0(A_{\CATA})}}
 \Big)\Big)\ast i_{\id_{\functor{F}_0(A_{\CATA})}},\\
\varepsilon^2_{\CATB}:=\Big(\pi^{-1}_{\functor{F}_1(f_{\CATA})\circ\id_{\functor{F}_0(A_{\CATA})}}
 \odot\pi^{-1}_{\functor{F}_1(f_{\CATA})}\odot
 \upsilon_{\functor{F}_1(f_{\CATA})}\Big)\ast i_{\id_{\functor{F}_0(A_{\CATA})}}.
\end{gather*}

Then using the proof of Corollary~\ref{cor-01}, the induced pseudonatural equivalence $\partial:
\functor{U}_{\SETWB}
\circ\functor{F}\Rightarrow\functor{N}\circ\functor{U}_{\SETWA}$
coincides with \eqref{eq-62} for each object $A_{\CATA}$; for each morphism $f_{\CATA}$ as above,
$\partial(f_{\CATA})$ is represented by the following diagram

\[
\begin{tikzpicture}[xscale=2.2,yscale=-0.8]
    \node (A0_2) at (2, 0) {$\functor{F}_0(A_{\CATA})$};
    \node (A2_0) at (0, 2) {$\functor{F}_0(A_{\CATA})$};
    \node (A2_2) at (2, 2) {$\functor{F}_0(A_{\CATA})$};
    \node (A2_4) at (4, 2) {$\functor{F}_0(B_{\CATA})$,};
    \node (A4_2) at (2, 4) {$\functor{F}_0(A_{\CATA})$};

    \node (A2_1) at (1.2, 2) {$\Downarrow\,\mu^1_{\CATB}$};
    \node (A2_3) at (2.8, 2) {$\Downarrow\,\mu^2_{\CATB}$};
    
    \path (A0_2) edge [->]node [auto,swap] {$\scriptstyle{\id_{\functor{F}_0(A_{\CATA})}
      \circ\id_{\functor{F}_0(A_{\CATA})}}$} (A2_0);
    \path (A0_2) edge [->]node [auto] {$\scriptstyle{\id_{\functor{F}_0(B_{\CATA})}
      \circ\functor{F}_1(f_{\CATA})}$} (A2_4);
    \path (A4_2) edge [->]node [auto,swap] {$\scriptstyle{\functor{F}_1(f_{\CATA})
      \circ\id_{\functor{F}_0(A_{\CATA})}}$} (A2_4);
    \path (A2_2) edge [->]node [auto,swap] {$\scriptstyle{\id_{\functor{F}_0(A_{\CATA})}}$} (A0_2);
    \path (A2_2) edge [->]node [auto] {$\scriptstyle{\id_{\functor{F}_0(A_{\CATA})}}$} (A4_2);
    \path (A4_2) edge [->]node [auto] {$\scriptstyle{\id_{\functor{F}_0(A_{\CATA})}
      \circ\functor{F}_1(\id_{A_{\CATA}})}$} (A2_0);
\end{tikzpicture}
\]
where

\begin{gather*}
\mu^1_{\CATB}:=\Big(i_{\id_{\functor{F}_0(A_{\CATA})}}\ast
 \left(\sigma^{\functor{F}}_{A_{\CATA}}\right)^{-1}\Big)\ast i_{\id_{\functor{F}_0(A_{\CATA})}},\quad
\mu^2_{\CATB}:=\Big(\pi_{\functor{F}_1(f_{\CATA})}^{-1}\odot
 \upsilon_{\functor{F}_1(f_{\CATA})}\Big)\ast i_{\id_{\functor{F}_0(A_{\CATA})}}.
\end{gather*}

In particular, if $\CATB$ is a $2$-category and $\functor{F}$ preserves $1$-identities, then
$\partial$ is the identical natural transformation of the pseudofunctor
$\functor{U}_{\SETWB}\circ\functor{F}=\functor{N}\circ\functor{U}_{\SETWA}$.
\end{rem}

\begin{cor}\label{cor-04}
Let us fix any pair $(\CATC,\SETW)$ satisfying conditions \emphatic{(BF)}.
Let us fix any pair of choices \emphatic{\hyperref[C]{C}}$^m(\SETW)$ for $m=1,2$
and let us denote by $\CATC^m\left[\SETWinv\right]$ and $\functor{U}^m_{\SETW}$ for $m=1,2$ the
associated bicategories of fractions and universal pseudofunctors.
Then there is a pseudofunctor

\begin{equation}\label{eq-80}
\functor{Q}:\,\CATC^1\left[\SETWinv\right]\longrightarrow\CATC^2\left[\SETWinv\right]
\end{equation}
that is the identity on objects, morphisms and $2$-morphisms. Moreover, there is a pseudonatural
equivalence

\[\phi:\,\functor{U}^2_{\SETW}\Longrightarrow\functor{Q}\circ\functor{U}^1_{\SETW}\quad\textrm{in}
\quad\operatorname{Hom}'_{\SETW}\left(\CATC,\CATC^2\left[\SETWinv\right]\right).\]
\end{cor}

The existence of an equivalence of bicategories as in \eqref{eq-80} and of $\phi$
is an obvious consequence
of the universal property of bicategories of fractions (see~\cite[Theorem~21]{Pr}). However,
in~\cite[Theorem~21]{Pr} there are no explicit descriptions on the behavior of $\functor{Q}$
on objects, morphisms and $2$-morphisms, so Corollary~\ref{cor-04} a priori is not trivial.

\begin{proof}
Let us fix any set of choices \hyperref[C]{C}$(\SETW)$ satisfying condition (\hyperref[C3]{C3}).
Then we apply Corollary~\ref{cor-01} to the case when:

\begin{itemize}
 \item $(\CATA,\SETWA):=(\CATC,\SETW)$ and the choices for this pair are given by
  \hyperref[C]{C}$\,^1(\SETW)$; the associated bicategory is then $\CATC^1\left[\SETWinv\right]$;
 \item $(\CATB,\SETWB):=(\CATC,\SETW)$ and the choices for this pair are given by
  \hyperref[C]{C}$(\SETW)$; we denote the associated bicategory by $\CATC\left[\SETWinv\right]$;
 \item the pseudofunctor $\functor{F}$ is the identity of $\CATC$.
\end{itemize}

Then there are an induced pseudofunctor

\[\functor{N}^1:\,\CATC^1\left[\SETWinv\right]\longrightarrow\CATC\left[\SETWinv\right]\]
given on objects, morphisms and $2$-morphisms as the identity (its associators are induced by the
choices \hyperref[C]{C}$\,^1(\SETW)$ and \hyperref[C]{C}$(\SETW)$) and a pseudonatural equivalence
of pseudofunctors

\[\partial^1:\functor{U}_{\SETW}\Longrightarrow\functor{N}^1\circ\functor{U}^1_{\SETW}
\quad\textrm{in}\quad\operatorname{Hom}'_{\SETW}\left(\CATC,\CATC\left[\SETWinv\right]\right).\]

Analogously, there are an induced pseudofunctor

\[\functor{N}^2:\,\CATC^2\left[\SETWinv\right]\longrightarrow\CATC\left[\SETWinv\right]\]
given on objects, morphisms and $2$-morphisms as the identity, and a pseudonatural equivalence

\[\partial^2:\functor{U}_{\SETW}\Longrightarrow\functor{N}^2\circ\functor{U}^2_{\SETW}
\quad\textrm{in}\quad\operatorname{Hom}'_{\SETW}\left(\CATC,\CATC\left[\SETWinv\right]\right).\]

Since the objects, morphisms and 
$2$-morphisms of the source of $\functor{N}^2$ are the same as those of its target, we have that
actually $\functor{N}^2$ is a bijection on objects, morphisms and $2$-morphisms. So
it is easy to construct a pseudofunctor

\[\widetilde{\functor{N}}^2:\,\CATC\left[\SETWinv\right]\longrightarrow\CATC^2\left[\SETWinv\right]\]
that is an inverse for $\functor{N}^2$: it is described as the identity on objects,
morphisms and $2$-morphisms; its associators and unitors are induced by the inverses of the
associators and unitors for $\functor{N}^2$. This induces also a pseudonatural equivalence

\[\tau:\,\functor{U}^2_{\SETW}\Longrightarrow\widetilde{\functor{N}}^2\circ\functor{U}_{\SETW}
\quad\textrm{in}\quad\operatorname{Hom}'_{\SETW}\left(\CATC,\CATC^2\left[\SETWinv\right]\right).\]

Then we define $\functor{Q}:=\widetilde{\functor{N}}^2\circ\functor{N}^1$ and we set

\[\phi:=\thetaa{\widetilde{\functor{N}}^2}{\functor{N}^1}{\functor{U}^1_{\SETW}}\odot
\Big(i_{\widetilde{\functor{N}}^2}\ast\partial^1\Big)\odot\tau:\,\,\functor{U}^2_{\SETW}
\Longrightarrow\functor{Q}\circ\functor{U}^1_{\SETW}\quad\textrm{in}
\quad\operatorname{Hom}'_{\SETW}\left(\CATC,\CATC^2\left[\SETWinv\right]\right).\]
\end{proof}

\begin{proof}[Proof of Theorem~\ref{theo-03}.]
First of all, we prove part (B) of the statement, so let us assume that $\functor{F}_1(\SETWA)
\subseteq\SETWB$. Let us suppose that the bicategory $\CATB\left[\SETWBinv\right]$ is induced
by a set of choices \hyperref[C]{C}$(\SETW)$. If \hyperref[C]{C}$(\SETW)$ satisfies condition
(\hyperref[C3]{C3}), then (B) coincides with Corollary~\ref{cor-01}. Otherwise,
let us fix another set of choices \hyperref[C]{C}$\,'(\SETWB)$ satisfying condition
(\hyperref[C3]{C3}),
let us denote by $\CATB'\left[\SETWBinv\right]$ the associated bicategory of fractions
and by $\functor{U}'_{\SETWB}$ the associated universal pseudofunctor.
Then there are a pseudofunctor

\[\functor{N}:\,\CATA\left[\SETWAinv\right]\longrightarrow\CATB'\left[\SETWBinv\right]\]
and a pseudonatural equivalence $\partial:\functor{U}'_{\SETWB}\circ\functor{F}\Rightarrow
\functor{N}\circ\functor{U}_{\SETWA}$ satisfying Corollary~\ref{cor-01}. Now we apply
Corollary~\ref{cor-04} for \hyperref[C]{C}$^1(\SETW):=$\hyperref[C]{C}$\,'(\SETW)$ and
\hyperref[C]{C}$^2(\SETW):=$\hyperref[C]{C}$(\SETW)$. So there are a pseudofunctor

\[\functor{Q}:\CATB'\left[\SETWBinv\right]\longrightarrow\CATB\left[\SETWBinv\right]\]
that is the identity on objects, morphisms and $2$-morphisms, and a pseudonatural equivalence

\[\phi:\,\functor{U}_{\SETWB}\Longrightarrow\functor{Q}\circ\functor{U}'_{\SETWB}
\quad\textrm{in}\quad\operatorname{Hom}'_{\SETW}\left(\CATC,\CATC\left[\SETWinv\right]\right).\]

We set $\widetilde{\functor{G}}:=\functor{Q}\circ\functor{N}:\CATA\left[\SETWAinv\right]
\rightarrow\CATB\left[\SETWBinv\right]$. Since $\functor{Q}$
is the identity on objects, morphisms and $2$-morphisms, then the description of
$\widetilde{\functor{G}}$ on such data coincides with the description of $\functor{N}$ on
the same data (see Corollary~\ref{cor-01}), so conditions (II), (III) and (IV) are satisfied.
Then we define:

\[\widetilde{\kappa}:=\thetaa{\functor{Q}}{\functor{N}}{\functor{U}_{\SETWA}}\odot
\Big(i_{\functor{Q}}\ast\partial\Big)\odot\thetab{\functor{Q}}{\functor{U}'_{\SETWB}}{\functor{F}}
\odot\Big(\phi\ast i_{\functor{F}}\Big):\,\,\functor{U}_{\SETWB}\circ\functor{F}\Longrightarrow
\widetilde{\functor{G}}\circ\functor{U}_{\SETWA}.\]

Using the properties of $\partial$ and $\phi$ already stated in Corollaries~\ref{cor-01}
and~\ref{cor-04}, we conclude that $\widetilde{\kappa}$ satisfies condition (I). This suffices
to prove part (B) of Theorem~\ref{theo-03}.\\

Now let us prove also part (A), so let us assume only that $\functor{F}_1(\SETWA)\subseteq\SETWBsat$.
Since $(\CATB,\SETWB)$ satisfies conditions (BF), then by Lemma~\ref{lem-05}
we have that also $(\CATB,\SETWBsat)$ satisfies conditions (BF). Therefore, we can
apply part (B) to the case when we replace $(\CATB,\SETWB)$ by $(\CATB,\SETWBsat)$. So there
are a pseudofunctor

\[\widetilde{\functor{G}}:\,\CATA\Big[\SETWAinv\Big]\longrightarrow\CATB'\Big[\SETWBsatinv
\Big]\]
and a pseudonatural equivalence $\widetilde{\kappa}:\functor{U}_{\SETWBsat}\circ\functor{F}
\Rightarrow\widetilde{\functor{G}}\circ\functor{U}_{\SETWA}$, such that

\begin{itemize}
 \item the pseudofunctor $\mu_{\widetilde{\kappa}}:\CATA\rightarrow\operatorname{Cyl}
  \left(\CATB\left[\SETWBsatinv\right]\right)$ associated to $\widetilde{\kappa}$ sends each morphism
  of $\SETWA$ to an internal equivalence;
\item conditions (II), (III) and (IV) hold.
\end{itemize}
\end{proof}

\begin{rem}\label{rem-06}
In the case when $\functor{F}_1(\SETWA)\cap(\SETWBsat\smallsetminus\SETWB)\neq\varnothing$, then a
pair
$(\functor{G},\kappa)$ as in Theorem~\ref{theo-04}(iv) can be obtained in any of the following $2$
ways. Both give the same pseudofunctor (up to pseudonatural equivalences), and in both cases such a
pseudofunctor is very complicated to study directly. The first possibility is to follow the proof
of~\cite[Theorem~21]{Pr}, as we did in the proof of Proposition~\ref{prop-06}. In this case, it is
much more difficult to give a set of data $(\functor{P}(\operatorname{w}_{\CATA}),
\Delta(\operatorname{w}_{\CATA}),\Xi(\operatorname{w}_{\CATA}))$ for each $\operatorname{w}_{\CATA}
\in\SETWA$; moreover in general one cannot express the composition of diagram \eqref{eq-02} in a
simple form. The second possibility is given as follows: first of all we consider
the pair $(\widetilde{\functor{G}},\widetilde{\kappa})$ described in Theorem~\ref{theo-03}(A).
Then we consider the pair

\[\functor{H}_{\CATB}:\,\CATB\Big[\SETWBsatinv\Big]\longrightarrow\CATB\Big[
\SETWBinv\Big],\quad\quad
\tau_{\CATB}:\,\functor{U}_{\SETWB}\Longrightarrow\functor{H}_{\CATB}\circ\functor{U}_{\SETWBsat}\]
associated to $(\CATB,\SETWB)$ by Proposition~\ref{prop-01} and we set:

\begin{itemize}
 \item $\functor{G}:=\functor{H}_{\CATB}\circ\widetilde{\functor{G}}$;
 \item $\kappa:=\thetaa{\functor{H}_{\CATB}}{\widetilde{\functor{G}}}
  {\functor{U}_{\SETWA}}\odot(i_{\functor{H}_{\CATB}}\ast\widetilde{\kappa})\odot
  \thetab{\functor{H}_{\CATB}}{\functor{U}_{\SETWBsat}}{\functor{F}}\odot 
  (\tau_{\CATB}\ast i_{\functor{F}})$.
\end{itemize}

In this case, the complexity of the pseudofunctor $\functor{G}$ is hidden in the complexity
of $\functor{H}_{\CATB}$, that was also implicitly obtained using~\cite[Theorem~21]{Pr}: indeed
$\functor{H}_{\CATB}$ was obtained in Proposition~\ref{prop-01}, that uses Theorem~\ref{theo-02},
that is essentially part of~\cite[Theorem~21]{Pr}.
Therefore, also in this case in general it is not possible
to give a simple description of $\functor{G}$.
\end{rem}

Now we are ready to prove the third main result of this paper.

\begin{proof}[Proof of Corollary~\ref{cor-03}.]
Let us fix any pair $(\functor{G},\kappa)$ as in Theorem~\ref{theo-04}(iv). By that theorem, we have
$\functor{F}_1(\SETWA)\subseteq\SETWBsat$. This implies that there
are a pseudofunctor

\[\widetilde{\functor{G}}:\,\CATA\Big[\SETWAinv\Big]\longrightarrow\CATB\Big[\SETWBsatinv\Big],\]
described as in Theorem~\ref{theo-04}(A), and a pseudonatural equivalence $\widetilde{\kappa}:
\functor{U}_{\SETWBsat}\circ\functor{F}\Rightarrow\widetilde{\functor{G}}\circ\functor{U}_{\SETWA}$,
that is an internal equivalence in $\operatorname{Hom}'_{\SETWA}\left(\CATA,\CATB\left[\SETWBsatinv
\right]\right)$. By Proposition~\ref{prop-01} applied to the pair
$(\CATB,\SETWB)$, there are an equivalence of bicategories

\[\functor{H}_{\CATB}:\,\CATB\Big[\SETWBsatinv\Big]\longrightarrow\CATB\Big[\SETWBinv\Big]\]
and a pseudonatural equivalence of pseudofunctors

\[\tau_{\CATB}:\,\functor{U}_{\SETWB}\Longrightarrow\functor{H}_{\CATB}\circ
\functor{U}_{\SETWBsat}\]
belonging to $\operatorname{Hom}'_{\SETWBsat}(\CATB,\CATB\left[\SETWBinv\right])$.
Now let us consider the following composition of pseudonatural equivalences of pseudofunctors:

\begin{gather*}
\eta:=\thetaa{\functor{H}_{\CATB}}{\widetilde{\functor{G}}}{\functor{U}_{\SETWA}}\odot
 \Big(i_{\functor{H}_{\CATB}}\ast\widetilde{\kappa}\Big)\odot
 \thetab{\functor{H}_{\CATB}}{\functor{U}_{\SETWBsat}}{\functor{F}}\odot\Big(\tau_{\CATB}
 \ast i_{\functor{F}}\Big)\odot\kappa^{-1}: \\
\functor{G}\circ\functor{U}_{\SETWA}\Longrightarrow
 (\functor{H}_{\CATB}\circ\widetilde{\functor{G}})\circ\functor{U}_{\SETWA}.
\end{gather*}

Now:

\begin{itemize}
 \item $\kappa$ belongs to $\operatorname{Hom}'_{\SETWA}\left(\CATA,\CATB\left[\SETWBinv\right]\right)$;
 \item $\functor{F}$ belongs to $\operatorname{Hom}(\CATA,\CATB)$ and is such that
  $\functor{F}_1(\SETWA)\subseteq\SETWBsat$; moreover $\tau_{\CATB}$ is a morphism in
  $\operatorname{Hom}'_{\SETWBsat}\left(\CATB,\CATB\left[\SETWBinv\right]\right)$, so $\tau_{\CATB}\ast
  i_{\functor{F}}$ is a morphism in $\operatorname{Hom}'_{\SETWA}\left(\CATA,\CATB\left[\SETWBinv\right]
  \right)$;
 \item $\widetilde{\kappa}$ is a morphism in $\operatorname{Hom}'_{\SETWA}\left(\CATA,\CATB\left[
  \SETWBsatinv,\right]\right)$, so $i_{\functor{H}_{\CATB}}\ast\widetilde{\kappa}$ is a morphism in
  $\operatorname{Hom}'_{\SETWA}(\CATA,\CATB\left[\SETWBinv\right])$.
\end{itemize}

We recall that by~Theorem~\ref{theo-01} we have an equivalence of bicategories

\begin{equation}\label{eq-47}
\functor{E}:\,\operatorname{Hom}\Big(\CATA\left[\SETWAinv\right],\CATB\left[\SETWBinv\right]\Big)
\longrightarrow\operatorname{Hom}_{\SETWA}\Big(\CATA,\CATB\left[\SETWBinv\right]\Big),
\end{equation}
given for each object

\[\functor{G}:\CATA\left[\SETWAinv\right]\longrightarrow\CATB\left[\SETWBinv\right]\]
in $\operatorname{Hom}(\CATA\left[\SETWAinv\right],\CATB\left[\SETWBinv\right])$ by

\[\functor{E}(\functor{G}):=\functor{G}\circ\functor{U}_{\SETWA}:\,\CATA\longrightarrow\CATB
\left[\SETWBinv\right].\]

So we have defined an internal equivalence

\[\eta:\,\functor{E}(\functor{G})\Longrightarrow\functor{E}(\functor{H}_{\CATB}\circ
\widetilde{\functor{G}})\]
in the bicategory $\operatorname{Hom}'_{\SETWA}(\CATA,\CATB\left[\SETWBinv\right])\subseteq
\operatorname{Hom}_{\SETWA}(\CATA,\CATB\left[\SETWBinv\right])$. Since
$\functor{E}$ is an equivalence of bicategories, this implies that there is an internal
equivalence from the pseudofunctor
$\functor{G}$ to $\functor{H}_{\CATB}\circ\widetilde{\functor{G}}$ in the bicategory
$\operatorname{Hom}(\CATA\left[\SETWAinv\right],\CATB\left[\SETWBinv\right])$. So
by Lemma~\ref{lem-01} there is also
an internal equivalence from $\functor{G}$ to $\functor{H}_{\CATB}\circ\widetilde{\functor{G}}$
in the bicategory $\operatorname{Hom}'(\CATA\left[\SETWAinv\right],\CATB\left[\SETWBinv\right])$,
i.e.\ a pseudonatural equivalence of pseudofunctors

\[\delta:\,\functor{G}\Longrightarrow\functor{H}_{\CATB}\circ\widetilde{\functor{G}}.\]

Now $\functor{H}_{\CATB}$ is an equivalence of bicategories, so $\functor{H}_{\CATB}
\circ\widetilde{\functor{G}}$ is an equivalence of
bicategories if and only if $\widetilde{\functor{G}}$ is so. So by Lemma~\ref{lem-07} applied to
$\delta$,
we conclude that $\functor{G}$ is an equivalence of bicategories if and only if
$\widetilde{\functor{G}}$ is so.
\end{proof}

As a consequence of Corollary~\ref{cor-03}, we have the following necessary (but in general not
sufficient) condition in order to have an induced equivalence between bicategories of fractions.

\begin{cor}
Let us fix any $2$ pairs $(\CATA,\SETWA)$ and $(\CATB,\SETWB)$, both satisfying conditions
\emphatic{(BF)}, any pseudofunctor $\functor{F}:\CATA\rightarrow\CATB$ and let us suppose that
there is a pair $(\functor{G},\kappa)$ as in \emphatic{Theorem~\ref{theo-04}(iv)}, such that
$\functor{G}:\CATA\left[\SETWAinv\right]\rightarrow\CATB\left[\SETWBinv\right]$ is an equivalence
of bicategories. Then $\functor{F}_1^{\,-1}(\SETWAsat)=\SETWBsat$.
\end{cor}

\begin{proof}
By Theorem~\ref{theo-04}(iv), we have that $\functor{F}_1(\SETWAsat)\subseteq\SETWBsat$, hence
$\SETWAsat\subseteq\functor{F}_1^{\,-1}(\SETWBsat)$, so we need only to prove the other inclusion.\\

So let us fix any morphism $\operatorname{w}_{\CATA}:A_{\CATA}\rightarrow B_{\CATA}$ such that
$\functor{F}_1(\operatorname{w}_{\CATA})\in\SETWBsat$ and let us prove that
$\operatorname{w}_{\CATA}$ belongs to $\SETWAsat$. Since $\id_{A_{\CATA}}$ belongs to $\SETWA$
by (BF1), then $\functor{F}_1(\id_{A_{\CATA}})$ belongs to $\functor{F}_1(\SETWAsat)\subseteq
\SETWBsat$.
Moreover, by hypothesis $\functor{F}_1(\operatorname{w}_{\CATA})$ belongs to $\SETWBsat$.
Therefore, if we apply Proposition~\ref{prop-03}(i) and Corollary~\ref{cor-02} to $(\CATB,\SETWBsat)$,
we get that the following is an internal equivalence in $\CATB\left[\SETWBsatinv\right]$:

\begin{equation}\label{eq-72}
\begin{tikzpicture}[xscale=2.7,yscale=-1.2]
    \node (A0_0) at (0, 0) {$\functor{F}_0(A_{\CATA})$};
    \node (A0_1) at (1, 0) {$\functor{F}_0(A_{\CATA})$};
    \node (A0_2) at (2, 0) {$\functor{F}_0(B_{\CATA})$.};
    \path (A0_1) edge [->]node [auto,swap] {$\scriptstyle{\functor{F}_1(\id_{A_{\CATA}})}$} (A0_0);
    \path (A0_1) edge [->]node [auto]
      {$\scriptstyle{\functor{F}_1(\operatorname{w}_{\CATA})}$} (A0_2);
\end{tikzpicture}
\end{equation}

This morphism is the image of the morphism

\begin{equation}\label{eq-73}
\begin{tikzpicture}[xscale=2.7,yscale=-1.2]
    \node (A0_0) at (0, 0) {$A_{\CATA}$};
    \node (A0_1) at (1, 0) {$A_{\CATA}$};
    \node (A0_2) at (2, 0) {$B_{\CATA}$};
    \path (A0_1) edge [->]node [auto,swap] {$\scriptstyle{\id_{A_{\CATA}}}$} (A0_0);
    \path (A0_1) edge [->]node [auto] {$\scriptstyle{\operatorname{w}_{\CATA}}$} (A0_2);
\end{tikzpicture}
\end{equation}
via the pseudofunctor

\[\widetilde{\functor{G}}:\,\CATA\Big[\SETWAinv\Big]\longrightarrow\CATB
\Big[\SETWBsatinv\Big]\]
described in Theorem~\ref{theo-04}(A). By Corollary~\ref{cor-03}, $\widetilde{\functor{G}}$ is
an equivalence of bicategories; since \eqref{eq-72} is an internal equivalence
then we conclude that \eqref{eq-73}
is an internal equivalence. By Corollary~\ref{cor-02} applied to $(\CATA,\SETWA)$, we conclude
that $\operatorname{w}_{\CATA}$ belongs to $\SETWA$.
\end{proof}

As we said above, the previous condition is only a necessary one.
In the next paper of this series (\cite{T5}) we are going to find a set of
necessary and sufficient conditions
such that $\widetilde{\functor{G}}$ is an equivalence of bicategories. Combining this with
Corollary~\ref{cor-03}, we will get necessary and sufficient conditions
such that $\functor{G}$ is an equivalence of bicategories for any pseudofunctor $\functor{G}$
satisfying the conditions of Theorem~\ref{theo-04}(iv).

\section{Applications to Morita equivalences of \'etale groupoids}\label{sec-01}
In this section we apply some of the previous results about saturations to the class
of Morita equivalences of \'etale differentiable
groupoids. We denote any (\'etale) Lie groupoid by $(\groupname{X}{})$ (omitting the structure
morphisms $m,i$ and $e$ only for clarity of exposition) or simply $\groupoidtot{X}{}$, and
any morphism of Lie groupoids either by $\groupoidmap{\phi}{}$ or by $\groupoidmaptot{\phi}{}$.\\

We recall (see~\cite[\S~2.4]{M}) that a morphism $\groupoidmaptot{\phi}{}:\groupoidtot{Y}{}
\rightarrow\groupoidtot{X}{}$
between Lie groupoids is a \emph{weak equivalence} (also known as \emph{Morita equivalence} or
\emph{essential equivalence}) if and only if the following $2$ conditions hold:

\begin{enumerate}[({V}1)]
 \item\label{V1} the smooth map $t\circ\pi^1:\fiber{\groupoid{X}_1}{s}{\phi_0}{\groupoid{Y}_0}
  \rightarrow\groupoid{X}_0$ is a surjective submersion (here $\pi^1$ is the projection
  $\fiber{\groupoid{X}_1}{s}{\phi_0}{\groupoid{Y}_0}\rightarrow\groupoid{X}_1$ and the fiber product
  is a manifold since $s$ is a submersion by definition of Lie groupoid);
 \item\label{V2} the following square is cartesian (it is commutative by definition of groupoid):
 
  \begin{equation}\label{eq-09}
  \begin{tikzpicture}[xscale=1.8,yscale=-0.8]
    \node (A0_0) at (0, 0) {$\groupoid{Y}_1$};
    \node (A0_2) at (2, 0) {$\groupoid{X}_1$};
    \node (A2_0) at (0, 2) {$\groupoid{Y}_0\times\groupoid{Y}_0$};
    \node (A2_2) at (2, 2) {$\groupoid{X}_0\times\groupoid{X}_0$.};
    
    \path (A0_0) edge [->]node [auto] {$\scriptstyle{\phi_1}$} (A0_2);
    \path (A0_0) edge [->]node [auto,swap] {$\scriptstyle{(s,t)}$} (A2_0);
    \path (A0_2) edge [->]node [auto] {$\scriptstyle{(s,t)}$} (A2_2);
    \path (A2_0) edge [->]node [auto,swap] {$\scriptstyle{(\phi_0\times\phi_0)}$} (A2_2);
  \end{tikzpicture}
  \end{equation}
\end{enumerate}

We denote by $\EGpd$ the $2$-category of \emph{\'etale} groupoids (i.e.\ Lie groupoids
$\groupoidtot{X}{}$ such that $\operatorname{dim}\groupoid{X}_0=\operatorname{dim}\groupoid{X}_1$,
equivalently such that either $s$ or $t$ are \'etale smooth maps, see~\cite[Exercise~5.16(2)]{MM})
and by $\WEGpd$ the class of all Morita equivalences between such objects. We recall that
by~\cite[Corollary~43]{Pr}:

\begin{itemize}
 \item the pair $(\EGpd,\WEGpd)$ satisfies conditions (BF);
 \item the induced bicategory of fractions $\EGpd\left[\WEGpdinv\right]$ is equivalent to the
  $2$-category of differentiable stacks.
\end{itemize}

We denote by $\PEGpd$ the $2$-category of proper, \'etale groupoids and by $\PEEGpd$ the
$2$-category of proper, effective, \'etale groupoids (see~\cite{MM}); moreover we
denote by $\WPEGpd$ and $\WPEEGpd$ the classes of all Morita equivalences in such $2$-categories.
Such classes satisfy again conditions (BF),
so a right bicalculus of fractions can be performed also in such frameworks.

\begin{prop}\label{prop-05}
Let us fix any triple of morphisms of \'etale groupoids as follows

\[\groupoidmaptot{\xi}{}:\,\groupoidtot{U}{}\longrightarrow\groupoidtot{Z}{},\quad\quad\quad
\groupoidmaptot{\psi}{}:\,\groupoidtot{Z}{}\longrightarrow\groupoidtot{Y}{},\quad\quad\quad
\groupoidmaptot{\phi}{}:\,\groupoidtot{Y}{}\longrightarrow\groupoidtot{X}{}\]
and let us suppose that both $\groupoidmaptot{\phi}{}\circ\groupoidmaptot{\psi}{}$ and
$\groupoidmaptot{\psi}{}\circ\groupoidmaptot{\xi}{}$ are Morita equivalences. Then
$\groupoidmaptot{\phi}{}$ is a Morita equivalence. Therefore, the \emphatic{(}right\emphatic{)}
saturation of the class $\WEGpd$ is exactly the same same class. The same holds for $\WPEGpd$
and $\WPEEGpd$.
\end{prop}

\begin{proof}
By~\cite[Exercise~5.16(4)]{MM} the following smooth maps are \'etale:

\begin{gather*}
\phi_0\circ\psi_0:\,\groupoid{Z}_0\longrightarrow\groupoid{X}_0,\quad\quad\phi_1\circ\psi_1:\,
 \groupoid{Z}_1\longrightarrow\groupoid{X}_1, \\
\psi_0\circ\xi_0:\,\groupoid{U}_0\longrightarrow\groupoid{Y}_0,\quad\quad\psi_1\circ\xi_1:\,
 \groupoid{U}_1\longrightarrow\groupoid{Y}_1.
\end{gather*}

Since $\phi_0\circ\psi_0$ is \'etale, then $\phi_0$ is a submersion; analogously we get that
$\psi_0$ is a submersion. Since their composition is \'etale, this implies that both $\phi_0$ and
$\psi_0$ are \'etale. In the same way, we prove that $\phi_1$ and $\psi_1$ are \'etale.\\

Then let us consider the following cartesian diagrams:

\[
\begin{tikzpicture}[xscale=1.8,yscale=-1.0]
    \node (B0_0) at (-2.4, 0) {$\fiber{\groupoid{X}_1}{s}{\phi_0\circ\psi_0}{\groupoid{Z}_0}$};
    \node (A0_0) at (-0.2, 0) {$\fiber{\Big(\fiber{\groupoid{X}_1}{s}{\phi_0}{\groupoid{Y}_0}\Big)}
      {\pi^2}{\psi_0}{\groupoid{Z}_0}$};
    \node (A0_2) at (2, 0) {$\fiber{\groupoid{X}_1}{s}{\phi_0}{\groupoid{Y}_0}$};
    \node (A0_4) at (3.5, 0) {$\groupoid{X}_1$};
    \node (A2_0) at (-0.2, 2) {$\groupoid{Z}_0$};
    \node (A2_2) at (2, 2) {$\groupoid{Y}_0$};
    \node (A2_4) at (3.5, 2) {$\groupoid{X}_0$.};

    \node (A1_1) at (1, 1) {$\square$};
    \node (A1_3) at (2.8, 1) {$\square$};
    
    \path (B0_0) edge [->]node [auto] {$\scriptstyle{}$} node [	sloped]
      {$\scriptstyle{\widetilde{\ \ \ }}$} (A0_0);
    \path (A2_0) edge [->]node [auto,swap] {$\scriptstyle{\psi_0}$} (A2_2);
    \path (A0_4) edge [->]node [auto] {$\scriptstyle{s}$} (A2_4);
    \path (A0_2) edge [->]node [auto] {$\scriptstyle{\pi^2}$} (A2_2);
    \path (A2_2) edge [->]node [auto,swap] {$\scriptstyle{\phi_0}$} (A2_4);
    \path (A0_2) edge [->]node [auto] {$\scriptstyle{\pi^1}$} (A0_4);
    \path (A0_0) edge [->]node [auto] {$\scriptstyle{\eta^1}$} (A0_2);
    \path (A0_0) edge [->]node [auto,swap] {$\scriptstyle{\eta^2}$} (A2_0);
\end{tikzpicture}
\]

Since $\groupoidmaptot{\phi}{}\circ\groupoidmaptot{\psi}{}$ is a Morita equivalence, then by
(\hyperref[V1]{V1}) the map $t\circ\pi^1\circ\eta^1$ is surjective, so $t\circ\pi^1$ is surjective.
Since $\phi_0$ is \'etale, so is $\pi^1$; moreover, $t$ is \'etale.
Therefore, $t\circ\pi^1$ is \'etale, hence a submersion. So we have proved that (\hyperref[V1]{V1})
holds for $\groupoidmaptot{\phi}{}$.\\

Now let us prove that (\hyperref[V2]{V2}) holds for $\groupoidmaptot{\phi}{}$. By definition of Lie
groupoid, we know that diagram \eqref{eq-09} is commutative, so we have a unique induced smooth map
$\gamma$, making the following diagram commute:

\begin{equation}\label{eq-90}
\begin{tikzpicture}[xscale=2.5,yscale=-1.2]
    \node (A0_0) at (0, 0) {$\groupoid{Y}_1$};
    \node (A1_1) at (1, 1)
      {$\fiber{\groupoid{X}_1}{(s,t)}{\phi_0\times\phi_0}{(\groupoid{Y}_0\times\groupoid{Y}_0)}$};
    \node (A1_3) at (3, 1) {$\groupoid{X}_1$};
    \node (A3_1) at (1, 3) {$\groupoid{Y}_0\times\groupoid{Y}_0$};
    \node (A3_3) at (3, 3) {$\groupoid{X}_0\times\groupoid{X}_0$.};
    
    \node (A1_2) at (1.5, 0.6) {$\curvearrowright$};
    \node (A2_2) at (2, 2) {$\square$};
    \node (A2_1) at (0.6, 1.5) {$\curvearrowright$};    

    \path (A0_0) edge [->,bend right=15]node [auto] {$\scriptstyle{\phi_1}$} (A1_3);
    \path (A3_1) edge [->]node [auto,swap] {$\scriptstyle{\phi_0\times\phi_0}$} (A3_3);
    \path (A0_0) edge [->,dashed]node [auto] {$\scriptstyle{\gamma}$} (A1_1);
    \path (A0_0) edge [->,bend left=15]node [auto,swap] {$\scriptstyle{(s,t)}$} (A3_1);
    \path (A1_1) edge [->]node [auto] {$\scriptstyle{\tau^2}$} (A3_1);
    \path (A1_1) edge [->]node [auto,swap] {$\scriptstyle{\tau^1}$} (A1_3);
    \path (A1_3) edge [->]node [auto] {$\scriptstyle{(s,t)}$} (A3_3);
\end{tikzpicture}
\end{equation}

Since $\phi_0$ is \'etale, so is $\phi_0\times\phi_0$, hence so is $\tau^1$. Since also $\phi_1$ is
\'etale, then we conclude that the smooth map $\gamma$ is also \'etale.\\

Now we claim that $\gamma$ is surjective. So let us fix any point $x_1\in\groupoid{X}_1$ and
any point $(y_0,y'_0)$ in $\groupoid{Y}_0\times\groupoid{Y}_0$, such that $(s,t)(x_1)=(\phi_0(y_0),
\phi_0(y'_0))$.
We need to prove that there is a point $y_1\in\groupoid{Y}_1$, such that $\phi_1(y_1)=x_1$,
and $(s,t)(y_1)=(y_0,y'_0)$. In order to prove such a claim, let us consider the following
fiber product

\[\begin{tikzpicture}[xscale=2.2,yscale=-1.0]
    \node (A0_0) at (0, 0) {$\fiber{\groupoid{Y}_1}{s}{\psi_0\circ\xi_0}{\groupoid{U}_0}$};
    \node (A0_2) at (2, 0) {$\groupoid{Y}_1$};
    \node (A2_0) at (0, 2) {$\groupoid{U}_0$};
    \node (A2_2) at (2, 2) {$\groupoid{Y}_0$.};
    
    \node (A1_1) at (1, 1) {$\square$};
        
    \path (A2_0) edge [->]node [auto,swap] {$\scriptstyle{\psi_0\circ\xi_0}$} (A2_2);
    \path (A0_0) edge [->]node [auto,swap] {$\scriptstyle{\sigma^2}$} (A2_0);
    \path (A0_2) edge [->]node [auto] {$\scriptstyle{s}$} (A2_2);
    \path (A0_0) edge [->]node [auto] {$\scriptstyle{\sigma^1}$} (A0_2);
\end{tikzpicture}\]

Since $\groupoidmaptot{\psi}{}\circ\groupoidmaptot{\xi}{}$ is a Morita equivalence, then $t\circ
\sigma^1$ is surjective. Therefore, there are a point $u_0\in\groupoid{U}_0$ and a point
$\overline{y}_1\in\groupoid{Y}_1$ such that $s(\overline{y}_1)=\psi_0\circ\xi_0(u_0)$ and
$t(\overline{y}_1)=y_0$. Analogously, there are a point $u'_0\in\groupoid{U}_0$ and a point
$\overline{y}'_1\in\groupoid{Y}_1$ such that $s(\overline{y}'_1)=\psi_0\circ\xi_0(u'_0)$ and
$t(\overline{y}'_1)=y'_0$. Then it makes sense to consider the following point

\begin{gather}
\label{eq-91} \widetilde{x}_1:=m\left(m\left(\phi_1\left(\overline{y}_1\right),x_1\right),i\circ
 \phi_1(\overline{y}'_1)\right)= \\
\nonumber =m\left(\phi_1(\overline{y}_1),m\left(x_1,i\circ\phi_1(\overline{y}'_1)\right)\right)\in
 \groupoid{X}_1.
\end{gather}

By construction,

\[s(\widetilde{x}_1)=s\circ\phi_1(\overline{y}_1)=\phi_0\circ s(\overline{y}_1)=\phi_0\circ
\psi_0\circ\xi_0(u_0)\]
and $t(\widetilde{x}_1)=
\phi_0\circ\psi_0\circ\xi_0(u'_0)$. Since $\groupoidmaptot{\phi}{}\circ\groupoidmaptot{\psi}{}$ is a
Morita equivalence, then by (\hyperref[V2]{V2}) the following diagram is cartesian

\begin{equation}\label{eq-23}
\begin{tikzpicture}[xscale=2.2,yscale=-1.0]
    \node (A0_0) at (0, 0) {$\groupoid{Z}_1$};
    \node (A0_2) at (2, 0) {$\groupoid{X}_1$};
    \node (A2_0) at (0, 2) {$\groupoid{Z}_0\times\groupoid{Z}_0$};
    \node (A2_2) at (2, 2) {$\groupoid{X}_0\times\groupoid{X}_0$.};
    
    \node (A1_1) at (1, 1) {$\square$};
    
    \path (A0_0) edge [->]node [auto] {$\scriptstyle{\phi_1\circ\psi_1}$} (A0_2);
    \path (A0_0) edge [->]node [auto,swap] {$\scriptstyle{(s,t)}$} (A2_0);
    \path (A0_2) edge [->]node [auto] {$\scriptstyle{(s,t)}$} (A2_2);
    \path (A2_0) edge [->]node [auto,swap]
      {$\scriptstyle{(\phi_0\circ\psi_0\times\phi_0\circ\psi_0)}$} (A2_2);
\end{tikzpicture}
\end{equation}

Therefore, there is a unique object $z_1\in\groupoid{Z}_1$ such that $(s,t)(z_1)=(\xi_0(u_0),
\xi_0(u'_0))$ and $\phi_1\circ\psi_1(z_1)=\widetilde{x}_1$. Then it makes sense to consider
the point

\[y_1:=m\left(m\left(i(\overline{y}_1),\psi_1(z_1)\right),\overline{y}'_1\right)=m\left(i
(\overline{y}_1),m\left(\psi_1(z_1),\overline{y}'_1\right)\right)\in\groupoid{Y}_1\]
and we have that

\begin{gather*}
\phi_1(y_1)=m\left(m\left(i\circ\phi_1(\overline{y}_1),\phi_1\circ\psi_1(z_1)\right),\phi_1
 (\overline{y}'_1)\right)= \\
=m\left(m\left(i\circ\phi_1(\overline{y}_1),\widetilde{x}_1)\right),\phi_1
 (\overline{y}'_1)\right)\stackrel{\eqref{eq-91}}{=}x_1.
\end{gather*}

Moreover, we have

\[(s,t)(y_1)=(s\circ i(\overline{y}_1),t(\overline{y}'_1))=(t(\overline{y}_1),t(\overline{y}'_1))=
(y_0,y'_0).\]

So we have proved that $\gamma$ is surjective.\\

Now we have also to prove that $\gamma$ is injective. So let us fix any pair of points $y_1^1,y_1^2$
in $\groupoid{Y}_1$ and let us suppose that $\gamma(y_1^1)=\gamma(y_1^2)$. For simplicity, we set

\[(y_0,y'_0):=(s,t)(y_1^1)=\tau^2\circ\gamma(y_1^1)=\tau^2\circ\gamma(y_1^2)=(s,t)(y_1^2)\]
and we choose a quadruple of points $(u_0,u'_0,\overline{y}_1,\overline{y}'_1)$ as in the previous
lines. Then for each $l=1,2$ we set:

\begin{equation}\label{eq-92}
\widetilde{y}_1^l:=m\left(m\left(\overline{y}_1,y_1^l\right),i(\overline{y}'_1)\right)=
m\left(\overline{y}_1,m\left(y_1^l,i(\overline{y}'_1)\right)\right)\in\groupoid{Y}_1
\end{equation}
and we have

\[s(\widetilde{y}_1^l)=s(\overline{y}_1)=\psi_0\circ\xi_0(u_0)\quad\textrm{and}\quad
t(\widetilde{y}_1^l)=\psi_0\circ\xi_0(u'_0)\quad\textrm{for}\,\,l=1,2.\]

Since $\psi_{\bullet}\circ\xi_{\bullet}$ is a Morita equivalence, then by (\hyperref[V2]{V2}) the
following diagram is cartesian:

\[
\begin{tikzpicture}[xscale=2.2,yscale=-0.8]
    \node (A0_0) at (0, 0) {$\groupoid{U}_1$};
    \node (A0_2) at (2, 0) {$\groupoid{Y}_1$};
    \node (A2_0) at (0, 2) {$\groupoid{U}_0\times\groupoid{U}_0$};
    \node (A2_2) at (2, 2) {$\groupoid{Y}_0\times\groupoid{Y}_0$.};
    
    \node (A1_1) at (1, 1) {$\square$};
    
    \path (A0_0) edge [->]node [auto] {$\scriptstyle{\psi_1\circ\xi_1}$} (A0_2);
    \path (A0_0) edge [->]node [auto,swap] {$\scriptstyle{(s,t)}$} (A2_0);
    \path (A0_2) edge [->]node [auto] {$\scriptstyle{(s,t)}$} (A2_2);
    \path (A2_0) edge [->]node [auto,swap]
      {$\scriptstyle{(\psi_0\circ\xi_0\times\psi_0\circ\xi_0)}$} (A2_2);
\end{tikzpicture}
\]

Therefore, for each $l=1,2$ there is a unique $u_1^l\in\groupoid{U}_1$ such that

\begin{equation}\label{eq-07}
(s,t)(u_1^l)=(u_0,u'_0)\quad\quad\textrm{and}\quad\quad\psi_1\circ\xi_1(u_1^l)=\widetilde{y}_1^l.
\end{equation}

Now by \eqref{eq-90} we have $\phi_1(y_1^1)=\tau^1\circ\gamma(y_1^1)=\tau^1\circ\gamma(y_1^2)=
\phi_1(y_1^2)$, hence

\begin{gather}
\nonumber \phi_1\circ\psi_1\circ\xi_1(u_1^1)\stackrel{\eqref{eq-07}}{=}\phi_1(\widetilde{y}_1^1)
 \stackrel{\eqref{eq-92}}{=}m\left(m\left(\phi_1
 (\overline{y}_1),\phi_1(y_1^1)\right),i\circ\phi_1(\overline{y}'_1)\right)= \\
\label{eq-06} =m\left(m\left(\phi_1(\overline{y}_1),\phi_1(y_1^2)\right),i\circ\phi_1
 (\overline{y}'_1)\right)\stackrel{\eqref{eq-92}}{=}\phi_1(\widetilde{y}_1^2)
 \stackrel{\eqref{eq-07}}{=}\phi_1\circ\psi_1\circ\xi_1(u_1^2).
\end{gather}

Moreover, we have

\begin{gather}
\nonumber (s,t)\circ\xi_1(u^1_1)=(\xi_0,\xi_0)\circ(s,t)(u^1_1)
 \stackrel{\eqref{eq-07}}{=}(\xi_0,\xi_0)(u_0,u'_0)\stackrel{\eqref{eq-07}}{=} \\
\label{eq-93} \stackrel{\eqref{eq-07}}{=}(\xi_0,\xi_0)\circ(s,t)(u^2_1)
 =(s,t)\circ\xi_1(u^2_1).
\end{gather}

Since diagram \eqref{eq-23} is cartesian, then by \eqref{eq-93} and \eqref{eq-06} we get that $\xi_1
(u_1^1)=\xi_1(u_1^2)$. Therefore,

\[\widetilde{y}_1^1\stackrel{\eqref{eq-07}}{=}\psi_1\circ\xi_1(u_1^1)=\psi_1\circ\xi_1(u_1^2)
\stackrel{\eqref{eq-07}}{=}\widetilde{y}_1^2.\]
From this and \eqref{eq-92} we conclude that $y_1^1=y_1^2$. This proves that $\gamma$ is injective.\\

So we have proved that $\gamma$ is an \'etale map that is a bijection, hence $\gamma$ is a
diffeomorphism of smooth manifolds. This means that diagram \eqref{eq-09} is cartesian, so
(\hyperref[V2]{V2}) holds, hence we have proved that $\groupoidmaptot{\phi}{}$ is a Morita
equivalence.
\end{proof}

Actually, since also
$\groupoidmaptot{\phi}{}\circ\groupoidmaptot{\psi}{}$ is a Morita equivalence by hypothesis, then
by~\cite[Lemma~8.1]{PS} we conclude that $\groupoidmaptot{\psi}{}$ is a Morita equivalence. Since
also $\groupoidmaptot{\psi}{}\circ\groupoidmaptot{\xi}{}$ is a Morita equivalence by hypothesis, then
again by~\cite[Lemma~8.1]{PS} we conclude that also $\groupoidmaptot{\xi}{}$ is a Morita equivalence.
The same result can also be obtained by remarking that the class $\WEGpd$ is (right) saturated by 
Proposition~\ref{prop-05}, hence we can apply Proposition~\ref{prop-03}(ii).

\begin{cor}
Let us fix any morphism $\groupoidmaptot{\phi}{}:\groupoidtot{Y}{}\rightarrow
\groupoidtot{X}{}$ in $\EGpd$. Then the following facts are equivalent:

\begin{enumerate}[\emphatic{(}a\emphatic{)}]
 \item for each \'etale groupoid $\groupoidtot{Y}{\prime}$ and for each Morita equivalence
  $\groupoidmaptot{\mu}{}:\groupoidtot{Y}{}\rightarrow\groupoidtot{Y}{\prime}$, the morphism
 
  \[
  \begin{tikzpicture}[xscale=2.4,yscale=-1.2]
    \node (A0_0) at (0, 0) {$\groupoidtot{Y}{\prime}$};
    \node (A0_1) at (1, 0) {$\groupoidtot{Y}{}$};
    \node (A0_2) at (2, 0) {$\groupoidtot{X}{}$};
    
    \path (A0_1) edge [->]node [auto,swap] {$\scriptstyle{\groupoidmaptot{\mu}{}}$} (A0_0);
    \path (A0_1) edge [->]node [auto] {$\scriptstyle{\groupoidmaptot{\phi}{}}$} (A0_2);
  \end{tikzpicture}
  \]
  is an internal equivalence in $\EGpd\left[\WEGpdinv\right]$;
 \item the morphism $\groupoidmaptot{\phi}{}$ belongs to $\WEGpdsat$;
 \item the morphism $\groupoidmaptot{\phi}{}$ belongs to $\WEGpd$ \emphatic{(}i.e.\ it is a Morita
  equivalence\emphatic{)}.
\end{enumerate}

In particular, any $2$ \'etale groupoids $\groupoidtot{X}{1},\groupoidtot{X}{2}$ are
\emph{equivalent in $\EGpd\left[\WEGpdinv\right]$} if and only if there are an \'etale groupoid
$\groupoidtot{X}{3}$ and a pair of Morita equivalences as follows

\[
\begin{tikzpicture}[xscale=2.4,yscale=-1.2]
    \node (A0_0) at (0, 0) {$\groupoidtot{X}{1}$};
    \node (A0_1) at (1, 0) {$\groupoidtot{X}{3}$};
    \node (A0_2) at (2, 0) {$\groupoidtot{X}{2}$};
    
    \path (A0_1) edge [->]node [auto,swap] {$\scriptstyle{\groupoidmaptot{\mu}{1}}$} (A0_0);
    \path (A0_1) edge [->]node [auto] {$\scriptstyle{\groupoidmaptot{\mu}{2}}$} (A0_2);
\end{tikzpicture}
\]
i.e.\ if and only if $\groupoidtot{X}{1}$ and $\groupoidtot{X}{2}$ are \emph{Morita equivalent}. The
same statements holds if we restrict to the bicategories $\PEGpd\left[\WPEGpdinv\right]$ and
$\PEEGpd\left[\WPEEGpdinv\right]$.
\end{cor}

The equivalence of (a) and (b) is a direct consequence of Corollary~\ref{cor-02}; the equivalence of
(b) and (c) is simply Proposition~\ref{prop-05} for the case of $\EGpd$. The claims for $\PEGpd$ and
for $\PEEGpd$ follow at once from this and~\cite[Proposition~5.6 and Example~5.2.1(2)]{MM}.\\

As we mentioned above, by~\cite[Corollary~43]{Pr}
the bicategory $\EGpd\left[\WEGpdinv\right]$ is equivalent to the
$2$-category of differentiable stacks. Therefore, if one wants to construct a $2$-category
(equivalent to the $2$-category) of differentiable stacks, a possible way for doing that
is the following:

\begin{itemize}
 \item construct a bicategory $\CATA$ and identify a suitable class of morphisms $\SETWA$ in it, so
  that there is a bicategory of fractions $\CATA\left[\SETWAinv\right]$;
 \item construct a pseudofunctor $\functor{F}:\CATA\rightarrow\EGpd$, such that $\functor{F}_1
  (\SETWA)$ is contained in the class $\WEGpd$ of Morita equivalences (in general we should impose
  that $\functor{F}_1(\SETWA)$ is contained in the right saturation of $\WEGpd$, but the $2$ classes
  coincide because of Proposition~\ref{prop-05});
 \item consider the induced pseudofunctor $\widetilde{\functor{G}}$ as described in
  Theorem~\ref{theo-03}(A) (since $\WEGpd$ is right saturated, this coincides with
  Theorem~\ref{theo-03}(B)) and verify whether it is an equivalence of bicategories.
\end{itemize}

Then the natural question to ask is the following: under which conditions on $(\CATA,\SETWA,\EGpd,
\WEGpd,\functor{F})$ is the induced pseudofunctor $\widetilde{\functor{G}}$ an equivalence
of bicategories? As we mentioned above, in the next paper~\cite{T5} we will tackle and solve this
question in the more general case
when the pair $(\EGpd,\WEGpd)$ is replaced by any pair $(\CATB,\SETWB)$ satisfying conditions (BF).

\section*{Appendix}

\begin{proof}[Proof of Lemma~\ref{lem-06}.]
Let us suppose that both $f$ and $g$ are internal equivalences. Then there are a pair of morphisms
$\overline{f}:A\rightarrow B$, $\overline{g}:B\rightarrow C$ and a quadruple of invertible
$2$-morphisms as follows:

\begin{gather}
\label{eq-86} \delta:\id_B\Longrightarrow\overline{f}\circ f,\quad\quad
  \xi:f\circ\overline{f}\Longrightarrow\id_A, \\
\label{eq-87} \nu:\id_C\Longrightarrow\overline{g}\circ g,\quad\quad
  \mu:g\circ\overline{g}\Longrightarrow\id_B.
\end{gather}

Then the following pair of compositions prove that $f\circ g$ is an internal equivalence, with
$\overline{g}\circ\overline{f}$ as quasi-inverse:

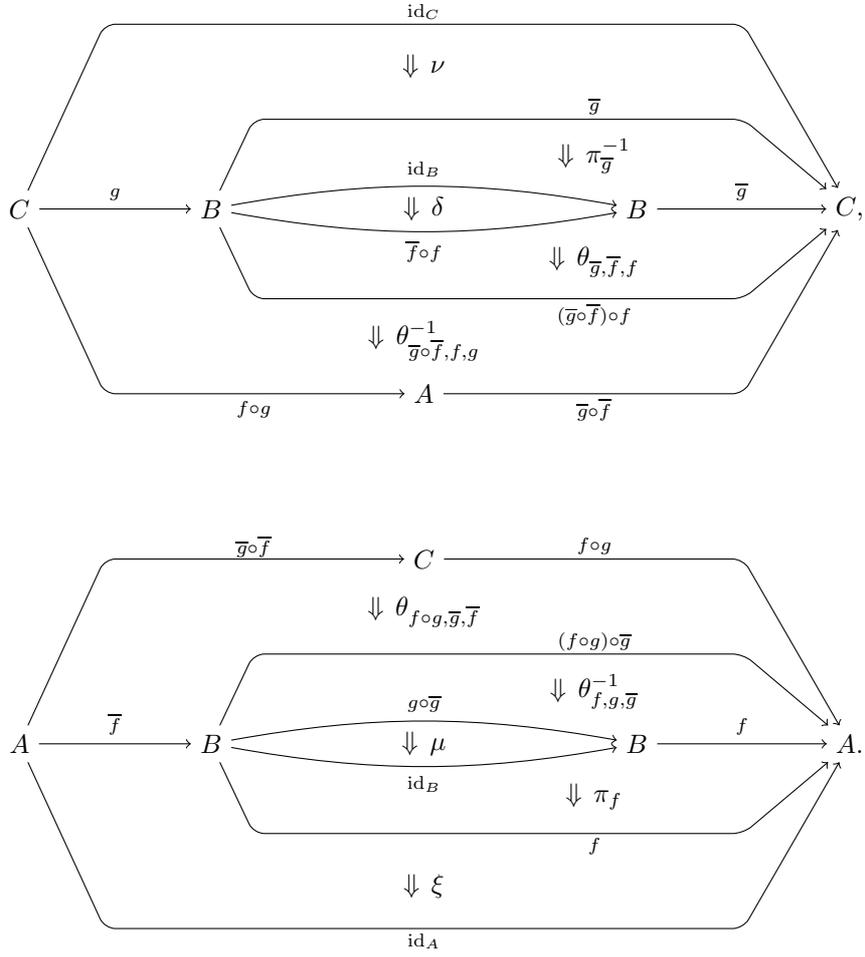
\[
\begin{tikzpicture}[xscale=2.8,yscale=1.4]
    \node (A1_2) at (2, 0.5) {$A$};
    \node (A2_0) at (0.1, 2.25) {$C$};
    \node (A2_1) at (1, 2.25) {$B$};
    \node (A2_4) at (4, 2.25) {$C$,};
    \node (A3_3) at (3, 2.25) {$B$};
  
    \node (A2_2) at (2, 1) {$\Downarrow\,\thetab{\overline{g}\circ\overline{f}}{f}{g}$};
    \node (A2_3) at (2.8, 1.75) {$\Downarrow\,\thetaa{\overline{g}}{\overline{f}}{f}$};
    \node (A3_1) at (2, 2.25) {$\Downarrow\,\delta$};
    \node (A4_2) at (2, 3.6) {$\Downarrow\,\nu$};
    \node (A4_3) at (2.8, 2.75) {$\Downarrow\,\pi_{\overline{g}}^{-1}$};
    
    \node (B1_1) at (2, 4.12) {$\scriptstyle{\id_C}$};
    \node (B2_2) at (2.8, 0.35) {$\scriptstyle{\overline{g}\circ\overline{f}}$};
    \node (B3_3) at (1.2, 0.35) {$\scriptstyle{f\circ g}$};
    \node (B4_4) at (2.8, 1.25) {$\scriptstyle{(\overline{g}\circ\overline{f})\circ f}$};
    \node (B5_5) at (2.8, 3.22) {$\scriptstyle{\overline{g}}$};

    \draw [->,rounded corners] (A1_2) to (3.5, 0.5) to (A2_4); 
    \draw [->, rounded corners] (A2_0) to (0.5, 0.5) to (A1_2); 
    \draw [->, rounded corners] (A2_0) to (0.5, 4) to (3.5, 4) to (A2_4);
    \draw [->,rounded corners] (A2_1) to (1.2, 1.4) to (3.5, 1.4) to (A2_4);
    \draw [->, rounded corners] (A2_1) to (1.2, 3.1) to (3.5, 3.1) to (A2_4);
  
    \path (A2_1) edge [->,bend right=20]node [auto,swap] {$\scriptstyle{\overline{f}\circ f}$} (A3_3);
    \path (A2_1) edge [->,bend left=20]node [auto] {$\scriptstyle{\id_B}$} (A3_3);
    \path (A2_0) edge [->]node [auto] {$\scriptstyle{g}$} (A2_1);
    \path (A3_3) edge [->]node [auto] {$\scriptstyle{\overline{g}}$} (A2_4);
\end{tikzpicture}
\]

\[
\begin{tikzpicture}[xscale=2.8,yscale=-1.4]
    \node (A1_2) at (2, 0.5) {$C$};
    \node (A2_0) at (0.1, 2.25) {$A$};
    \node (A2_1) at (1, 2.25) {$B$};
    \node (A2_4) at (4, 2.25) {$A$.};
    \node (A3_3) at (3, 2.25) {$B$};
  
    \node (A2_2) at (2, 1) {$\Downarrow\,\thetaa{f\circ g}{\overline{g}}{\overline{f}}$};
    \node (A2_3) at (2.8, 1.75) {$\Downarrow\,\thetab{f}{g}{\overline{g}}$};
    \node (A3_1) at (2, 2.25) {$\Downarrow\,\mu$};
    \node (A4_2) at (2, 3.6) {$\Downarrow\,\xi$};
    \node (A4_3) at (2.8, 2.75) {$\Downarrow\,\pi_f$};
    
    \node (B1_1) at (2, 4.12) {$\scriptstyle{\id_A}$};
    \node (B2_2) at (2.8, 0.38) {$\scriptstyle{f\circ g}$};
    \node (B3_3) at (1.2, 0.38) {$\scriptstyle{\overline{g}\circ\overline{f}}$};
    \node (B4_4) at (2.8, 1.28) {$\scriptstyle{(f\circ g)\circ\overline{g}}$};
    \node (B5_5) at (2.8, 3.22) {$\scriptstyle{f}$};
    
    \draw [->,rounded corners] (A1_2) to (3.5, 0.5) to (A2_4); 
    \draw [->, rounded corners] (A2_0) to (0.5, 0.5) to (A1_2); 
    \draw [->, rounded corners] (A2_0) to (0.5, 4) to (3.5, 4) to (A2_4);
    \draw [->,rounded corners] (A2_1) to (1.2, 1.4) to (3.5, 1.4) to (A2_4);
    \draw [->, rounded corners] (A2_1) to (1.2, 3.1) to (3.5, 3.1) to (A2_4);
    
    \path (A2_1) edge [->,bend right=20]node [auto] {$\scriptstyle{g\circ\overline{g}}$} (A3_3);
    \path (A2_1) edge [->,bend left=20]node [auto,swap] {$\scriptstyle{\id_B}$} (A3_3);
    \path (A2_0) edge [->]node [auto] {$\scriptstyle{\overline{f}}$} (A2_1);
    \path (A3_3) edge [->]node [auto] {$\scriptstyle{f}$} (A2_4);
\end{tikzpicture}
\]

Now let us suppose that both $f$ and $f\circ g$ are internal equivalences. Then there are a morphism
$\overline{f}:A\rightarrow B$ and invertible $2$-morphisms as in \eqref{eq-86}; moreover there are
a morphism $h:A\rightarrow C$ and a pair of invertible $2$-morphisms:

\begin{gather}
\label{eq-88} \alpha:\id_C\Longrightarrow h\circ(f\circ g),\quad\quad\quad\beta:
(f\circ g)\circ h\Longrightarrow\id_A.
\end{gather}

Then the following pair of compositions prove that $g$ is an internal equivalence, with $h\circ f$
as quasi-inverse (below for simplicity we write $\theta_{\bullet}$ for any composition of
$2$-identities, associators or inverses of associators):

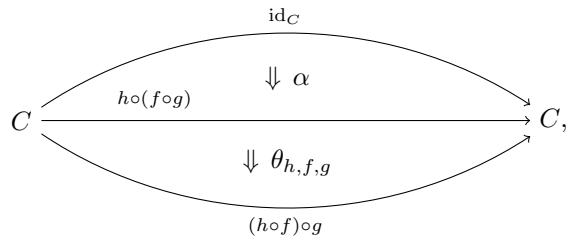
\[
\begin{tikzpicture}[xscale=3.5,yscale=-2.8]
    \node (A1_0) at (0, 1) {$C$};
    \node (A1_2) at (2, 1) {$C$,};

    \node (A0_1) at (1, 0.8) {$\Downarrow\,\alpha$};
    \node (A2_1) at (1, 1.2) {$\Downarrow\,\thetaa{h}{f}{g}$};
    
    \node (B1_1) at (0.5, 0.9) {$\scriptstyle{h\circ (f\circ g)}$};
    
    \path (A1_0) edge [->,bend right=40]node [auto] {$\scriptstyle{\id_C}$} (A1_2);
    \path (A1_0) edge [->]node [auto] {} (A1_2);
    \path (A1_0) edge [->,bend left=40]node [auto,swap] {$\scriptstyle{(h\circ f)\circ g}$} (A1_2);
\end{tikzpicture}
\]

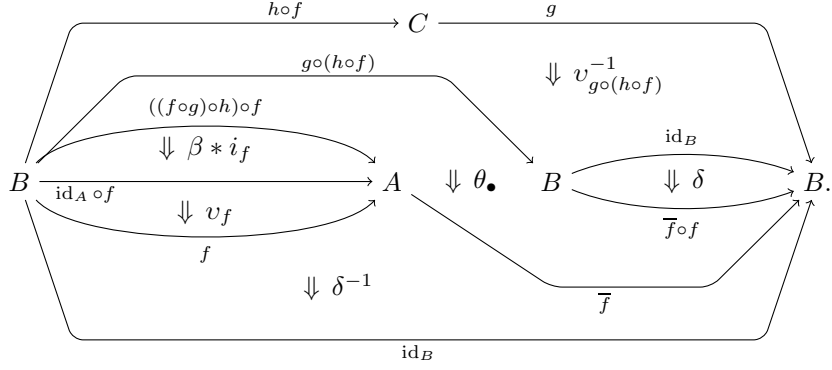
\[
\begin{tikzpicture}[xscale=3.5,yscale=-1.4]
    \node (A0_2) at (1.5, 0.5) {$C$};
    \node (A2_0) at (0, 2) {$B$};
    \node (A2_2) at (2, 2) {$B$};
    \node (A2_3) at (3, 2) {$B$.};
    \node (A3_1) at (1.4, 2) {$A$};

    \node (A4_0) at (0.7, 2.3) {$\Downarrow\,\upsilon_f$};
    \node (A2_1) at (1.7, 2) {$\Downarrow\,\theta_{\bullet}$};
    \node (A1_1) at (2.2, 1) {$\Downarrow\,\upsilon^{-1}_{g\circ(h\circ f)}$};
    \node (A1_2) at (2.5, 2) {$\Downarrow\,\delta$};
    \node (A3_2) at (1.2, 3) {$\Downarrow\,\delta^{-1}$};
    \node (A3_0) at (0.7, 1.7) {$\Downarrow\,\beta\ast i_f$};
    
    \node (B1_1) at (1.5, 3.62) {$\scriptstyle{\id_B}$};
    \node (B2_2) at (2, 0.38) {$\scriptstyle{g}$};
    \node (B3_3) at (1, 0.38) {$\scriptstyle{h\circ f}$};
    \node (B4_4) at (1.2, 0.88) {$\scriptstyle{g\circ(h\circ f)}$};
    \node (B5_5) at (2.2, 3.15) {$\scriptstyle{\overline{f}}$};
    \node (B6_6) at (0.25, 2.12) {$\scriptstyle{\id_A\circ f}$};
    
    \draw [->,rounded corners] (A2_0) to (0.2, 3.5) to (2.8, 3.5) to  (A2_3);
    \draw [->, rounded corners] (A0_2) to (2.8, 0.5) to (A2_3);
    \draw [->, rounded corners] (A2_0) to (0.2, 0.5) to (A0_2);
    \draw [->, rounded corners] (A2_0) to (0.4, 1) to (1.6, 1) to (A2_2);
    \draw [->, rounded corners] (A3_1) to (2, 3) to (2.6, 3) to (A2_3);
    
    \path (A2_0) edge [->,bend right=70]node [auto]
      {$\scriptstyle{((f\circ g)\circ h)\circ f}$} (A3_1);
    \path (A2_0) edge [->,bend left=70]node [auto,swap] {$\scriptstyle{f}$} (A3_1);
    \path (A2_0) edge [->]node [auto,swap] {} (A3_1);
    \path (A2_2) edge [->,bend right=45]node [auto] {$\scriptstyle{\id_B}$} (A2_3);
    \path (A2_2) edge [->,bend left=45]node [auto,swap] {$\scriptstyle{\overline{f}\circ f}$} (A2_3);
\end{tikzpicture}
\]

Lastly, let us suppose that both $f\circ g$ and $g$ are internal equivalences. Then there are a pair
of morphisms $h:A\rightarrow C$ and $\overline{g}:B\rightarrow C$ and invertible $2$-morphisms
as in \eqref{eq-87} and \eqref{eq-88}. Then the following compositions prove that $f$ is an
internal equivalence, with $g\circ h$ as quasi-inverse.

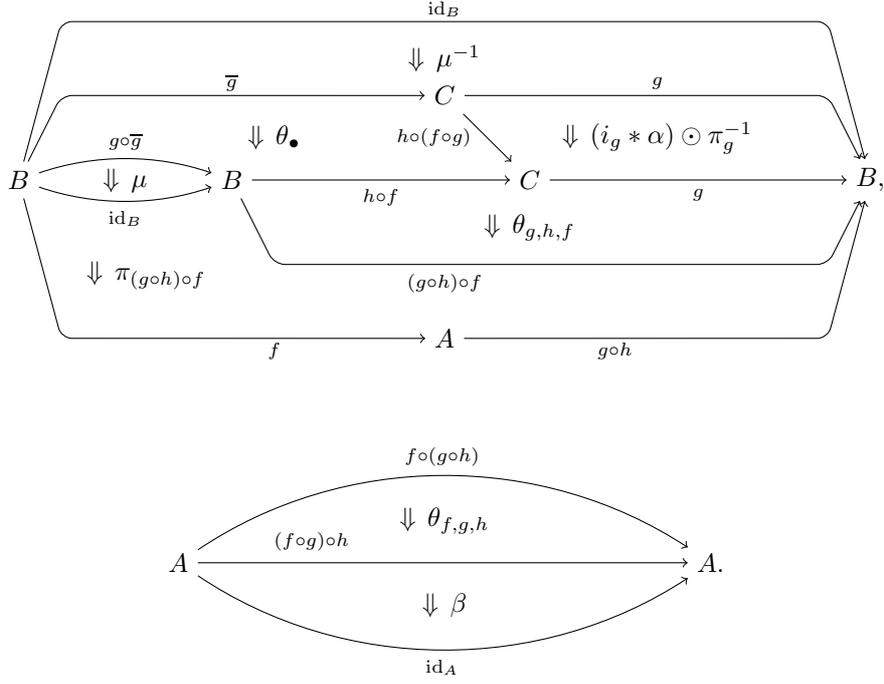
\[
\begin{tikzpicture}[xscale=2.8,yscale=-1.4]
    \node (A1_2) at (2, 1.2) {$C$};
    \node (A2_0) at (0, 2) {$B$};
    \node (A2_1) at (1, 2) {$B$};
    \node (A2_3) at (2.4, 2) {$C$};
    \node (A2_4) at (4, 2) {$B$,};
    \node (A4_2) at (2, 3.5) {$A$};
    
    \node (A1_1) at (2, 0.85) {$\Downarrow\,\mu^{-1}$};
    \node (A3_0) at (1.2, 1.6) {$\Downarrow\,\theta_{\bullet}$};
    \node (A3_1) at (0.5, 2) {$\Downarrow\,\mu$};
    \node (A3_3) at (0.6, 2.9) {$\Downarrow\,\pi_{(g\circ h)\circ f}$};    
    \node (A1_3) at (3, 1.6) {$\Downarrow\,(i_g\ast\alpha)\odot\pi^{-1}_g$};
    \node (C1_1) at (2.4, 2.45) {$\Downarrow\,\thetaa{g}{h}{f}$};
    \node (C1_1) at (1.95, 1.6) {$\scriptstyle{h\circ(f\circ g)}$};
        
    \node (B1_1) at (1, 1.08) {$\scriptstyle{\overline{g}}$};
    \node (B6_6) at (3, 1.08) {$\scriptstyle{g}$};
    \node (B2_2) at (2, 0.38) {$\scriptstyle{\id_B}$};
    \node (B3_3) at (1.2, 3.62) {$\scriptstyle{f}$};
    \node (B4_4) at (2.8, 3.62) {$\scriptstyle{g\circ h}$};
    \node (B5_5) at (2, 2.98) {$\scriptstyle{(g\circ h)\circ f}$};
    
    \draw [->,rounded corners] (A2_0) to (0.2, 1.2) to (A1_2);
    \draw [->,rounded corners] (A2_0) to (0.2, 0.5) to (3.8, 0.5) to (A2_4);
    \draw [->,rounded corners] (A2_0) to (0.2, 3.5) to (A4_2);
    \draw [->,rounded corners] (A4_2) to (3.8, 3.5) to (A2_4);
    \draw [->,rounded corners] (A2_1) to (1.2, 2.8) to (3.8, 2.8) to (A2_4);
    \draw [->,rounded corners] (A1_2) to (3.8, 1.2) to (A2_4);

    \path (A1_2) edge [->]node [auto,swap] {} (A2_3);
    \path (A2_0) edge [->,bend right=35]node [auto] {$\scriptstyle{g\circ\overline{g}}$} (A2_1);
    \path (A2_0) edge [->,bend left=35]node [auto,swap] {$\scriptstyle{\id_B}$} (A2_1);
    \path (A2_1) edge [->]node [auto,swap] {$\scriptstyle{h\circ f}$} (A2_3);
    \path (A2_3) edge [->]node [auto,swap] {$\scriptstyle{g}$} (A2_4);
\end{tikzpicture}
\]

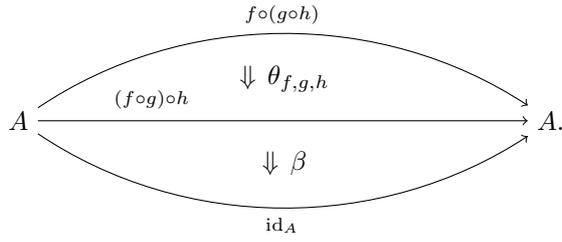
\[
\begin{tikzpicture}[xscale=3.5,yscale=2.8]
    \node (A1_0) at (0, 1) {$A$};
    \node (A1_2) at (2, 1) {$A$.};

    \node (A0_1) at (1, 0.8) {$\Downarrow\,\beta$};
    \node (A2_1) at (1, 1.2) {$\Downarrow\,\thetaa{f}{g}{h}$};
    
    \node (B1_1) at (0.5, 1.1) {$\scriptstyle{(f\circ g)\circ h}$};
    
    \path (A1_0) edge [->,bend right=40]node [auto,swap] {$\scriptstyle{\id_A}$} (A1_2);
    \path (A1_0) edge [->]node [auto] {} (A1_2);
    \path (A1_0) edge [->,bend left=40]node [auto] {$\scriptstyle{f\circ(g\circ h)}$} (A1_2);
\end{tikzpicture}
\]
\end{proof}

\begin{proof}[Proof of Lemma~\ref{lem-02}]
By definition of internal equivalence, there are morphisms $m:A\rightarrow C$ and $n:B\rightarrow D$
and a quadruple of invertible $2$-morphisms as follows:

\begin{gather*}
\delta:\,\id_C\Longrightarrow m\circ(f\circ g),\quad\quad\xi:\,(f\circ g)\circ m\Longrightarrow
 \id_A,\\
\eta:\,\id_D\Longrightarrow n\circ(g\circ h),\quad\quad\mu:\,(g\circ h)\circ n\Longrightarrow\id_B.
\end{gather*}

Then the following pair of compositions prove that $f$ is an internal equivalence of $\CATC$, with
$g\circ m$ as quasi-inverse:

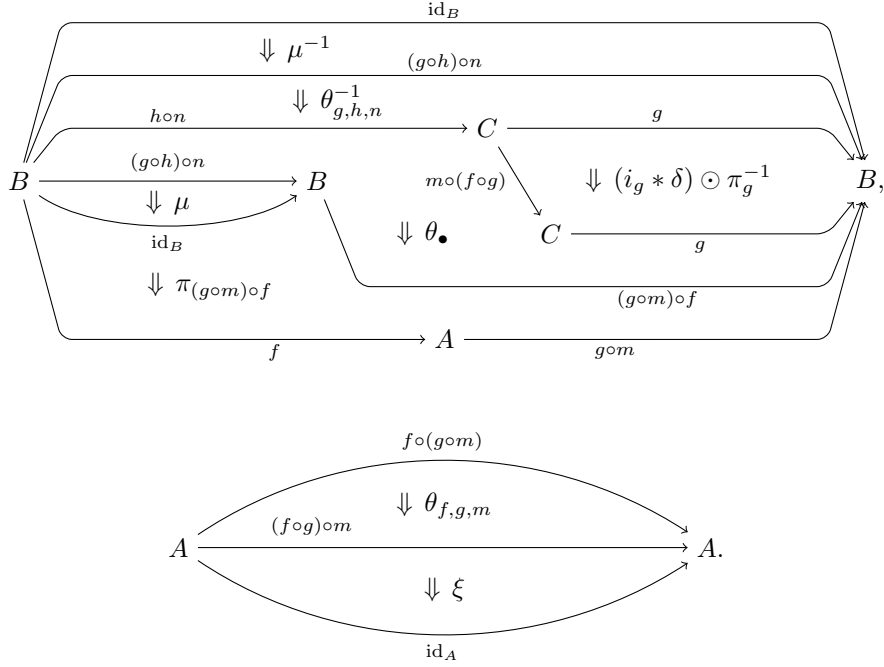
\[
\begin{tikzpicture}[xscale=2.8,yscale=-1.4]
    \node (A2_0) at (0, 2) {$B$};
    \node (A2_1) at (2.2, 1.5) {$C$};
    \node (A2_3) at (2.5, 2.5) {$C$};
    \node (A2_4) at (4, 2) {$B$,};
    \node (A3_2) at (1.4, 2) {$B$};
    \node (A4_2) at (2, 3.5) {$A$};
    
    \node (A3_3) at (0.9, 3) {$\Downarrow\,\pi_{(g\circ m)\circ f}$};
    \node (A3_1) at (0.7, 2.2) {$\Downarrow\,\mu$};
    \node (A0_1) at (1.3, 0.75) {$\Downarrow\,\mu^{-1}$};
    \node (A1_2) at (1.5, 1.25) {$\Downarrow\,\thetab{g}{h}{n}$};
    \node (A1_3) at (3.1, 2) {$\Downarrow\,(i_g\ast\delta)\odot\pi_g^{-1}$};
    \node (A2_2) at (1.9, 2.5) {$\Downarrow\,\theta_{\bullet}$};
 
    \node (B1_1) at (1.2, 3.62) {$\scriptstyle{f}$};
    \node (B2_2) at (2.8, 3.62) {$\scriptstyle{g\circ m}$};
    \node (B3_3) at (2, 0.38) {$\scriptstyle{\id_B}$};
    \node (B4_4) at (2, 0.88) {$\scriptstyle{(g\circ h)\circ n}$};
    \node (B5_5) at (3, 3.12) {$\scriptstyle{(g\circ m)\circ f}$};
    \node (B6_6) at (3, 1.38) {$\scriptstyle{g}$};
    \node (B7_7) at (0.7, 1.38) {$\scriptstyle{h\circ n}$};
    \node (B8_8) at (3.2, 2.62) {$\scriptstyle{g}$};
    \node (B9_9) at (2.1, 2) {$\scriptstyle{m\circ(f\circ g)}$};
    
    \draw [->,rounded corners] (A2_0) to (0.2, 3.5) to (A4_2);
    \draw [->,rounded corners] (A4_2) to (3.8, 3.5) to (A2_4);
    \draw [->,rounded corners] (A2_0) to (0.2, 0.5) to (3.8, 0.5) to (A2_4);
    \draw [->,rounded corners] (A2_0) to (0.2, 1) to (3.8, 1) to (A2_4);
    \draw [->,rounded corners] (A3_2) to (1.6, 3) to (3.8, 3) to (A2_4);
    \draw [->,rounded corners] (A2_1) to (3.8, 1.5) to (A2_4);
    \draw [->,rounded corners] (A2_0) to (0.2, 1.5) to (A2_1);
    \draw [->,rounded corners] (A2_3) to (3.8, 2.5) to (A2_4);
    
    \path (A2_0) edge [->]node [auto] {$\scriptstyle{(g\circ h)\circ n}$} (A3_2);
    \path (A2_0) edge [->,bend left=55]node [auto,swap] {$\scriptstyle{\id_B}$} (A3_2);
    \path (A2_1) edge [->]node [auto,swap] {} (A2_3);
\end{tikzpicture}
\]

\[
\begin{tikzpicture}[xscale=3.5,yscale=-2.8]
    \node (A1_0) at (0, 1) {$A$};
    \node (A1_2) at (2, 1) {$A$.};

    \node (A0_1) at (1, 0.8) {$\Downarrow\,\thetaa{f}{g}{m}$};
    \node (A2_1) at (1, 1.2) {$\Downarrow\,\xi$};
    
    \node (B1_1) at (0.5, 0.9) {$\scriptstyle{(f\circ g)\circ m}$};
    
    \path (A1_0) edge [->,bend right=40]node [auto] {$\scriptstyle{f\circ(g\circ m)}$} (A1_2);
    \path (A1_0) edge [->]node [auto] {} (A1_2);
    \path (A1_0) edge [->,bend left=40]node [auto,swap] {$\scriptstyle{\id_A}$} (A1_2);
\end{tikzpicture}
\]

Now both $f\circ g$ and $f$ are internal equivalences, so by Lemma~\ref{lem-06} we conclude that also
$g$ is an internal equivalence. Since both $g$ and $g\circ h$ are internal equivalences, then again
by Lemma~\ref{lem-06} we conclude that also $h$ is an internal equivalence.
\end{proof}

\begin{proof}[Proof of Lemma~\ref{lem-07}]
Since $\phi$ is a pseudonatural equivalence of pseudofunctors, then it is described by:

\begin{itemize}
 \item a collection of internal equivalences $\phi(A_{\CATC}):\functor{F}_0(A_{\CATC})\rightarrow
  \functor{G}_0(A_{\CATC})$ in $\CATD$ for each object $A_{\CATC}$;
 \item a collection of invertible $2$-morphisms $\phi(f_{\CATC})$ in $\CATD$ for each morphism
  $f_{\CATC}:A_{\CATC}\rightarrow B_{\CATC}$, as follows
  
  \[
  \begin{tikzpicture}[xscale=1.5,yscale=-1.2]
    \node (A0_0) at (0, 0) {$\functor{F}_0(A_{\CATC})$};
    \node (A0_2) at (2, 0) {$\functor{F}_0(B_{\CATC})$};
    \node (A2_0) at (0, 2) {$\functor{G}_0(A_{\CATC})$};
    \node (A2_2) at (2, 2) {$\functor{G}_0(B_{\CATC})$,};
    
    \node (B1_1) [rotate=225] at (0.7, 1) {$\Longrightarrow$};
    \node (A1_1) at (1.2, 1) {$\phi(f_{\CATC})$};
    
    \path (A2_0) edge [->]node [auto,swap] {$\scriptstyle{\functor{G}_1(f_{\CATC})}$} (A2_2);
    \path (A0_0) edge [->]node [auto,swap] {$\scriptstyle{\phi(A_{\CATC})}$} (A2_0);
    \path (A0_2) edge [->]node [auto] {$\scriptstyle{\phi(B_{\CATC})}$} (A2_2);
    \path (A0_0) edge [->]node [auto] {$\scriptstyle{\functor{F}_1(f_{\CATC})}$} (A0_2);
  \end{tikzpicture}
  \]
\end{itemize}
satisfying some coherence conditions. Now let us fix any object $A_{\CATD}$: by (\hyperref[X1]{X1})
for $\functor{F}$ there are an object $A_{\CATC}$ and an internal equivalence $e_{\CATD}:\functor{F}_0
(A_{\CATC})\rightarrow A_{\CATD}$. Since $\phi(A_{\CATC})$ is an internal equivalence, we denote by
$\psi(A_{\CATC}):\functor{G}_0(A_{\CATC})\rightarrow\functor{F}_0(A_{\CATC})$ any chosen
quasi-inverse for $\phi(A_{\CATC})$. Then the internal equivalence
$e_{\CATD}\circ\psi(A_{\CATC})$ proves that (\hyperref[X1]{X1}) holds for $\functor{G}$.\\

Proving (\hyperref[X2]{X2}) for $\functor{G}$ is equivalent to proving the following $3$ conditions
for each pair of objects $A_{\CATC},B_{\CATC}$:

\begin{enumerate}[({X2}a)]
 \item\label{X2a} for each morphism $f_{\CATD}:\functor{G}_0(A_{\CATC})\rightarrow\functor{G}_0
  (B_{\CATC})$, there are a morphism $f_{\CATC}:A_{\CATC}\rightarrow B_{\CATC}$ and an invertible
  $2$-morphism $\alpha_{\CATD}:\functor{G}_1(f_{\CATC})\Rightarrow f_{\CATD}$;
 \item\label{X2b} for each pair of morphisms $f^1_{\CATC},f^2_{\CATC}:A_{\CATC}\rightarrow
  B_{\CATC}$ and for each pair of $2$-morphisms $\alpha^1_{\CATC},\alpha^2_{\CATC}:f^1_{\CATC}
  \Rightarrow f^2_{\CATC}$, if $\functor{G}_2(\alpha^1_{\CATC})=\functor{G}_2(\alpha^2_{\CATC})$,
  then $\alpha^1_{\CATC}=\alpha^2_{\CATC}$;
 \item\label{X2c} for each pair $f^1_{\CATC},f^2_{\CATC}$ as above and for each $2$-morphism
  $\alpha_{\CATD}:\functor{G}_1(f^1_{\CATC})\Rightarrow\functor{G}_1(f^2_{\CATC})$, there is a
  $2$-morphism $\alpha_{\CATC}:f^1_{\CATC}\Rightarrow f^2_{\CATC}$ such that $\functor{G}_2
  (\alpha_{\CATC})=\alpha_{\CATD}$.
\end{enumerate}

Let us prove (\hyperref[X2a]{X2a}), so let us fix
any morphism $f_{\CATD}:\functor{G}_0(A_{\CATC})\rightarrow\functor{G}_0(B_{\CATC})$. Since
both $\phi(A_{\CATC})$ and $\phi(B_{\CATC})$ are internal equivalences, then there are
internal equivalences $\psi(A_{\CATC}):\functor{G}_0(A_{\CATC})\rightarrow\functor{F}_0(A_{\CATC})$,
$\psi(B_{\CATC}):\functor{G}_0(B_{\CATC})\rightarrow\functor{F}_0(B_{\CATC})$ and invertible
$2$-morphisms as follows in $\CATD$:

\begin{gather*}
\delta_{A_{\CATC}}:\id_{\functor{F}_0(A_{\CATC})}\Longrightarrow\psi(A_{\CATC})\circ\phi(A_{\CATC}),
 \quad\quad\xi_{A_{\CATC}}:\phi(A_{\CATC})\circ\psi(A_{\CATC})\Longrightarrow
 \id_{\functor{G}_0(A_{\CATC})}, \\
\delta_{B_{\CATC}}:\id_{\functor{F}_0(B_{\CATC})}\Longrightarrow\psi(B_{\CATC})\circ\phi(B_{\CATC}),
 \quad\quad\xi_{B_{\CATC}}:\phi(B_{\CATC})\circ\psi(B_{\CATC})\Longrightarrow
 \id_{\functor{G}_0(B_{\CATC})}.
\end{gather*}

By (\hyperref[X2a]{X2a}) for $\functor{F}$, there are a morphism $f_{\CATC}:A_{\CATC}\rightarrow
B_{\CATC}$ and an invertible $2$-morphism

\[\alpha_{\CATD}:\,\functor{F}_1(f_{\CATC})\Longrightarrow\psi(B_{\CATC})\circ f_{\CATD}\circ
\phi(A_{\CATC}).\]

Then the composition of the following invertible $2$-morphism proves that
(\hyperref[X2a]{X2a}) holds for $\functor{G}$ (for simplicity, we omit all the unitors
and associators of $\CATD$):

\[
\begin{tikzpicture}[xscale=2.8,yscale=-1.8]
    \node (A0_2) at (2, 0) {$\functor{G}_0(A_{\CATC})$};
    \node (A1_0) at (0, 1) {$\functor{G}_0(A_{\CATC})$};
    \node (A1_1) at (1, 1) {$\functor{F}_0(A_{\CATC})$};
    \node (A1_3) at (3, 1) {$\functor{F}_0(B_{\CATC})$};
    \node (A1_4) at (4, 1) {$\functor{G}_0(B_{\CATC})$.};
    \node (A2_2) at (2, 2) {$\functor{G}_0(A_{\CATC})$};
    \node (B2_2) at (3, 2) {$\functor{G}_0(B_{\CATC})$};
    
    \node (A2_3) at (2.3, 1.5) {$\Downarrow\,\alpha_{\CATD}$};
    \node (A2_1) at (1, 1.4) {$\Downarrow\,\xi_{A_{\CATC}}$};
    \node (A0_3) at (2.5, 0.5) {$\Downarrow\,(\phi(f_{\CATC}))^{-1}$};
    \node (A0_1) at (1, 0.6) {$\Downarrow\,\xi_{A_{\CATC}}^{-1}$};
    \node (B1_1) at (3.4, 1.4) {$\Downarrow\,\xi_{B_{\CATC}}$};
    
    \path (A2_2) edge [->]node [auto,swap] {$\scriptstyle{f_{\CATD}}$} (B2_2);
    \path (B2_2) edge [->,bend left=15]node [auto,swap]
      {$\scriptstyle{\id_{\functor{G}_0(B_{\CATC})}}$} (A1_4);
    \path (B2_2) edge [->]node [auto] {$\scriptstyle{\psi(B_{\CATC})}$} (A1_3);
      
    \path (A1_1) edge [->]node [auto,swap] {$\scriptstyle{\phi(A_{\CATC})}$} (A0_2);
    \path (A1_0) edge [->,bend left=15]node [auto,swap]
      {$\scriptstyle{\id_{\functor{G}_0(A_{\CATC})}}$} (A2_2);
    \path (A1_0) edge [->,bend right=15]node [auto]
      {$\scriptstyle{\id_{\functor{G}_0(A_{\CATC})}}$} (A0_2);
    \path (A1_1) edge [->]node [auto] {$\scriptstyle{\phi(A_{\CATC})}$} (A2_2);
    \path (A1_1) edge [->]node [auto] {$\scriptstyle{\functor{F}_1(f_{\CATC})}$} (A1_3);
    \path (A0_2) edge [->,bend right=15]node [auto]
      {$\scriptstyle{\functor{G}_1(f_{\CATC})}$} (A1_4);
    \path (A1_0) edge [->]node [auto] {$\scriptstyle{\psi(A_{\CATC})}$} (A1_1);
    \path (A1_3) edge [->]node [auto] {$\scriptstyle{\phi(B_{\CATC})}$} (A1_4);
\end{tikzpicture}
\]

Now let us prove (\hyperref[X2b]{X2b}), so let us fix any
pair of morphisms $f^1_{\CATC},f^2_{\CATC}:A_{\CATC}\rightarrow B_{\CATC}$, any pair of
$2$-morphisms $\alpha^1_{\CATC},\alpha^2_{\CATC}:f^1_{\CATC}
\Rightarrow f^2_{\CATC}$ and let us suppose that $\functor{G}_2(\alpha^1_{\CATC})=\functor{G}_2
(\alpha^2_{\CATC})$. By the coherence conditions on $\phi$, we get that

\begin{gather*}
\Big(i_{\phi(B_{\CATC})}\ast\functor{F}_2(\alpha^1_{\CATC})\Big)=\phi(f^2_{\CATC})^{-1}\odot\Big(
 \functor{G}_2(\alpha^1_{\CATC})\ast i_{\phi(A_{\CATC})}\Big)\odot\phi(f^1_{\CATC})= \\
=\phi(f^2_{\CATC})^{-1}\odot\Big(\functor{G}_2(\alpha^2_{\CATC})\ast i_{\phi(A_{\CATC})}\Big)\odot
 \phi(f^1_{\CATC})=\Big(i_{\phi(B_{\CATC})}\ast\functor{F}_2(\alpha^2_{\CATC})\Big).
\end{gather*}

Then we get (associators and unitors of $\CATD$ omitted)

\begin{gather*}
\functor{F}_2(\alpha^1_{\CATC})=\Big(\delta^{-1}_{B_{\CATC}}\ast i_{\functor{F}_1(f^2_{\CATC})}\Big)
 \odot\Big(i_{\psi(B_{\CATC})}\ast\Big(i_{\phi(B_{\CATC})}\ast\functor{F}_2(\alpha^1_{\CATC})\Big)
 \Big)\odot\Big(\delta_{B_{\CATC}}\ast i_{\functor{F}_1(f^1_{\CATC})}\Big)= \\
=\Big(\delta^{-1}_{B_{\CATC}}\ast i_{\functor{F}_1(f^2_{\CATC})}\Big)\odot\Big(i_{\psi(B_{\CATC})}\ast
 \Big(i_{\phi(B_{\CATC})}\ast\functor{F}_2(\alpha^2_{\CATC})\Big)\Big)\odot\Big(\delta_{B_{\CATC}}
 \ast i_{\functor{F}_1(f^1_{\CATC})}\Big)=\functor{F}_2(\alpha^2_{\CATC}),
\end{gather*}
so by (\hyperref[X2b]{X2b}) for $\functor{F}$ we conclude that $\alpha^1_{\CATC}=\alpha^2_{\CATC}$.
So (\hyperref[X2b]{X2b}) holds for $\functor{G}$; the proof that (\hyperref[X2c]{X2c}) holds for
$\functor{G}$ is similar.
\end{proof}

\begin{proof}[Proof of Proposition~\ref{prop-02}]
Let us suppose that $\functor{U}_{\SETW,1}(f)$ is an internal equivalence. Then there are an 
internal equivalence

\[
\begin{tikzpicture}[xscale=2.4,yscale=-1.2]
    \node (A0_0) at (-0.15, 0) {$\underline{e}:=\Big(A$};
    \node (A0_1) at (1, 0) {$\widetilde{C}$};
    \node (A0_2) at (2, 0) {$B\Big)$};
    
    \path (A0_1) edge [->]node [auto,swap] {$\scriptstyle{\operatorname{v}}$} (A0_0);
    \path (A0_1) edge [->]node [auto] {$\scriptstyle{\widetilde{g}}$} (A0_2);
\end{tikzpicture}
\]
in $\CATC\left[\SETWinv\right]$ and an invertible $2$-morphism in $\CATC\left[\SETWinv\right]$

\[\Xi:\,\,\functor{U}_{\SETW,1}(f)\circ\underline{e}\Longrightarrow\Big(A,\id_A,\id_A\Big)\]
(the target of $\Xi$ is the identity of $A$ in $\CATC\left[\SETWinv\right]$).
By definition of composition in $\CATC\left[\SETWinv\right]$ (see~\cite[\S~2.2]{Pr}) and by
condition (\hyperref[C2]{C2}), we have $\functor{U}_{\SETW,1}(f)\circ\underline{e}=(\widetilde{C},
\operatorname{v}\circ\id_{\widetilde{C}},f\circ\widetilde{g})$, so by~\cite[\S~2.3]{Pr} any
representative for $\Xi$ is
given as follows, with $(\operatorname{v}\circ\id_{\widetilde{C}})\circ\operatorname{u}$ in
$\SETW$ and $\xi^1$ invertible:

\[
\begin{tikzpicture}[xscale=2.1,yscale=-0.8]
    \node (A2_0) at (0, 2) {$A$};
    \node (A0_2) at (2, 0) {$\widetilde{C}$};
    \node (A2_2) at (2, 2) {$C$};
    \node (A2_4) at (4, 2) {$A$;};
    \node (A4_2) at (2, 4) {$A$};
    
    \node (A2_3) at (2.9, 2) {$\Downarrow\,\xi^2$};
    \node (A2_1) at (1.1, 2) {$\Downarrow\,\xi^1$};
    
    \path (A4_2) edge [->]node [auto,swap] {$\scriptstyle{\id_A}$} (A2_4);
    \path (A0_2) edge [->]node [auto] {$\scriptstyle{f\circ\widetilde{g}}$} (A2_4);
    \path (A2_2) edge [->]node [auto,swap] {$\scriptstyle{\operatorname{u}}$} (A0_2);
    \path (A2_2) edge [->]node [auto] {$\scriptstyle{\operatorname{z}}$} (A4_2);
    \path (A0_2) edge [->]node [auto,swap] {$\scriptstyle{\operatorname{v}
      \circ\id_{\widetilde{C}}}$} (A2_0);
    \path (A4_2) edge [->]node [auto] {$\scriptstyle{\id_A}$} (A2_0);
\end{tikzpicture}
\]
using~\cite[Proposition~0.8]{T3} we can choose the data above in such a way that also $\xi^2$ is
invertible in $\CATC$. Then we define $g:=\widetilde{g}\circ\operatorname{u}:C\rightarrow B$ and we
consider the invertible $2$-morphism in $\CATC$

\[\nu:=\left(\xi^1\right)^{-1}\odot\xi^2\odot\thetaa{f}{\widetilde{g}}{\operatorname{u}}:\,\,
f\circ g\Longrightarrow(\operatorname{v}\circ\id_{\widetilde{C}})\circ\operatorname{u}.\]

Since the target of $\nu$ belongs to $\SETW$, then by (BF5)
we have that $f\circ g$ belongs to $\SETW$.\\

Since $\underline{e}$ is a morphism in $\CATC\left[\SETWinv\right]$, then $\operatorname{v}$ belongs
to $\SETW$; let us suppose that choices \hyperref[C]{C}$(\SETW)$ give data as in the upper part of the following
diagram, with $\operatorname{r}^1$ in $\SETW$ and $\eta$ invertible:

\[
\begin{tikzpicture}[xscale=2.2,yscale=-0.8]
    \node (A0_1) at (1, 0) {$C'$};
    \node (A1_0) at (0, 2) {$\widetilde{C}$};
    \node (A1_2) at (2, 2) {$\widetilde{C}$.};
    \node (A2_1) at (1, 2) {$A$};
    
    \node (A1_1) at (1, 0.95) {$\eta$};
    \node (B1_1) at (1, 1.4) {$\Rightarrow$};
    
    \path (A1_2) edge [->]node [auto] {$\scriptstyle{\operatorname{v}}$} (A2_1);
    \path (A0_1) edge [->]node [auto] {$\scriptstyle{\operatorname{r}^2}$} (A1_2);
    \path (A1_0) edge [->]node [auto,swap] {$\scriptstyle{\operatorname{v}}$} (A2_1);
    \path (A0_1) edge [->]node [auto,swap] {$\scriptstyle{\operatorname{r}^1}$} (A1_0);
\end{tikzpicture}
\]

Then by~\cite[\S~2.2]{Pr} and \eqref{eq-70} we have $\underline{e}\circ\functor{U}_{\SETW,1}
(\operatorname{v})=
(C',\id_{\widetilde{C}}\circ\operatorname{r}^1,\widetilde{g}\circ\operatorname{r}^2)$. By (BF4a)
and (BF4b) applied to $\eta$,
there are an object $C''$, a morphism $\operatorname{r}^3:C''\rightarrow C'$ in
$\SETW$ and an invertible $2$-morphism $\varepsilon:\operatorname{r}^1\circ\operatorname{r}^3
\Rightarrow\operatorname{r}^2\circ\operatorname{r}^3$. Then we define an invertible $2$-morphism
in $\CATC\left[\SETWinv\right]$ as follows:

\[\Gamma:=\Big[C'',\operatorname{r}^3,\operatorname{r}^2\circ\operatorname{r}^3,
\Big(i_{\id_{\widetilde{C}}}\ast\varepsilon\Big)\odot\thetab{\id_{\widetilde{C}}}{\operatorname{r}^1}
{\operatorname{r}^3},\thetab{\widetilde{g}}{\operatorname{r}^2}{\operatorname{r}^3}\Big]:
\underline{e}\circ\functor{U}_{\SETW,1}(\operatorname{v})\Longrightarrow
\Big(\widetilde{C},\id_{\widetilde{C}},\widetilde{g}\Big).\]

Since $(\operatorname{v}\circ\id_{\widetilde{C}})\circ\operatorname{u}$ belongs to $\SETW$, then
we get easily that also $\operatorname{v}\circ\operatorname{u}$ belongs to $\SETW$. So by
Theorem~\ref{theo-01} the
morphism $\functor{U}_{\SETW,1}(\operatorname{v}\circ\operatorname{u})$ is an internal equivalence in
$\CATC\left[\SETWinv\right]$. Moreover, by construction also $\underline{e}$ is an internal
equivalence. Therefore, $\underline{e}\circ\functor{U}_{\SETW,1}(\operatorname{v}\circ
\operatorname{u})$ is an internal equivalence. Now let us consider the invertible $2$-morphism
in $\CATC\left[\SETWinv\right]$

\begin{gather}
\nonumber \Big(\Gamma\ast i_{\functor{U}_{\SETW,1}(\operatorname{u})}\Big)\odot
 \Theta_{\underline{e},\,\functor{U}_{\SETW,1}(\operatorname{v}),\,\functor{U}_{\SETW,1}
 (\operatorname{u})}^{\CATC,\SETW}\odot\Big(i_{\underline{e}}\ast
 \psi^{\functor{U}_{\SETW}}_{\operatorname{v},\operatorname{u}}\Big)^{-1}: \\
\label{eq-53} \underline{e}\circ
 \functor{U}_{\SETW,1}(\operatorname{v}\circ\operatorname{u})\Longrightarrow
 \Big(\widetilde{C},\id_{\widetilde{C}},\widetilde{g}\Big)\circ
 \functor{U}_{\SETW,1}(\operatorname{u});
\end{gather}
here $\psi^{\functor{U}_{\SETW}}_{\bullet}$ denotes the associator of $\functor{U}_{\SETW}$
relative to the pair $(\operatorname{v},\operatorname{u})$ and $\Theta_{\bullet}^{\CATC,\SETW}$
is the associator of $\CATC\left[\SETWinv\right]$ relative to the triple $(\underline{e},
\functor{U}_{\SETW,1}(\operatorname{v}),\functor{U}_{\SETW,1}(\operatorname{u}))$.
Using \eqref{eq-53}, Lemma~\ref{lem-04} and condition (\hyperref[C2]{C2}), we get that the morphism

\begin{equation}\label{eq-71}
\begin{tikzpicture}[xscale=2.4,yscale=-1.2]
    \node (A0_0) at (-1.7, 0) {$\Big(\widetilde{C},\id_{\widetilde{C}},\widetilde{g}\Big)
      \circ\functor{U}_{\SETW,1}(\operatorname{u})=\Big(C$};
    \node (A0_1) at (0.5, 0) {$C$};
    \node (A0_2) at (2, 0) {$B\Big)$};
    
    \path (A0_1) edge [->]node [auto,swap] {$\scriptstyle{\id_C\circ\id_C}$} (A0_0);
    \path (A0_1) edge [->]node [auto] {$\scriptstyle{\widetilde{g}\circ
      \operatorname{u}}$} (A0_2);
\end{tikzpicture}
\end{equation}
is an internal equivalence of $\CATC\left[\SETWinv\right]$. Now there is an obvious invertible
$2$-morphism in $\CATC\left[\SETWinv\right]$ from $(C,\id_C,\widetilde{g}\circ\operatorname{u})=
\functor{U}_{\SETW,1}(g)$ to \eqref{eq-71}, so
again by Lemma~\ref{lem-04} we have that $\functor{U}_{\SETW,1}(g)$ is an internal equivalence
of $\CATC\left[\SETWinv\right]$.\\

So if we perform on $g$ the same computations that we did on $f$,
we get an object $D$ and a morphism $h:D\rightarrow C$ such that
$g\circ h$ belongs to $\SETW$. By comparing with Definition~\ref{def-01}, this proves that $f$
belongs to $\SETWsat$.\\

Conversely, let us suppose that $f:B\rightarrow A$ belongs to $\SETWsat$, so let us suppose that
there are a pair of objects $C,D$ and a pair of morphisms $g:C\rightarrow B$, $h:D\rightarrow C$, such
that both $f\circ g$ and $g\circ h$ belong to $\SETW$. Then it makes sense to define a morphism in
$\CATC\left[\SETWinv\right]$ as follows:

\[
\begin{tikzpicture}[xscale=2.4,yscale=-1.2]
    \node (A0_0) at (-0.1, 0) {$\underline{t}:=\Big(A$};
    \node (A0_1) at (1, 0) {$C$};
    \node (A0_2) at (2, 0) {$B\Big)$.};
    
    \path (A0_1) edge [->]node [auto,swap] {$\scriptstyle{f\circ g}$} (A0_0);
    \path (A0_1) edge [->]node [auto] {$\scriptstyle{g}$} (A0_2);
\end{tikzpicture}
\]

We want to prove that $\underline{t}$ is a quasi-inverse for $\functor{U}_{\SETW,1}(f)$.
By (\hyperref[C2]{C2}) we have

\[
\begin{tikzpicture}[xscale=2.4,yscale=-1.2]
    \node (A0_0) at (-0.4, 0) {$\functor{U}_{\SETW,1}(f)\circ\underline{t}=\Big(A$};
    \node (A0_1) at (1, 0) {$C$};
    \node (A0_2) at (2, 0) {$A\Big)$,};
    
    \path (A0_1) edge [->]node [auto,swap] {$\scriptstyle{(f\circ g)\circ\id_C}$} (A0_0);
    \path (A0_1) edge [->]node [auto] {$\scriptstyle{f\circ g}$} (A0_2);
\end{tikzpicture}
\]
so we can define an invertible $2$-morphism in $\CATC\left[\SETWinv\right]$ as follows:

\[\Xi:=\Big[C,\id_C,f\circ g,\upsilon_{f\circ g}^{-1}\odot\pi_{f\circ g}\odot\pi_{(f\circ g)
\circ\id_C},\upsilon_{f\circ g}^{-1}\odot\pi_{f\circ g}\Big]:\,\,\functor{U}_{\SETW,1}(f)\circ
\underline{t}\Rightarrow(A,\id_A,\id_A).\]

So in order to conclude that $\functor{U}_{\SETW,1}(f)$ is an internal equivalence, we need only to
find and invertible $2$-morphism $\Delta:(B,\id_B,\id_B)$
$\Rightarrow\underline{t}\circ\functor{U}_{\SETW,1}(f)$. In order to do that, let us suppose that
the fixed choices \hyperref[C]{C}$(\SETW)$ give data as in the upper part of the following diagram,
with $l$ in $\SETW$ and $\rho$ invertible:

\[
\begin{tikzpicture}[xscale=1.8,yscale=-0.8]
    \node (A0_1) at (1, 0) {$E$};
    \node (A1_0) at (0, 2) {$B$};
    \node (A1_2) at (2, 2) {$C$.};
    \node (A2_1) at (1, 2) {$A$};
    
    \node (A1_1) at (1, 1) {$\rho$};
    \node (B1_1) at (1, 1.4) {$\Rightarrow$};
    
    \path (A1_2) edge [->]node [auto] {$\scriptstyle{f\circ g}$} (A2_1);
    \path (A0_1) edge [->]node [auto] {$\scriptstyle{m}$} (A1_2);
    \path (A1_0) edge [->]node [auto,swap] {$\scriptstyle{f}$} (A2_1);
    \path (A0_1) edge [->]node [auto,swap] {$\scriptstyle{l}$} (A1_0);
\end{tikzpicture}
\]

By~\cite[\S~2.2]{Pr}, this implies that

\[
\begin{tikzpicture}[xscale=2.4,yscale=-1.2]
    \node (A0_0) at (-0.4, 0) {$\underline{t}\circ\functor{U}_{\SETW,1}(f)=\Big(B$};
    \node (A0_1) at (1, 0) {$E$};
    \node (A0_2) at (2, 0) {$B\Big)$.};
    
    \path (A0_1) edge [->]node [auto,swap] {$\scriptstyle{\id_B\circ l}$} (A0_0);
    \path (A0_1) edge [->]node [auto] {$\scriptstyle{g\circ m}$} (A0_2);
\end{tikzpicture}
\]

By (BF3) there are data as in the upper part of the following diagram,
with $p$ in $\SETW$ and $\sigma$ invertible:

\[
\begin{tikzpicture}[xscale=1.8,yscale=-0.8]
    \node (A0_1) at (1, 0) {$F$};
    \node (A1_0) at (0, 2) {$E$};
    \node (A1_2) at (2, 2) {$D$.};
    \node (A2_1) at (1, 2) {$B$};
    
    \node (A1_1) at (1, 1) {$\sigma$};
    \node (B1_1) at (1, 1.4) {$\Rightarrow$};
    
    \path (A1_2) edge [->]node [auto] {$\scriptstyle{g\circ h}$} (A2_1);
    \path (A0_1) edge [->]node [auto] {$\scriptstyle{n}$} (A1_2);
    \path (A1_0) edge [->]node [auto,swap] {$\scriptstyle{l}$} (A2_1);
    \path (A0_1) edge [->]node [auto,swap] {$\scriptstyle{p}$} (A1_0);
\end{tikzpicture}
\]

Then it makes sense to consider the following invertible $2$-morphism in $\CATC$

\begin{gather*}
\alpha:=\thetaa{f}{g}{h\circ n}\odot\Big(i_f\ast\thetab{g}{h}{n}\Big)\odot\Big(i_f\ast\sigma\Big)
 \odot\thetab{f}{l}{p}\odot\Big(\rho^{-1}\ast i_p\Big)\odot\thetaa{f\circ g}{m}{p}: \\
\phantom{\Big(}(f\circ g)\circ(m\circ p)\Longrightarrow(f\circ g)\circ(h\circ n).
\end{gather*}

Since $f\circ g$ belongs to $\SETW$, by (BF4a) and (BF4b)
there are an object $G$, a morphism $q:G\rightarrow F$ in $\SETW$ and an invertible $2$-morphism
$\beta:(m\circ p)\circ q\Rightarrow(h\circ n)\circ q$.
Then we define an invertible $2$-morphism in $\CATC$

\begin{gather*}
\delta:=\Big(\upsilon^{-1}_l\ast i_{p\circ q}\Big)\odot\thetab{l}{p}{q}\odot\Big(\sigma^{-1}\ast
 i_q\Big)\odot\Big(\thetaa{g}{h}{n}\ast i_q\Big)\odot\thetaa{g}{h\circ n}{q}\odot \\
\odot\Big(i_g\ast\beta\Big)\odot\Big(i_g\ast\thetaa{m}{p}{q}\Big)\odot\thetab{g}{m}{p\circ q}\odot
 \upsilon_{(g\circ m)\circ(p\circ q)}: \\
\id_B\circ((g\circ m)\circ(p\circ q))\Longrightarrow(\id_B\circ l)\circ(p\circ q).
\end{gather*}

Then it makes sense to define an invertible $2$-morphism in $\CATC\left[\SETWinv\right]$ as follows:

\begin{gather*}
\Delta:=\Big[G,(g\circ m)\circ(p\circ q),p\circ q,\delta,\upsilon_{(g\circ m)\circ(p\circ q)}
 \Big]: \\
\Big(B,\id_B,\id_B\Big)\Longrightarrow\Big(E,\id_B\circ l,g\circ m\Big)=\underline{t}\circ
 \functor{U}_{\SETW,1}(f).
\end{gather*}

This suffices to conclude.
\end{proof}

Actually, a direct computation using~\cite[pagg.~260--261]{Pr} proves that the quadruple
$(\functor{U}_{\SETW,1}(f),\underline{t},\Delta, \Xi)$ is an adjoint equivalence, but this fact
was not needed for the proof above.

\begin{proof}[Proof of Lemma~\ref{lem-05}]
Condition (BF1)
is obvious since $\SETW\subseteq\SETWsat$.\\

Let us fix any pair of morphisms $\operatorname{w}:B\rightarrow A$ and $\operatorname{v}:C\rightarrow
B$, both belonging to $\SETWsat$. By Proposition~\ref{prop-02}, we have that

\[
\begin{tikzpicture}[xscale=1.8,yscale=-1.2]
    \node (A0_0) at (0, 0) {$C$};
    \node (A0_1) at (1, 0) {$C$};
    \node (A0_2) at (2, 0) {$B$};
    \node (A0_3) at (3, 0) {$\textrm{and}$};
    
    \path (A0_1) edge [->]node [auto,swap] {$\scriptstyle{\id_C}$} (A0_0);
    \path (A0_1) edge [->]node [auto] {$\scriptstyle{\operatorname{v}}$} (A0_2);
    
    \node (A0_4) at (4, 0) {$B$};
    \node (A0_5) at (5, 0) {$B$};
    \node (A0_6) at (6, 0) {$A$};
    
    \path (A0_5) edge [->]node [auto,swap] {$\scriptstyle{\id_B}$} (A0_4);
    \path (A0_5) edge [->]node [auto] {$\scriptstyle{\operatorname{w}}$} (A0_6);
\end{tikzpicture}
\]
are both internal equivalences in $\CATC\left[\SETWinv\right]$. So by (\hyperref[C2]{C2}) and
Lemma~\ref{lem-06}, also their composition

\[
\begin{tikzpicture}[xscale=2.4,yscale=-1.2]
    \node (A0_0) at (0, 0) {$C$};
    \node (A0_1) at (1, 0) {$C$};
    \node (A0_2) at (2, 0) {$A$};
    
    \path (A0_1) edge [->]node [auto,swap] {$\scriptstyle{\id_C\circ\id_C}$} (A0_0);
    \path (A0_1) edge [->]node [auto] {$\scriptstyle{\operatorname{w}\circ\operatorname{v}}$} (A0_2);
\end{tikzpicture}
\]
is an internal equivalence. So using Lemma~\ref{lem-04} and \eqref{eq-70},
we get that also $\functor{U}_{\SETW,1}(\operatorname{w}\circ\operatorname{v})$ is an internal
equivalence. Again by Proposition~\ref{prop-02} this implies that $\operatorname{w}
\circ\operatorname{v}$ belongs to $\SETWsat$. Therefore (BF2)
holds for $(\CATC,\SETWsat)$.\\

Let us prove (BF3),
so let us fix any morphism $\operatorname{w}:A\rightarrow B$ in
$\SETWsat$ and any morphism $f:C\rightarrow B$. Since $\operatorname{w}$ belongs to $\SETWsat$,
there are an object $A'$ and a morphism $\operatorname{v}:A'\rightarrow A$ such that
$\operatorname{w}\circ\operatorname{v}$ belongs to $\SETW$. Since (BF3)
holds for $\SETW$, there are an object $D$, a morphism $\operatorname{w}':D\rightarrow C$ in
$\SETW\subseteq\SETWsat$, a morphism $f':
D\rightarrow A'$ and an invertible $2$-morphism $\alpha:(\operatorname{w}\circ\operatorname{v})\circ
f'\Rightarrow f\circ\operatorname{w}'$. Then the data

\[D,\quad\operatorname{w}',\quad\operatorname{v}\circ f',\quad\alpha\odot\thetaa{\operatorname{w}}
{\operatorname{v}}{f'}\]
prove that (BF3) holds for $(\CATC,\SETWsat)$.\\

\emph{For simplicity of exposition, we give the proof of} (BF4)
\emph{only in the special case when $\CATC$ is a $2$-category. The proof of the general case
follows the same ideas, adding associators and unitors wherever it is necessary}.
In order to prove (BF4a),
let us fix any morphism $\operatorname{w}:B\rightarrow
A$ in $\SETWsat$, any pair of morphisms $f^1,f^2:C\rightarrow B$ and any $2$-morphism $\alpha:
\operatorname{w}\circ f^1\Rightarrow\operatorname{w}\circ f^2$. Since $\operatorname{w}$ belongs to
$\SETWsat$, then there are a pair of objects $B',B''$ and a pair of morphisms $\operatorname{w}':B'
\rightarrow B$ and $\operatorname{w}'':B''\rightarrow B'$ such that both $\operatorname{w}\circ
\operatorname{w}'$ and $\operatorname{w}'\circ\operatorname{w}''$ belong to $\SETW$. By
(BF3) for $(\CATC,\SETW)$, for each $m=1,2$ there is a set of data ($D^m,g^m,
\operatorname{u}^m,\gamma^m)$ in $\CATC$ as in the following diagram, with $\operatorname{u}^m$
in $\SETW$ and $\gamma^m$ invertible:
\[
\begin{tikzpicture}[xscale=1.8,yscale=-0.8]
    \node (A0_1) at (1, 0) {$D^m$};
    \node (A1_0) at (0, 2) {$C$};
    \node (A1_2) at (2, 2) {$B''$.};
    \node (A2_1) at (1, 2) {$B$};

    \node (A1_1) at (1, 1) {$\gamma^m$};
    \node (B1_1) at (1, 1.4) {$\Rightarrow$};
    
    \path (A1_2) edge [->]node [auto] {$\scriptstyle{\operatorname{w}'
      \circ\operatorname{w}''}$} (A2_1);
    \path (A0_1) edge [->]node [auto] {$\scriptstyle{g^m}$} (A1_2);
    \path (A1_0) edge [->]node [auto,swap] {$\scriptstyle{f^m}$} (A2_1);
    \path (A0_1) edge [->]node [auto,swap] {$\scriptstyle{\operatorname{u}^m}$} (A1_0);
\end{tikzpicture}
\]

Again by (BF3)
for $(\CATC,\SETW)$, there is a set of data $(D^3,\operatorname{z}^1,
\operatorname{z}^2,\phi)$ in $\CATC$ as follows,with $\operatorname{z}^1$ in $\SETW$ and
$\phi$ invertible

\[
\begin{tikzpicture}[xscale=1.5,yscale=-0.8]
    \node (A0_1) at (1, 0) {$D^3$};
    \node (A1_0) at (0, 2) {$D^1$};
    \node (A1_2) at (2, 2) {$D^2$.};
    \node (A2_1) at (1, 2) {$C$};

    \node (A1_1) at (1, 1) {$\phi$};
    \node (B1_1) at (1, 1.4) {$\Rightarrow$};
    
    \path (A1_2) edge [->]node [auto] {$\scriptstyle{\operatorname{u}^2}$} (A2_1);
    \path (A0_1) edge [->]node [auto] {$\scriptstyle{\operatorname{z}^2}$} (A1_2);
    \path (A1_0) edge [->]node [auto,swap] {$\scriptstyle{\operatorname{u}^1}$} (A2_1);
    \path (A0_1) edge [->]node [auto,swap] {$\scriptstyle{\operatorname{z}^1}$} (A1_0);
\end{tikzpicture}
\]

Now we define a $2$-morphism in $\CATC$ as follows:

\begin{gather}
\nonumber \overline{\alpha}:=\Big(i_{\operatorname{w}}\ast\gamma^2\ast
 i_{\operatorname{z}^2}\Big)\odot\Big(i_{\operatorname{w}\circ f^2}\ast\phi\Big)\odot\Big(\alpha
 \ast i_{\operatorname{u}^1\circ\operatorname{z}^1}\Big)\odot\Big(i_{\operatorname{w}}\ast
 \left(\gamma^1\right)^{-1}\ast i_{\operatorname{z}^1}\Big): \\
\label{eq-69} \operatorname{w}\circ\operatorname{w}'\circ\operatorname{w}''\circ
 g^1\circ\operatorname{z}^1\Longrightarrow\operatorname{w}\circ\operatorname{w}'\circ
 \operatorname{w}''\circ g^2\circ\operatorname{z}^2.
\end{gather}

Since $\operatorname{w}\circ\operatorname{w}'$ belongs to $\SETW$, then by (BF4a)
for $(\CATC,\SETW)$ there are an object $D$, a morphism $\operatorname{z}:D\rightarrow D^3$ in $\SETW$
and a $2$-morphism

\[\overline{\beta}:\,\operatorname{w}''\circ g^1\circ\operatorname{z}^1\circ\operatorname{z}
\Longrightarrow\operatorname{w}''\circ g^2\circ\operatorname{z}^2\circ\operatorname{z},\]
such that

\begin{equation}\label{eq-65}
\overline{\alpha}\ast i_{\operatorname{z}}=i_{\operatorname{w}
\circ\operatorname{w}'}\ast\overline{\beta}.
\end{equation}

Now let us set

\[\operatorname{v}:=\operatorname{u}^1\circ\operatorname{z}^1\circ\operatorname{z}:\,D
\longrightarrow C;\]
such a morphism belongs to $\SETW$ (hence also to $\SETWsat$) because of (BF2)
for $(\CATC,\SETW)$. Then it makes sense to define:

\begin{equation}\label{eq-57}
\beta:=\Big(i_{f^2}\ast\phi^{-1}\ast i_{\operatorname{z}}\Big)\odot\Big(\left(\gamma^2\right)^{-1}
\ast i_{\operatorname{z}^2\circ\operatorname{z}}\Big)\odot\Big(i_{\operatorname{w}'}\ast
\overline{\beta}\Big)\odot\Big(\gamma^1\ast i_{\operatorname{z}^1\circ\operatorname{z}}\Big):
f^1\circ\operatorname{v}\Longrightarrow f^2\circ\operatorname{v}.
\end{equation}

By replacing \eqref{eq-69} in \eqref{eq-65}, we get:

\begin{gather*}
\alpha\ast i_{\operatorname{v}}=\alpha\ast i_{\operatorname{u}^1\circ\operatorname{z}^1
 \circ\operatorname{z}}= \\
=i_{\operatorname{w}}\ast\Big(\Big(i_{f^2}\ast\phi^{-1}\ast
 i_{\operatorname{z}}\Big)\odot\Big(\left(\gamma^2\right)^{-1}\ast i_{\operatorname{z}^2\circ
 \operatorname{z}}\Big)
 \odot\Big(i_{\operatorname{w}'}\ast\overline{\beta}\Big)\odot\Big(\gamma^1\ast
  i_{\operatorname{z}^1\circ\operatorname{z}}\Big)\Big)
 \stackrel{\eqref{eq-57}}{=} i_{\operatorname{w}}\ast\beta.
\end{gather*}

So we have proved that (BF4a)
holds for $(\CATC,\SETWsat)$.\\

Moreover, if in the previous computations we assume that $\alpha$ is invertible, then so is
$\overline{\alpha}$ because $\gamma^1,\gamma^2$ and $\phi$ are invertible. Then by (BF4b)
for $(\CATC,\SETW)$ we get that also $\overline{\beta}$ is invertible, hence
also $\beta$ is invertible, so (BF4b)
is verified for $(\CATC,\SETWsat)$.\\

Now let us prove that also (BF4c)
holds for $(\CATC,\SETWsat)$. So let us suppose
that there are an object $D'$, a morphism $\operatorname{v}':D'\rightarrow C$ in $\SETWsat$ and a
$2$-morphism $\beta':f^1\circ\operatorname{v}'\Rightarrow f^2\circ\operatorname{v}'$, such that

\begin{equation}\label{eq-66}
\alpha\ast i_{\operatorname{v}'}=i_{\operatorname{w}}\ast\beta'.
\end{equation}

Following the proof of (BF3)
for $(\CATC,\SETWsat)$ above, there are
data $(\overline{D},\operatorname{t},\operatorname{t}',\mu)$ as in the following diagram, with
$\operatorname{t}$ in $\SETW$ and $\mu$ invertible:

\[
\begin{tikzpicture}[xscale=2.2,yscale=-0.8]
    \node (A0_1) at (1, 0) {$\overline{D}$};
    \node (A1_0) at (0, 2) {$D$};
    \node (A1_2) at (2, 2) {$D'$.};
    \node (A2_1) at (1, 2) {$C$};

    \node (A1_1) at (1, 1) {$\mu$};
    \node (B1_1) at (1, 1.4) {$\Rightarrow$};
    
    \path (A1_2) edge [->]node [auto] {$\scriptstyle{\operatorname{v}'}$} (A2_1);
    \path (A0_1) edge [->]node [auto] {$\scriptstyle{\operatorname{t}'}$} (A1_2);
    \path (A1_0) edge [->]node [auto,swap] {$\scriptstyle{\operatorname{v}=
      \operatorname{u}^1\circ\operatorname{z}^1\circ\operatorname{z}}$} (A2_1);
    \path (A0_1) edge [->]node [auto,swap] {$\scriptstyle{\operatorname{t}}$} (A1_0);
\end{tikzpicture}
\]

Then we define a $2$-morphism as follows

\begin{gather}
\nonumber \lambda:=\Big(\gamma^2\ast i_{\operatorname{z}^2\circ\operatorname{z}
 \circ\operatorname{t}}\Big)\odot\Big(i_{f^2}\ast\phi\ast i_{\operatorname{z}\circ
 \operatorname{t}}\Big)\odot\Big(i_{f^2}\ast\mu^{-1}\Big)\odot \\
\nonumber \odot\Big(\beta'\ast i_{\operatorname{t}'}\Big)\odot\Big(i_{f^1}\ast\mu\Big)\odot
 \Big(\left(\gamma^1\right)^{-1}\ast i_{\operatorname{z}^1\circ\operatorname{z}\circ
 \operatorname{t}}\Big): \\
\label{eq-17} \phantom{\Big(}\operatorname{w}'\circ\operatorname{w}''\circ g^1\circ
 \operatorname{z}^1\circ\operatorname{z}\circ\operatorname{t}\Longrightarrow\operatorname{w}'
 \circ\operatorname{w}''\circ g^2\circ\operatorname{z}^2\circ\operatorname{z}
 \circ\operatorname{t}.
\end{gather}

By construction, $\operatorname{w}'\circ\operatorname{w}''$ belongs to $\SETW$, so by (BF4a)
for $(\CATC,\SETW)$ there are an object $\widetilde{D}$, a morphism
$\operatorname{r}:\widetilde{D}\rightarrow\overline{D}$ in $\SETW$ and a $2$-morphism

\[\psi:\,g^1\circ\operatorname{z}^1\circ\operatorname{z}\circ\operatorname{t}
\circ\operatorname{r}\Longrightarrow g^2\circ\operatorname{z}^2\circ\operatorname{z}
\circ\operatorname{t}\circ\operatorname{r},\]
such that

\begin{equation}\label{eq-64}
\lambda\ast i_{\operatorname{r}}=i_{\operatorname{w}'\circ\operatorname{w}''}\ast\psi.
\end{equation}

Then we have

\begin{gather}
\nonumber \overline{\alpha}\ast i_{\operatorname{z}\circ\operatorname{t}\circ\operatorname{r}}
 \stackrel{\eqref{eq-69}}{=}\Big(i_{\operatorname{w}}\ast\gamma^2
 \ast i_{\operatorname{z}^2\circ\operatorname{z}\circ\operatorname{t}\circ
 \operatorname{r}}\Big)\odot\Big(i_{\operatorname{w}\circ f^2}\ast\phi\ast
 i_{\operatorname{z}\circ
 \operatorname{t}\circ\operatorname{r}}\Big)\odot \\
\nonumber\odot\Big(\alpha\ast
 i_{\operatorname{u}^1\circ\operatorname{z}^1\circ\operatorname{z}\circ\operatorname{t}\circ
 \operatorname{r}}\Big)\odot\Big(i_{\operatorname{w}}\ast\left(\gamma^1\right)^{-1}\ast
 i_{\operatorname{z}^1\circ\operatorname{z}\circ\operatorname{t}\circ\operatorname{r}}\Big)
 \stackrel{(\ast)}{=} \\
\nonumber \stackrel{(\ast)}{=}\Big(i_{\operatorname{w}}\ast
 \gamma^2\ast i_{\operatorname{z}^2\circ\operatorname{z}\circ
 \operatorname{t}\circ\operatorname{r}}\Big)\odot\Big(i_{\operatorname{w}
 \circ f^2}\ast\phi\ast i_{\operatorname{z}\circ\operatorname{t}\circ\operatorname{r}}
 \Big)\odot\Big(i_{\operatorname{w}\circ f^2}\ast\mu^{-1}\ast i_{\operatorname{r}}
 \Big)\odot \\
\nonumber \odot\Big(\alpha\ast i_{\operatorname{v}'\circ\operatorname{t}'\circ\operatorname{r}}
 \Big)\odot\Big(i_{\operatorname{w}\circ f^1}\ast\mu\ast i_{\operatorname{r}}\Big)\odot
 \Big(i_{\operatorname{w}}\ast\left(\gamma^1\right)^{-1}\ast
 i_{\operatorname{z}^1\circ\operatorname{z}\circ\operatorname{t}\circ\operatorname{r}}\Big)
 \stackrel{\eqref{eq-66}}{=} \\
\nonumber \stackrel{\eqref{eq-66}}{=}\Big(i_{\operatorname{w}}\ast
 \gamma^2\ast i_{\operatorname{z}^2\circ\operatorname{z}\circ\operatorname{t}
 \circ\operatorname{r}}\Big)\odot\Big(i_{\operatorname{w}\circ f^2}\ast
 \phi\ast i_{\operatorname{z}\circ\operatorname{t}\circ\operatorname{r}}\Big)\odot
 \Big(i_{\operatorname{w}\circ f^2}\ast\mu^{-1}\ast i_{\operatorname{r}}\Big)\odot \\
\nonumber \odot\Big(i_{\operatorname{w}}\ast\beta'\ast i_{\operatorname{t}'\circ\operatorname{r}}
 \Big)\odot\Big(i_{\operatorname{w}\circ f^1}\ast\mu\ast i_{\operatorname{r}}\Big)\odot
 \Big(i_{\operatorname{w}}\ast\left(\gamma^1\right)^{-1}\ast i_{\operatorname{z}^1\circ
 \operatorname{z}\circ\operatorname{t}\circ\operatorname{r}}\Big)\stackrel{\eqref{eq-17}}{=} \\
\label{eq-18} \stackrel{\eqref{eq-17}}{=}\Big(i_{\operatorname{w}}\ast
 \lambda\ast i_{\operatorname{r}}\Big)\stackrel{\eqref{eq-64}}{=}
 i_{\operatorname{w}\circ\operatorname{w}'\circ\operatorname{w}''}\ast\psi,
\end{gather}
where the passage denoted by $(\ast)$ is given by the interchange law. Then we define
$\operatorname{z}':=\operatorname{z}\circ\operatorname{t}\circ\operatorname{r}:\widetilde{D}
\rightarrow D^3$; such a morphism belongs to $\SETW$ by (BF2)
for $(\CATC,\SETW)$. Moreover, we set:

\begin{equation}
\label{eq-58} \overline{\beta}':=i_{\operatorname{w}''}\ast\psi:\,\,\operatorname{w}''\circ g^1
\circ\operatorname{z}^1\circ\operatorname{z}'\Longrightarrow\operatorname{w}''\circ g^2\circ
\operatorname{z}^2\circ\operatorname{z}'.
\end{equation}

Then \eqref{eq-18} reads as follows:

\begin{equation}\label{eq-60}
\overline{\alpha}\ast i_{\operatorname{z}'}=i_{\operatorname{w}
\circ\operatorname{w}'}\ast\overline{\beta}'.
\end{equation}

Now $\operatorname{w}\circ\operatorname{w}'$ belongs to $\SETW$ by construction. Therefore, we can
apply (BF4c)
for $(\CATC,\SETW)$ to the pair of identities given by \eqref{eq-65}
and \eqref{eq-60}. So there is a set of data $(E,\operatorname{p},\operatorname{s},\nu)$ as in
the following diagram

\[
\begin{tikzpicture}[xscale=1.4,yscale=-0.6]
    \node (A0_0) at (0, 1) {$E$};
    \node (A0_2) at (2, 0) {$D$};
    \node (A0_4) at (4, 1) {$D^3$,};
    \node (A2_2) at (2, 2) {$\widetilde{D}$};

    \node (A1_2) at (2, 1) {$\Downarrow\,\nu$};

    \path (A0_0) edge [->,bend left=15]node [auto,swap] {$\scriptstyle{\operatorname{p}}$} (A2_2);
    \path (A0_2) edge [->,bend right=15]node [auto] {$\scriptstyle{\operatorname{z}}$} (A0_4);
    \path (A2_2) edge [->,bend left=15]node [auto,swap] {$\scriptstyle{\operatorname{z}'=
      \operatorname{z}\circ\operatorname{t}\circ\operatorname{r}}$} (A0_4);
    \path (A0_0) edge [->,bend right=15]node [auto] {$\scriptstyle{\operatorname{s}}$} (A0_2);
\end{tikzpicture}
\]
such that $\operatorname{z}\circ\operatorname{s}$ belongs to $\SETW$, $\nu$ is invertible and

\begin{equation}
\label{eq-59} \Big(\overline{\beta}'\ast i_{\operatorname{p}}\Big)\odot\Big(i_{\operatorname{w}''
\circ g^1\circ\operatorname{z}^1}\ast\nu\Big)=\Big(i_{\operatorname{w}''\circ g^2\circ
\operatorname{z}^2}\ast\nu\Big)\odot\Big(\overline{\beta}\ast i_{\operatorname{s}}\Big).
\end{equation}

Now we have:

\begin{gather}
\nonumber i_{\operatorname{w}'}\ast\overline{\beta}'\ast i_{\operatorname{p}}
 \stackrel{\eqref{eq-58}}{=}i_{\operatorname{w}'\circ
 \operatorname{w}''}\ast\psi\ast i_{\operatorname{p}}\stackrel{\eqref{eq-64}}{=}\lambda\ast
 i_{\operatorname{r}\circ\operatorname{p}}\stackrel{\eqref{eq-17}}{=} \\
\nonumber \stackrel{\eqref{eq-17}}{=}\Big(\gamma^2\ast
 i_{\operatorname{z}^2\circ\operatorname{z}\circ\operatorname{t}\circ\operatorname{r}\circ
 \operatorname{p}}\Big)\odot\Big(i_{f^2}\ast\phi\ast
 i_{\operatorname{z}\circ\operatorname{t}\circ\operatorname{r}\circ\operatorname{p}}\Big)\odot
 \Big(i_{f^2}\ast\mu^{-1}\ast i_{\operatorname{r}\circ\operatorname{p}}\Big)\odot \\
\label{eq-49} \odot\Big(\beta'\ast i_{\operatorname{t}'\circ
 \operatorname{r}\circ\operatorname{p}}\Big)\odot\Big(i_{f^1}\ast\mu\ast i_{\operatorname{r}
 \circ\operatorname{p}}
 \Big)\odot\Big(\left(\gamma^1\right)^{-1}\ast i_{\operatorname{z}^1\circ\operatorname{z}\circ
 \operatorname{t}\circ\operatorname{r}\circ\operatorname{p}}\Big).
\end{gather}

Moreover, by \eqref{eq-57} we have:

\begin{equation}\label{eq-54}
i_{\operatorname{w}'}\ast\overline{\beta}\ast i_{\operatorname{s}}=\Big(\gamma^2
\ast i_{\operatorname{z}^2\circ
\operatorname{z}\circ\operatorname{s}}\Big)\odot\Big(i_{f^2}\ast
\phi\ast i_{\operatorname{z}\circ\operatorname{s}}\Big)\odot\Big(\beta\ast
i_{\operatorname{s}}\Big)\odot\Big(\left(\gamma^1\right)^{-1}\ast i_{\operatorname{z}^1
\circ\operatorname{z}\circ\operatorname{s}}\Big).
\end{equation}

Then:

\begin{gather}
\nonumber \Big(\gamma^2\ast i_{\operatorname{z}^2\circ
 \operatorname{z}\circ\operatorname{t}\circ\operatorname{r}\circ\operatorname{p}}\Big)\odot
 \Big(i_{f^2}\ast\phi\ast i_{\operatorname{z}\circ\operatorname{t}
 \circ\operatorname{r}\circ\operatorname{p}}\Big)\odot\Big(i_{f^2}\ast\mu^{-1}\ast i_{\operatorname{r}
 \circ\operatorname{p}}\Big)\odot \\
\nonumber \odot\Big(\beta'\ast i_{\operatorname{t}'\circ\operatorname{r}\circ
 \operatorname{p}}\Big)\odot\Big(i_{f^1}\ast\mu\ast
 i_{\operatorname{r}\circ\operatorname{p}}\Big)\odot\Big(i_{f^1\circ\operatorname{u}^1\circ
 \operatorname{z}^1}\ast\nu\Big)\odot\Big(\left(\gamma^1\right)^{-1}\ast i_{\operatorname{z}^1\circ
 \operatorname{z}\circ\operatorname{s}}\Big)\stackrel{(\ast)}{=} \\
\nonumber \stackrel{(\ast)}{=}\Big(\gamma^2\ast
 i_{\operatorname{z}^2\circ\operatorname{z}\circ\operatorname{t}\circ\operatorname{r}\circ
 \operatorname{p}}\Big)\odot\Big(i_{f^2}\ast\phi\ast
 i_{\operatorname{z}\circ\operatorname{t}\circ\operatorname{r}\circ\operatorname{p}}\Big)\odot \\
\nonumber \odot\Big(i_{f^2}\ast\mu^{-1}\ast i_{\operatorname{r}\circ\operatorname{p}}\Big)\odot
 \Big(\beta'\ast i_{\operatorname{t}'\circ\operatorname{r}\circ
 \operatorname{p}}\Big)\odot\Big(i_{f^1}\ast\mu\ast
 i_{\operatorname{r}\circ\operatorname{p}}\Big)\odot \\
\nonumber \odot\Big(\left(\gamma^1\right)^{-1}\ast i_{\operatorname{z}^1\circ
 \operatorname{z}\circ\operatorname{t}\circ\operatorname{r}\circ\operatorname{p}}\Big)\odot
 \Big(i_{\operatorname{w}'\circ\operatorname{w}''\circ g^1\circ
 \operatorname{z}^1}\ast\nu\Big)\stackrel{\eqref{eq-49}}{=} \\
\nonumber \stackrel{\eqref{eq-49}}{=}\Big(i_{\operatorname{w}'}\ast
 \overline{\beta}'\ast i_{\operatorname{p}}\Big)\odot\Big(
 i_{\operatorname{w}'\circ\operatorname{w}''\circ g^1\circ\operatorname{z}^1}\ast\nu\Big)
 \stackrel{\eqref{eq-59}}{=} \\
\nonumber \stackrel{\eqref{eq-59}}{=}\Big(i_{\operatorname{w}'\circ
 \operatorname{w}''\circ g^2\circ\operatorname{z}^2}\ast\nu\Big)
 \odot\Big(i_{\operatorname{w}'}\ast\overline{\beta}\ast i_{\operatorname{s}}\Big)
 \stackrel{\eqref{eq-54}}{=} \\
\nonumber \stackrel{\eqref{eq-54}}{=}\Big(i_{\operatorname{w}'\circ
 \operatorname{w}''\circ g^2\circ\operatorname{z}^2}\ast\nu\Big)\odot\Big(
 \gamma^2\ast i_{\operatorname{z}^2\circ\operatorname{z}\circ\operatorname{s}}
 \Big)\odot \\
\nonumber \odot\Big(i_{f^2}\ast\phi\ast i_{\operatorname{z}\circ\operatorname{s}}\Big)
 \odot\Big(\beta\ast i_{\operatorname{s}}\Big)\odot\Big(\left(\gamma^1\right)^{-1}\ast
 i_{\operatorname{z}^1\circ\operatorname{z}\circ\operatorname{s}}\Big)
 \stackrel{(\ast)}{=} \\
\nonumber \stackrel{(\ast)}{=}\Big(\gamma^2\ast
 i_{\operatorname{z}^2\circ\operatorname{z}\circ\operatorname{t}\circ\operatorname{r}\circ
 \operatorname{p}}\Big)\odot\Big(i_{f^2\circ\operatorname{u}^2\circ
 \operatorname{z}^2}\ast\nu\Big)\odot \\
\nonumber \odot\Big(i_{f^2}\ast\phi\ast i_{\operatorname{z}\circ\operatorname{s}}\Big)
 \odot\Big(\beta\ast i_{\operatorname{s}}\Big)\odot\Big(\left(\gamma^1\right)^{-1}\ast
 i_{\operatorname{z}^1\circ\operatorname{z}\circ\operatorname{s}}\Big)
 \stackrel{(\ast)}{=} \\
\nonumber \stackrel{(\ast)}{=}\Big(\gamma^2\ast
 i_{\operatorname{z}^2\circ\operatorname{z}\circ\operatorname{t}\circ\operatorname{r}\circ
 \operatorname{p}}\Big)\odot\Big(i_{f^2}\ast\phi\ast
 i_{\operatorname{z}\circ\operatorname{t}\circ\operatorname{r}\circ\operatorname{p}}\Big)\odot \\
\label{eq-48} \odot\Big(i_{f^2\circ\operatorname{u}^1\circ
 \operatorname{z}^1}\ast\nu\Big)\odot\Big(\beta\ast i_{\operatorname{s}}\Big)
 \odot\Big(\left(\gamma^1\right)^{-1}\ast i_{\operatorname{z}^1\circ\operatorname{z}\circ
 \operatorname{s}}\Big),
\end{gather}
where all the identities denoted by $(\ast)$ are given by the interchange law in $\CATC$. Since
$\gamma^1,\gamma^2$ and $\phi$ are invertible, then identity \eqref{eq-48} implies that 

\begin{gather*}
\Big(i_{f^2}\ast\mu^{-1}\ast i_{\operatorname{r}\circ
 \operatorname{p}}\Big)\odot\Big(\beta'\ast i_{\operatorname{t}'\circ
 \operatorname{r}\circ\operatorname{p}}\Big)\odot\Big(i_{f^1}\ast\mu\ast
 i_{\operatorname{r}\circ\operatorname{p}}\Big)\odot
 \Big(i_{f^1\circ\operatorname{u}^1\circ\operatorname{z}^1}\ast\nu\Big)= \\
=\Big(i_{f^2\circ\operatorname{u}^1\circ\operatorname{z}^1}\ast
 \nu\Big)\odot\Big(\beta\ast i_{\operatorname{s}}\Big).
\end{gather*}

Therefore, we have the following identity:

\begin{gather}
\nonumber \Big(\beta'\ast i_{\operatorname{t}'\circ\operatorname{r}\circ
 \operatorname{p}}\Big)\odot\Big\{i_{f^1}\ast\Big[\Big(\mu\ast i_{\operatorname{r}\circ
 \operatorname{p}}\Big)\odot\Big(i_{\operatorname{u}^1\circ\operatorname{z}^1}
 \ast\nu\Big)\Big]\Big\}= \\
\label{eq-40} =\Big\{i_{f^2}\ast\Big[\Big(\mu\ast i_{\operatorname{r}\circ
 \operatorname{p}}\Big)\odot\Big(i_{\operatorname{u}^1\circ\operatorname{z}^1}
 \ast\nu\Big)\Big]\Big\}\odot\Big(\beta\ast i_{\operatorname{s}}\Big).
\end{gather}

Now we define a morphism $\operatorname{s}':=\operatorname{t}'\circ\operatorname{r}\circ
\operatorname{p}:\,E\rightarrow D'$ and an invertible $2$-morphism

\begin{gather*}
\zeta:=\Big(\mu\ast i_{\operatorname{r}\circ\operatorname{p}}\Big)\odot\Big(
 i_{\operatorname{u}^1\circ\operatorname{z}^1}\ast\nu\Big): \\
\phantom{\Big(}\operatorname{v}\circ\operatorname{s}=\operatorname{u}^1\circ\operatorname{z}^1
 \circ\operatorname{z}\circ\operatorname{s}\Longrightarrow\operatorname{v}'\circ\operatorname{t}'
 \circ\operatorname{r}\circ\operatorname{p}=\operatorname{v}'\circ\operatorname{s}'.
\end{gather*}

Then \eqref{eq-40} reads as follows:

\[\Big(\beta'\ast i_{\operatorname{s}'}\Big)\odot\Big(i_{f^1}\ast\zeta\Big)=
\Big(i_{f^2}\ast\zeta\Big)\odot\Big(\beta\ast i_{\operatorname{s}}\Big).\]

By construction, $\operatorname{u}^1$, $\operatorname{z}^1$ and $\operatorname{z}\circ
\operatorname{s}$ belong to $\SETWsat$, so by the already proved condition (BF2)
for $(\CATC,\SETWsat)$ we have that also $\operatorname{v}\circ\operatorname{s}=
\operatorname{u}^1\circ\operatorname{z}^1\circ\operatorname{z}\circ\operatorname{s}$ belongs to
$\SETWsat$; this proves that (BF4c)
holds for $(\CATC,\SETWsat)$.\\

Lastly, let us fix any pair of morphisms $\operatorname{w},\operatorname{v}:B\rightarrow A$, any
invertible $2$-morphism $\alpha:\operatorname{v}\Rightarrow\operatorname{w}$ and let us suppose that
$\operatorname{w}$ belongs to $\SETWsat$. Then there are a pair of objects $C,D$ and a pair of
morphisms $\operatorname{w}':C\rightarrow B$ and $\operatorname{w}'':D\rightarrow C$, such that
both $\operatorname{w}\circ\operatorname{w}'$ and $\operatorname{w}'\circ\operatorname{w}''$ belong
to $\SETW$. By (BF5) for $(\CATC,\SETW)$ applied to $\alpha\ast i_{\operatorname{w}'}$, we get that
$\operatorname{v}\circ\operatorname{w}'$ belongs to $\SETW$. Therefore, $\operatorname{v}$ belongs
to $\SETWsat$, i.e.\ (BF5) holds for $(\CATC,\SETWsat)$.
\end{proof}


\end{document}